\documentclass{amsart}
\usepackage{amssymb,amsthm,amsmath}
\usepackage{enumerate}
\input{xy}\xyoption{all}
\usepackage{graphicx}
\usepackage[breaklinks,colorlinks=true]{hyperref}
\newcommand{\apref}[3]{\hyperref[#2]{#1\ref*{#2}#3}} 

\newcommand{\textbmat}[4]{\left[\begin{smallmatrix} #1&#2 \\ #3&#4
\end{smallmatrix}\right]}
\newcommand{\bmat}[4]{\begin{bmatrix} #1&#2\\#3&#4\end{bmatrix}}
\newcommand{\h}{\mathbb{H}}

\DeclareMathOperator{\Ext}{ext}
\DeclareMathOperator{\Int}{int}
\DeclareMathOperator{\height}{ht}
\DeclareMathOperator{\cl}{cl}
\DeclareMathOperator{\Cod}{Cod}
\DeclareMathOperator{\NC}{NC}
\DeclareMathOperator{\D}{D}

\DeclareMathOperator{\vb}{vb}
\DeclareMathOperator{\vc}{vc}
\DeclareMathOperator{\bd}{bd}
\DeclareMathOperator{\BS}{BS}
\DeclareMathOperator{\CS}{CS}
\DeclareMathOperator{\DS}{DS}

\DeclareMathOperator{\NIC}{NIC}
\DeclareMathOperator{\Sides}{Sides}
\DeclareMathOperator{\Seq}{Seq}
\DeclareMathOperator{\cyl}{cyl}

\newcommand{\pre}{\text{pre}}
\newcommand{\hg}{\overline \h^g}
\newcommand{\chg}{\cl_{\overline \h^g}}
\newcommand{\bhg}{\partial_g}

\newcommand{\rueck}{\hspace{-.9mm}}
\newcommand{\fd}{\mc F}
\newcommand{\pch}{\mc A}
\newcommand{\fpch}{\mathbb A}
\newcommand{\ch}{\mc B}
\newcommand{\fch}{\mathbb B}
\newcommand{\choices}{\mathbb S}
\newcommand{\shmap}{\mathbb T}

\newcommand{\rd}{\text{red}}
\newcommand{\st}{\text{st}}
\newcommand{\all}{\text{all}}
\newcommand{\bk}{\text{bk}}

\newcommand\N{\mathbb{N}}
\newcommand\Q{\mathbb{Q}}
\newcommand\R{\mathbb{R}}
\newcommand\Z{\mathbb{Z}}
\newcommand\C{\mathbb{C}}
\newcommand{\mc}[1]{\mathcal #1}

\newcommand{\wt}{\widetilde}
\newcommand{\wh}{\widehat}

\DeclareMathOperator{\PSL}{PSL}
\DeclareMathOperator{\PGL}{PGL}
\newcommand{\sceq}{\mathrel{\mathop:}=}
\newcommand{\seqc}{\mathrel{=\mkern-4.5mu{\mathop:}}}
\DeclareMathOperator{\Ima}{Im}
\DeclareMathOperator{\Rea}{Re}
\newcommand{\mat}[4]{\begin{pmatrix} #1&#2\\#3&#4\end{pmatrix}}
\newcommand{\textmat}[4]{\left(\begin{smallmatrix} #1&#2 \\ #3&#4
\end{smallmatrix}\right)}
\DeclareMathOperator{\id}{id}

\newcommand\mminus{\!\smallsetminus\!}
\newcommand\ie{\mbox{i.\,e., }}
\newcommand\eg{\mbox{e.\,g., }}
\newcommand\wrt{\mbox{w.\,r.\,t.\@ }}

\newcommand{\eps}{\varepsilon}
\DeclareMathOperator{\pr}{pr}

\DeclareMathOperator{\PGamma}{P\Gamma}
\DeclareMathOperator{\Fct}{Fct}

\theoremstyle{plain}
\newtheorem{prop}{Proposition}[section]
\newtheorem{lemma}[prop]{Lemma}

\newtheorem{thm}[prop]{Theorem}

\newtheorem{cor}[prop]{Corollary}

\theoremstyle{definition}

\newtheorem{defi}[prop]{Definition}
\newtheorem{defirem}[prop]{Definition and Remark}

\newtheorem{example}[prop]{Example}

\newtheorem{construction}[prop]{Construction}
\newtheorem{constrdefi}[prop]{Construction and Definition}

\theoremstyle{remark}
\newtheorem{remark}[prop]{Remark}

\begin{document}

\title[Symbolic dynamics]{Symbolic dynamics for the geodesic flow on two-dimensional hyperbolic good orbifolds}
\author{Anke D.\@ Pohl}
\address{Mathematisches Institut, Georg-August-Universit\"at G\"ottingen,  Bunsenstr. 3-5, 37073 G\"ottingen}
\email{pohl@uni-math.gwdg.de}
\thanks{The author acknowledges support from the International Research Training Group 1133 ``Geometry and Analysis of Symmetries'' and the Volkswagen Foundation}
\date{}
\begin{abstract}
We construct cross sections for the geodesic flow on the orbifolds $\Gamma\backslash\h$ which are tailor-made for the requirements of transfer operator approaches to Maass cusp forms and Selberg zeta functions. Here, $\h$ denotes the hyperbolic plane and $\Gamma$ is a nonuniform geometrically finite Fuchsian group (not necessarily a lattice, not necessarily arithmetic) which satisfies an additional condition of geometric nature. The construction of the cross sections is uniform, geometric, explicit and algorithmic.
\end{abstract}

\maketitle



\section{Introduction}

The construction and use of symbolic dynamics in various setups has a long history. It goes back to work of Hadamard \cite{Hadamard} in 1898. Since then, symbolic dynamics found influence in several fields. One of these are transfer operator approaches to Maass cusp forms (and other modular forms and functions) and Selberg zeta functions. These transfer operator approaches have their origin in thermodynamic formalism. They provide a link between the classical and the quantum dynamical systems of the considered orbifolds. In this article, we construct symbolic dynamics which are tailor-made for these transfer operator approaches. In the following we briefly present the initial example of such a transfer operator approach and recall the state of the field prior to the existence of the symbolic dynamics constructed here. Then we state the most important properties of the symbolic dynamics provided in this article and we list the progress made using these for transfer operator approaches.

\subsubsection*{Transfer operator approaches prior to the symbolic dynamics constructed here}
The modular group $\PSL_2(\Z)$ was the first Fuchsian lattice for which the complete transfer operator approach to both Maass cusp forms and the Selberg zeta function could be established. This approach is a  combination of ground-breaking work by Artin, Series, Mayer, Chang and Mayer, and Lewis and Zagier. In \cite{Series}, Series, based on work by Artin \cite{Artin}, provided a symbolic dynamics for the geodesic flow on the modular surface $\PSL_2(\Z)\backslash \h$ which relates this flow to the Gauss map
\[
 F\colon [0,1]\setminus\Q  \to  [0,1]\setminus\Q, \quad x  \mapsto \tfrac{1}{x} - \left\lfloor\tfrac{1}{x}\right\rfloor
\]
in a way that periodic geodesics on $\PSL_2(\Z)\backslash \h$ correspond to finite orbits of $F$. Here, $\h$ denotes the hyperbolic plane. Mayer \cite{Mayer_zeta, Mayer_thermo, Mayer_thermoPSL} investigated the transfer operator (weighted evolution operator) with parameter $\beta\in\C$ associated to $F$, which here takes the form  
\[
 \big(\mc L_{F,\beta} f\big)(x) = \sum_{n\in\N} (x+n)^{-2\beta} f\left(\frac{1}{x+n}\right).
\]
He found a Banach space $\mc B$ such that for $\Rea \beta > \tfrac12$, the transfer operator $\mc L_{F,\beta}$ acts on $\mc B$, is nuclear of order $0$, and has a meromorphic extension $\wt{\mc L}_{F,\beta}$ to the whole complex $\beta$-plane with values in nuclear operators of order $0$ on $\mc B$. He showed that the Selberg zeta function $Z$ is represented by the product of the Fredholm determinants of $\pm\wt{\mc L}_{F,\beta}$:
\begin{equation}\label{factorMayer}
  Z(\beta) = \det(1-\wt{\mc L}_{F,\beta})\cdot \det(1+ \wt{\mc L}_{F,\beta}).
\end{equation}
On the other side, Lewis and Zagier \cite{Lewis_Zagier} (see Lewis \cite{Lewis} for even Maass cusp forms) proved that Maass cusp forms for $\PSL_2(\Z)$ with eigenvalue $\beta(1-\beta)$ are linearly isomorphic to real-analytic functions $f$ on $\R_{>0}$ which satisfy the functional equation
\begin{equation}\label{LZ}
 f(x) = f(x+1) + (x+1)^{-2\beta} f\left(\frac{x}{x+1}\right)
\end{equation}
and are of strong decay at $0$ and $\infty$. Functions of this kind are called period functions for $\PSL_2(\Z)$. Using this characterization of Maass cusp forms, Chang and Mayer \cite{Chang_Mayer_transop} as well as Lewis and Zagier \cite{Lewis_Zagier} deduced that even resp.\@ odd Maass cusp forms for $\PSL_2(\Z)$ are linearly isomorphic to the $\pm 1$-eigenspaces of Mayer's transfer operator. Efrat \cite{Efrat_spectral} proved this earlier on a spectral level, that is,  the parameters $\beta$ for which there exist a $\pm 1$-eigenfunction of $\wt{\mc L}_{F,\beta}$ are just the spectral parameters of even resp.\@ odd Maass cusp forms. In addition, Bruggeman \cite{Bruggeman} provided a hyperfunction approach to the period functions for $\PSL_2(\Z)$.

By using representations and a hyperfunction approach, Deitmar and Hilgert \cite{Deitmar_Hilgert} could induce the period functions for $\PSL_2(\Z)$ to finite index subgroups and show the correspondence between the period functions for these subgroups and Maass cusp forms. Also using representations, Chang and Mayer \cite{Chang_Mayer_eigen, Chang_Mayer_extension} extended the transfer operator of the modular group to its finite index subgroups and represented the Selberg zeta function as the Fredholm determinant of the induced transfer operator family. For Hecke congruence subgroups, the combination of \cite{Hilgert_Mayer_Movasati} and \cite{Fraczek_Mayer_Muehlenbruch} shows a close relation between eigenfunctions of certain transfer operators and these period functions.

Moreover, the following transfer operator approaches to the Selberg zeta function have been established:
\begin{itemize}
\item Pollicott \cite{Pollicott} considered the transfer operator family associated to the Series symbolic dynamics for cocompact lattices. Here the coding sequences are not unique, so the Fredholm determinant of the arising transfer operator family is not exactly the Selberg zeta function. Instead one has a representation of the form 
\[
 Z(s) = h(s) \det(1-\mc L_s),
\]
where $h$ is a function compensating for multiple codings.
\item Fried \cite{Fried_triangle} developed a symbolic dynamics for the billiard flow on $\Lambda\backslash \h$ for triangle groups $\Lambda \subset \PGL_2(\R)$, and induced it to its finite index subgroups. The Fredholm determinant of the associated transfer operator family represents the Selberg zeta function without compensation factor.
\item Morita \cite{Morita_transfer} used a modified Bowen-Series symbolic dynamics for a wide class of non-cocompact lattices and investigated the associated transfer operator families. Here the problem with multiple codings arises again.
\item Mayer \textit{et.\@ al.\@} \cite{Mayer_Stroemberg, Mayer_Muehlenbruch_Stroemberg} used a symbolic dynamics related to Rosen continued fractions for Hecke triangle groups $\Gamma$. Also this has multiple codings.
\end{itemize}

Direct transfer operator approaches to Maass cusp forms were not known. However, related to this field of investigations is a characterization of Maass cusp forms in parabolic $1$-cohomology for any Fuchsian lattice provided by Bruggeman, Lewis and Zagier \cite{BLZ_part2} (see \cite{Bunke_Olbrich_1, Bunke_Olbrich_2,Deitmar_Hilgert_cohom} for earlier relations between cohomology spaces and Maass cusp forms or automorphic forms). 

\subsubsection*{Main properties of the symbolic dynamics constructed here}
The symbolic dynamics for the geodesic flow on the orbifolds $\Gamma\backslash\h$ provided in this article enjoy several properties which make them well-adapted for transfer operator approaches. The most important ones are the following:
\begin{itemize}
\item The symbolic dynamics arise from cross sections for the geodesic flow. The labels (the alphabet of the symbolic dynamics) arise in a natural geometric way. They encode the location of the next intersection of the cross section and the geodesic under consideration.
\item The alphabet used in the coding is finite, all labels are elements of $\Gamma$.
\item The symbolic dynamics is conjugate to a discrete dynamical system $F$ on subsets of $\R$. The map $F$ is piecewise real-analytic. The action on the pieces (intervals) are given by M\"obius transformations with the labels. The boundary points of the intervals are cuspidal.
\item In turn, the associated transfer operators have only finitely many terms. They are finite sums of the action of principal series representation by elements formed from the labels in an algorithmic way.
\item All periodic geodesics are coded, and their coding sequences are unique. Thus, compensation factors for multiple codings do not occur.
\item The construction allows a number of choices which result in different symbolic dynamics models and transfer operator families. To some degree, this freedom can be used to control properties of the transfer operators. 
\item The first step in the construction consists of the construction of a fundamental domain of a specific type. Then the construction is a finite number of algorithmic steps. This enables constructive proofs. Moreover, it allows us to test conjectures with the help of a computer.
\end{itemize}

\subsubsection*{New transfer operator approaches using the symbolic dynamics constructed here}
The transfer operator families which arise from the symbolic dynamics constructed in this article can be used for direct and constructive transfer operator approaches to Maass cusp forms. For Hecke triangle groups this is shown in \cite{Moeller_Pohl}, including separate period functions for odd resp.\@ even Maass cusp forms. In particular, the functional equation \eqref{LZ} is just the defining equation for $1$-eigenfunctions of the transfer operators for $\PSL_2(\Z)$. For the Hecke congruence subgroups $\Gamma_0(p)$, $p$ prime, this transfer operator approach to Maass cusp forms is performed in \cite{Pohl_mcf_Gamma0p}. Finally, in \cite{Pohl_mcf_general} it is shown for all Fuchsian lattices which are admissible here, including a discussion of the relation between period functions arising from different choices in the construction, as well as a brief comparison to period functions arising from the other methods mentioned above. 

In addition, transfer operator approaches to Selberg zeta functions are possible. Using a certain acceleration procedure of the symbolic dynamics, in \cite{Moeller_Pohl}, the Selberg zeta function for Hecke triangle groups is represented as the Fredholm determinant of the arising transfer operator family. Compatibility with a specific orientation-reversing Riemannian isometry on $\h$ allows a factorization of the Fredholm determinant as in \eqref{factorMayer}. For the modular group, these results reproduce Mayer's transfer operator.

In \cite{Pohl_spectral_hecke} it is proven, by extending the symbolic dynamics to one of a billiard flow, that this factorization for general Hecke triangle groups corresponds to the splitting into odd and even spectrum and that both Fredholm determinants are Fredholm determinants of transfer operator families belonging to the billiard flow. This extends Efrat's results \cite{Efrat_spectral} to Hecke triangle groups. Further results on transfer operator approaches to Selberg zeta functions will appear in future work.

\subsubsection*{Outline of this article} 
In Section~\ref{sec_prelims} below we provide the necessary preliminaries on the geometry of the hyperbolic plane and hyperbolic orbifolds as well as on fundamental domains and symbolic dynamics. In Section~\ref{sec_sketch} below we present a sketch the construction of the symbolic dynamics, discrete dynamical system and transfer operator families. Then Sections~\ref{sec_preH}-\ref{sec_arithm} below contain a detailed proof of the construction, and in Section~\ref{sec_transop} below we briefly discuss the structure of the arising transfer operators.

\section{Preliminaries}\label{sec_prelims}

\subsection{Two-dimensional hyperbolic good orbifolds}
We use the upper half plane
\[
\h \sceq \{ z\in \C \mid \Ima z > 0\}
\]
with the Riemannian metric given by the line element $ds^2 = y^{-2}(dx^2+dy^2)$ as model for the two-dimensional real hyperbolic space. The associated Riemannian metric will be denoted by $d_\h$. We identify the group of orientation-preserving Riemannian isometries with $\PSL_2(\R)$ via the well-known left action
\[
\PSL_2(\R) \times \h  \to  \h,\quad \left( g= \bmat{a}{b}{c}{d}, z \right)  \mapsto  g.z=\frac{az+b}{cz+d}.
\]
Let $\Gamma$ be a Fuchsian group, that is, a discrete subgroup of $\PSL_2(\R)$. The orbit space
\[
 Y\sceq \Gamma\backslash \h
\]
is naturally endowed with the structure of a good Riemannian orbifold. We call $Y$ a \textit{two-dimensional hyperbolic good orbifold}. The orbifold $Y$ inherits all geometric properties of $\h$ that are $\Gamma$-invariant. Vice versa, several geometric entities of $Y$ can be understood as the $\Gamma$-equivalence class of the corresponding geometric entity on $\h$. In particular,  the unit tangent bundle $SY$ of $Y$ can be identified with the orbit space of the induced $\Gamma$-action on the unit tangent bundle $S\h$ of $\h$.

We consider all geodesics to be parametrized by arc length. For $v\in S\h$ let $\gamma_v$ denote the (unit speed) geodesic on $\h$ determined by $\gamma'_v(0)=v$. The (unit speed) geodesic flow on $\h$ will be denoted by $\Phi$, hence 
\[
\Phi \colon \R\times S\h  \to  S\h, \quad (t,v)  \mapsto  \gamma'_v(t).
\]
Let $\pi\colon \h\to Y$ and $\pi\colon S\h\to SY$ denote the canonical projection maps (there will be no danger in using the same notation for both). Then the geodesic flow on $Y$ is given by 
\[
 \wh\Phi \sceq \pi \circ \Phi \circ (\id\times \pi^{-1}) \colon \R\times SY\to SY.
\]
Here, $\pi^{-1}$ is an arbitrary section of $\pi$. One easily sees that $\wh \Phi$ does not depend on the choice of $\pi^{-1}$.
Throughout, we use the convention that if $e$ denotes an element belonging in some sense to $\h$ (like geodesics on $\h$, subsets of $\h$), then $\wh e$ denotes the corresponding element belonging to $Y$.

The one-point compactification of the closure of $\h$ in $\C$ will be denoted by $\overline \h^g$, hence
\[
\overline \h^g = \{ z\in \C \mid \Ima z \geq 0\} \cup \{\infty\}.
\]
It is homeomorphic to the geodesic compactification of $\h$. The action of $\PSL_2(\R)$ extends continuously to the boundary $\partial_g \h = \R\cup \{\infty\}$ of $\h$ in $\overline \h^g$.

We let $\overline \R \sceq \R\cup \{ \pm\infty\}$ denote the two-point compactification of $\R$ and extend the ordering of $\R$ to  $\overline \R$ by the definition $-\infty < a < \infty$ for each $a\in\R$.

Let $I$ be an interval in $\R$. A curve $\alpha \colon I\to H$ is called a \textit{geodesic arc} if $\alpha$ can be extended to a geodesic. The image of a geodesic arc is called a \textit{geodesic segment}. If $\alpha$ is a geodesic, then $\alpha(\R)$ is called a \textit{complete geodesic segment}. A geodesic segment is called \textit{non-trivial} if it contains more than one element. The geodesic segments of the geodesics on $\h$ are the semicircles centered on the real line and the vertical lines. 

If $\alpha\colon I\to \h$ is a geodesic arc and $a<b$ are the boundary points of $I$ in $\overline \R$, then the points 
\[
\alpha(a) \sceq \lim_{t\to a}\alpha(t) \in \hg 
\qquad\text{and}\qquad
\alpha(b) \sceq \lim_{t\to b} \alpha(t) \in \hg
\]
are called the \textit{endpoints}%
\index{endpoints}\index{geodesic arc! endpoints}\index{geodesic segment! endpoints}
of $\alpha$ and of the associated geodesic segment $\alpha(I)$. 
The geodesic segment $\alpha(I)$ is often denoted as\label{def_geodsegm}
\[
\alpha(I) = 
\begin{cases}
[\alpha(a),\alpha(b)] & \text{if $a,b\in I$,}
\\
[\alpha(a), \alpha(b)) & \text{if $a\in I$, $b\notin I$,}
\\
(\alpha(a), \alpha(b)] & \text{if $a\notin I$, $b\in I$,}
\\
(\alpha(a), \alpha(b)) & \text{if $a,b\notin I$.}
\end{cases}
\]
If $\alpha(a),\alpha(b) \in \bhg \h$, it will always be made clear whether we refer to a geodesic segment or an interval in $\R$.

Let $U$ be a subset of $\h$. The closure of $U$ in $\h$ is denoted by $\overline U$ or $\cl(U)$, its boundary is denoted by $\partial U$, and its interior is denoted by $U^\circ$. To increase clarity, we denote the closure of a subset $V$ of $\hg$ in $\hg$  by $\overline V^g$ or $\chg V$. Moreover, we set $\bhg V \sceq \overline V^g \cap \bhg \h$. For a subset $W\subseteq \R$ let $\Int_\R(W)$ denote the interior of $W$ in $\R$ and $\partial_\R W$ the boundary of $W$ in $\R$. If $X$ is a subset of $\bhg \h$, then $\Int_g (X)$ denotes the interior of $X$ in $\bhg \h$. If $X\subseteq\R$, then $\Int_g (X) = \Int_\R (X)$.

For two sets $A,B$, the complement of $B$ in $A$ is denoted by $A\mminus B$. 
In contrast, if $\Gamma$ acts on $A$, the space of left cosets is written as $\Gamma\backslash A$.

\subsection{Fundamental domains}
Let $\Gamma$ be a Fuchsian group.  A subset $\mc F$ of $\h$ is a \textit{fundamental region} in $\h$ for $\Gamma$  if \label{fund_defi_gen}
\begin{enumerate}[(F1)]
\item\label{F1} the set $\mc F$ is open in $\h$,
\item\label{F2} the members of the family $\{g\mc F\mid g\in\Gamma\}$ are pairwise disjoint, and
\item\label{F3} $\h=\bigcup \{ g\overline{\mc F} \mid g\in\Gamma\}$.
\end{enumerate}
If, in addition, $\mc F$ is connected, then it is a \textit{fundamental domain} for $\Gamma$ in $\h$. The Fuchsian group $\Gamma$ is called \textit{geometrically finite} if there exists a convex fundamental region for $\Gamma$ in $\h$ with finitely many sides.
In this article we will use isometric fundamental regions (Ford fundamental regions), whose construction and existence we recall in the following.

Let
\[
 \Gamma_\infty \sceq \{g\in\Gamma\mid g.\infty=\infty\}
\]
denote the stabilizer subgroup of $\infty$ in $\Gamma$. Let $g=\textbmat{a}{b}{c}{d}\in\Gamma\setminus\Gamma_\infty$, which implies that $c\not=0$. Let $|\cdot|$ denote the Euclidean norm on $\C$. The \textit{isometric sphere} of $g$ is defined as
\begin{align*}
 I(g) &\sceq \{z\in H \mid |g'(z)| = 1 \} = \{ z\in H\mid |cz+d|= 1\}.
\end{align*}
The \textit{exterior} of $I(g)$ is the set
\[
 \Ext I(g) \sceq \{ z\in H\mid |cz+d|>1\},
\]
and
\[
 \Int I(g) \sceq \{ z\in H \mid |cz+d|<1\}
\]
is its \textit{interior}. If the representative of $g$ is chosen such that $c>0$, then the isometric sphere $I(g)$ is the complete geodesic segment with endpoints $-\tfrac{d}{c} - \tfrac{1}{c}$ and $-\tfrac{d}{c}+\tfrac{1}{c}$.  If $z_0=x_0+iy_0$ is an element of $I(g)$, then the geodesic segment $(z_0,\infty)$ is contained in $\Ext I(g)$, and the geodesic segment $(x_0,z_0)$ belongs to $\Int I(g)$. Moreover, 
\[
\h = \Ext I(g) \cup I(g) \cup \Int I(g)
\]
is a partition of $H$ into convex subsets such that $\partial \Ext I(g)  = I(g) = \partial \Int I(g)$.

Let
\[
 \mc K \sceq \mc K_\Gamma\sceq \bigcap_{g\in\Gamma\setminus\Gamma_\infty} \Ext I(g)
\]
denote the common part of the exteriors of all isometric spheres of $\Gamma$.

We call a point $z\in \bhg\h$ a \textit{cuspidal point} for $\Gamma$ if $\Gamma$ contains a parabolic element that stabilizes $z$. In other words, $z$ is called cuspidal if it is a representative of a cusp of $\Gamma$ or, equivalently, of $Y = \Gamma\backslash\h$. If $\infty$ is cuspidal for $\Gamma$, then 
\[
 \Gamma_\infty = \left\langle \bmat{1}{\lambda}{0}{1}\right\rangle
\]
for some $\lambda>0$. For each $r\in\R$, the set 
\[
 \mc F_\infty(r) \sceq \mc F_{\infty,\Gamma}(r) \sceq (r,r+\lambda) + i\R_{>0}
\]
is a fundamental domain for $\Gamma_\infty$ in $\h$. The following proposition states the existence of isometric fundamental domains. This proposition is well-known. For proofs in various generalities we refer to, e.g., \cite{Ford, Katok_fuchsian, Ratcliffe, Pohl_isofunddom, Pohl_diss}.

\begin{prop}\label{prop_isomdom}
Let $\Gamma$ be a geometrically finite Fuchsian group that has $\infty$ as cuspidal point. Then, for any $r\in\R$, the set
\[
 \mc F(r) \sceq \mc F_\infty(r) \cap \mc K
\]
is a convex fundamental domain for $\Gamma$ in $\h$ with finitely many sides.
\end{prop}

\subsection{Symbolic dynamics}\label{sec_symdyn}

As before, let $\Gamma$ be a Fuchsian group and set $Y=\Gamma\backslash \h$. Let $\widehat \CS$ (``cross section'') be a subset of $SY$. Suppose that $\wh\gamma$ is a geodesic on $Y$. If $\wh\gamma'(t)\in\wh\CS$, then we say that $\wh\gamma$ \textit{intersects $\wh\CS$ in $t$}. Further, $\wh\gamma$ is said to \textit{intersect $\widehat \CS$ infinitely often in the future} 
if we find a sequence $(t_n)_{n\in\N}$ with $t_n\to\infty$ as $n\to\infty$ and $\wh\gamma'(t_n) \in \wh \CS$ for all $n\in\N$. Analogously, $\wh\gamma$ is said to \textit{intersect $\widehat \CS$ infinitely often in the past} if we find
a sequence $(t_n)_{n\in\N}$ with $t_n\to-\infty$ as $n\to\infty$ and
$\wh\gamma'(t_n) \in\widehat \CS$ for all $n\in\N$. Let $\mu$ be a
measure on the space of geodesics on $Y$. The set $\wh\CS$ is called a \textit{cross section} (w.r.t.\@ $\mu$) for the geodesic flow $\widehat \Phi$ if
\begin{enumerate}[(C1)]
\item \label{C1} $\mu$-almost every geodesic $\wh\gamma$ on $Y$ intersects
$\widehat \CS$ infinitely often in the past and in the future,
\item \label{C2} each intersection of $\wh\gamma$ and $\widehat \CS$
is \textit{discrete in time:} if $\wh\gamma'(t) \in\widehat \CS$,
then there is $\eps > 0$ such that
$\wh\gamma'( (t-\eps, t+\eps) )\cap \widehat \CS = \{ \wh\gamma'(t) \}$.
\end{enumerate}

We call a subset $\wh U$ of $Y$ a \textit{totally geodesic suborbifold of $Y$} if $\pi^{-1}(\wh U)$ is a totally geodesic submanifold of $\h$. Let $\pr\colon SY\to Y$ denote the canonical projection on base points. 
If $\pr(\widehat \CS)$ is a totally geodesic suborbifold of $Y$
and $\widehat \CS$ does not contain elements tangent to $\pr(\widehat
\CS)$, then $\widehat \CS$ automatically satisfies (\apref{C}{C2}{}). 

Suppose that $\widehat \CS$ is a cross section for $\widehat \Phi$. If, in addition,
$\widehat \CS$ satisfies the property that \textit{each}
geodesic intersecting $\widehat \CS$ at all intersects it infinitely
often in the past and in the future, then $\widehat \CS$ will be called a
\textit{strong cross section}, otherwise a \textit{weak cross
section}.  Clearly, every weak cross section contains a strong cross
section.

The \textit{first return map} of $\widehat\Phi$ \wrt the strong cross
section $\widehat \CS$ is the map
\[
R\colon\widehat \CS  \to  \widehat \CS, \quad \wh v  \mapsto  \widehat{\gamma_v}'(t_0),
\]
where $v\in S\h$ with $\pi(v) = \wh v$, $\pi(\gamma_v) = \widehat{\gamma_v}$, and
\[
t_0 \sceq \min \left\{ t>0 \left\vert\  \wh\gamma'_v(t) \in\wh\CS\right.\right\}
\]
is the \textit{first return time} of $\hat v$ or $\widehat{\gamma_v}$. This definition requires that $t_0=t_0(\wh v)$ exists for each $\wh v\in \wh\CS$, which will indeed be the case in our situation. For a weak cross section $\widehat
\CS$, the first return map can only be defined on a subset of
$\widehat \CS$. In general, this subset is larger than the maximal
strong cross section contained in $\widehat \CS$.

Suppose that $\widehat \CS$ is a strong cross section and let $\Sigma$ be
an at most countable set. Decompose $\widehat \CS$ into a disjoint
union $\bigcup_{\alpha\in \Sigma}\widehat \CS_\alpha$. To each $\wh v\in
\widehat \CS$ we assign the (two-sided infinite) \textit{coding
sequence} $(a_n)_{n\in\Z} \in \Sigma^\Z$ defined by
\[
a_n \sceq \alpha \quad\text{if and only if $R^n(\wh v) \in \widehat \CS_\alpha$.}
\]
Note that $R$ is invertible  and let $\Lambda$ be the set of all
sequences that arise in this way. Then $\Lambda$ is invariant under
the left shift 
\[
\sigma\colon \Sigma^\Z \to \Sigma^\Z,\quad \left( \sigma\big((a_n)_{n\in\Z} \big)\right)_k \sceq a_{k+1}.
\]
Suppose that the map
$\widehat \CS\to \Lambda$  is also injective, which it will be in our
case. Then we have the inverse map $\Cod\colon \Lambda \to \widehat
\CS$ which maps a coding sequence to the element in $\widehat \CS$ it
was assigned to. Obviously, the diagram
\[
\xymatrix{
\wh\CS \ar[r]^R & \wh\CS
\\
\Lambda \ar[u]^{\Cod} \ar[r]^\sigma & \Lambda \ar[u]_{\Cod}
}
\]
commutes. The pair
$(\Lambda,\sigma)$ is called a \textit{symbolic dynamics} for $\widehat\Phi$ with alphabet $\Sigma$. If $\widehat \CS$ is only a weak cross section and
hence $R$ is only partially defined, then $\Lambda$ also contains
one- or two-sided finite coding sequences.

Let $\CS'$ be a \textit{set of representatives} for the cross section $\widehat
\CS$, that is, $\CS'$ is a subset of $SH$ such that $\pi\vert_{\CS'}$ is a bijection
$\CS'\to \widehat \CS$. 
Relative to $\CS'$, we define the map $\tau\colon \widehat \CS\to
\bhg \h\times \bhg \h$ by 
\[
\tau(\wh v) \sceq (\gamma_v(\infty), \gamma_v(-\infty))
\]
where $v\sceq(\pi\vert_{\CS'})^{-1}(\hat v)$. 
For some
cross sections $\widehat \CS$ it is possible to choose $\CS'$ in such a
way that $\tau$ is a bijection between $\widehat \CS$ and some subset
$\widetilde \DS$ of $\R\times\R$.  In this case the dynamical system
$(\widehat \CS, R)$ is conjugate to $(\widetilde \DS, \widetilde F)$ by $\tau$,
where $\widetilde F  \sceq \tau\circ R\circ \tau^{-1}$ is the
induced selfmap on $\widetilde \DS$ (partially defined if $\widehat \CS$
is only a weak cross section). Moreover, to construct a symbolic
dynamics for $\widehat\Phi$, one can start with a decomposition of
$\widetilde \DS$ into pairwise disjoint subsets $\widetilde D_\alpha$,
$\alpha\in \Sigma$.

Finally, let $(\Lambda, \sigma)$ be a symbolic dynamics with
alphabet $\Sigma$. Suppose that we have a map $i\colon\Lambda \to \DS$ for
some $\DS\subseteq \R$ such that $i\big( (a_n)_{n\in\Z}\big)$ depends only on
$(a_n)_{n\in\N_0}$, a (partial) selfmap $F\colon \DS\to \DS$, and a
decomposition of $\DS$ into a disjoint union $\bigcup_{\alpha\in \Sigma}
D_\alpha$ such that
\[ F\big( i( (a_n)_{n\in\Z} )\big) \in D_\alpha \quad\Leftrightarrow\quad a_1 = \alpha\]
for all $(a_n)_{n\in\Z}\in\Lambda$. Then $F$, more precisely the
triple $\big(F,i, (D_\alpha)_{\alpha\in \Sigma}\big)$, is called a
\textit{generating function for the future part} of
$(\Lambda,\sigma)$. If such a generating function exists, then the
future part of a coding sequence is independent of the past part.

\section{A sketch of the construction}\label{sec_sketch}

In this section we provide a sketch of the construction of cross sections and symbolic dynamics, illustrated by examples. For further examples we refer to \cite{Hilgert_Pohl, Moeller_Pohl, Pohl_mcf_Gamma0p, Pohl_mcf_general}. As we will see already in this sketch, the cuspidal point $\infty$ has a distinguished role. Throughout let $\Gamma$ be a geometrically finite Fuchsian group with $\infty$ as cuspidal point.

\subsection*{Additional condition on $\Gamma$}
To state the additionally required condition on $\Gamma$ we need a few notions. The \textit{height} of a point $z\in\h$ is defined to be
\[
 \height(z) \sceq \Ima z.
\]
Let $g=\textbmat{a}{b}{c}{d} \in \Gamma\mminus\Gamma_\infty$ with $c$ chosen positive, and consider its isometric sphere
\[
 I(g) = \left\{ z\in\h \ \left\vert\ \left|z+\frac{d}{c}\right| = \frac1c \right.\right\}.
\]
Its point of maximal height
\[
 s= -\frac{d}{c} + i \frac1c
\]
is called the \textit{summit} of $I(g)$.

Recall the set 
\[
 \mc K = \bigcap_{g\in\Gamma\mminus\Gamma_\infty} \Ext I(g),
\]
which is the subset of $\h$ contained in the exterior of all isometric spheres of $\Gamma$. We call an isometric sphere $I$ of $\Gamma$ \textit{relevant} if $I$ contributes nontrivially to the boundary of $\mc K$, that is, if $I\cap\partial\mc K$ contains a submanifold of $\h$ of codimension $1$. If the isometric sphere $I$ is relevant, then $I\cap\partial\mc K$ is called its \textit{relevant part}. An endpoint of the relevant part of a relevant isometric sphere is called a \textit{vertex} of $\mc K$. 

From now on we impose the following condition on $\Gamma$:
\begin{equation}\tag{A}\label{A4}
\begin{minipage}{6cm}
For each relevant isometric sphere, its summit is contained in $\partial \mathcal K$ but not a vertex of $\mc K$.
\end{minipage}
\end{equation}

\begin{example}\label{Gammaex}
\begin{enumerate}[{\rm (i)}]
\item\label{Hecke} For $n\in\N$, $n\geq 3$, let $\lambda_n\sceq 2\cos\tfrac{\pi}{n}$. The \textit{Hecke triangle group $G_n$} is the subgroup of $\PSL_2(\R)$ which is generated by 
\[
 S\sceq \bmat{0}{-1}{1}{0} \quad\text{and}\quad T_n\sceq \bmat{1}{\lambda_n}{0}{1}.
\]
The relevant isometric spheres are $I(S)$ and its $T_n^k$-translates, $k\in\Z$.
\begin{figure}[h]
\begin{center}
\includegraphics*{Hecke5.9}
\end{center}
\caption{The set $\mc K$ for $G_5$.}
\end{figure}
\item\label{Gamma05} Let 
\[
\PGamma_0(5)\sceq \left\{ \bmat{a}{b}{c}{d}\in\PSL_2(\Z) \left\vert\  c\equiv 0 \mod 5 \vphantom{\bmat{a}{b}{c}{d}\in\PSL(2,\Z)}   \right.\right\}.
\]
The isometric spheres are the sets
\[
I_{c,d} = \left\{ z\in \h\left\vert\ \left|z+\tfrac{d}{5c} \right|= \tfrac{1}{5c} \right.\right\}
\]
where $c\in\N$ and $d\in\Z$. The set $\mc K$ is given by (see Figure~\ref{Gamma05_K})
\[
\mc K = \bigcap_{d\in\Z}  \left\{ z\in \h\left\vert\ \left|z+\tfrac{d}{5} \right| > \tfrac{1}{5} \right.\right\}.
\]
\begin{figure}[h]
\begin{center}
\includegraphics*{Gamma05.8} 
\end{center}
\caption{The set $\mc K$ for $\PGamma_0(5)$.}\label{Gamma05_K}
\end{figure}
\item \label{Nichtcof} Let $S\sceq \textbmat{0}{-1}{1}{0}$ and $T\sceq \textbmat{1}{4}{0}{1}$, and denote by $\Gamma$ the subgroup of $\PSL_2(\R)$ which is generated by $S$ and $T$. The set $\mc K$ is indicated in Figure~\ref{KNichtcof}.
\begin{figure}[h]
\begin{center}
\includegraphics*{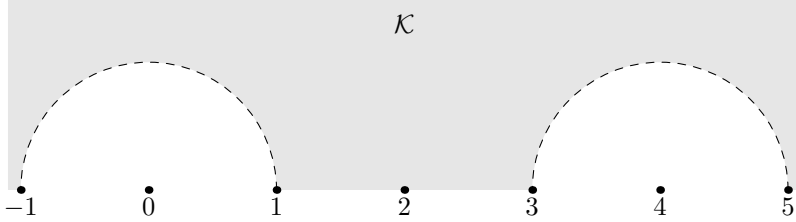} 
\end{center}
\caption{The set $\mc K$ for the group $\Gamma$ from Example~\ref{Gammaex}\eqref{Nichtcof}.}\label{KNichtcof}
\end{figure}
\end{enumerate}
\end{example}

\subsection*{The first cross section}

The (first) cross section $\wh\CS$ for the geodesic flow on $\Gamma\backslash\h$ will be defined with the help of a very specific tesselation of $\mc K$. For that we note that there are two kinds of vertices of $\mc K$. Let $v$ be a vertex of $\mc K$. If $v\in\h$, then we call $v$ an \textit{inner vertex}, otherwise $v$ is called an \textit{infinite vertex}. Now we construct a family $\mc G$ of geodesic segments in the following way: Whenever $v$ is an infinite vertex of $\mc K$, then the (complete) geodesic segment $(v,\infty)$ belongs to this family. Moreover, whenever $s$ is the summit of a relevant isometric sphere, then the geodesic segment $(s,\infty)$ belongs to this family. Let $\wt\CS$ denote the set of unit tangent vectors on $\h$ which are based on the geodesic segments in this family but which are not tangent to any of these geodesic segments. In other words, $\wt\CS$ consists of the unit tangent vectors whose base points are on any of these geodesic segments and which point to the right or the 
left, but not straight up or down. Then
\[
 \wh\CS \sceq \pi(\wt\CS)
\]
is a weak cross section for the geodesic flow. By eliminating a certain set of vectors from it, we will also construct a strong cross section.

\subsection*{Construction of a set of representatives}
The family $\mc G$ of geodesic segments induces a tesselation of $\mc K$ into hyperbolic triangles, quadrilateral and strips. We call these tesselation elements \textit{precells in $\h$}. The union of certain finite subfamilies of these precells in $\h$ form isometric fundamental regions or even, if the union is connected, isometric fundamental domains. This relation is useful for the following two reasons.

Pick such a subfamily of precells in $\h$ and let $\CS'_\pre$ denote the set of unit tangent which are based on the vertical sides of these precells and point into them. Then the link to isometric fundamental regions yields that $\CS'_\pre$ is a set of representatives for $\wh\CS$. This set, however, is not very useful for coding purposes. For this reason, we use another property of isometric fundamental regions to modify the set of representatives.

Isometric fundamental regions allow to define cycles of their sides as in the Poincar\'e Fundamental Polyhedron Theorem. The acting elements in these cycles are just the generating elements of the relevant isometric spheres which contribute to the boundary of the fundamental region. Using these cycles we glue translates of precells in $\h$ to ideal polyhedrons in $\h$, that is, polyhedrons all of which vertices are in $\bhg\h$. Even more, all vertices are in $\Gamma.\infty$. At this point, Condition~\eqref{A4} is essential. We call these polyhedron \textit{cells in $\h$}. All the arising cells in $\h$ have two vertical sides, and all the non-vertical sides are $\Gamma$-translates of some vertical sides of some cells in $\h$. The translation elements can be determined by an algorithms from the generators of the relevant isometric spheres.

Lifting this construction to the unit tangent bundle $S\h$ and using this to redistribute $\CS'_\pre$ allows us to find a set of representatives $\CS'$ such that $\CS'$ is the disjoint union
\[
 \CS'= \bigcup_{\alpha\in A} \CS'_\alpha
\]
for some finite index set $A$ such that for each $\alpha \in A$ there exists a vertical side (a complete geodesic segment) of some cell in $\h$ such that $\CS'_\alpha$ is the set of unit tangent vectors based on this side and pointing into the cell.

The definition of precells in $\h$ and the construction of cells in $\h$ from precells is based on ideas in \cite{Vulakh}. Our construction differs from Vulakh's in three important aspects: We define three kinds of precells in $H$ of which only one are precells in sense of Vulakh. Finally, contrary to Vulakh, we extend the considerations to precells and cells in unit tangent bundle. 

\begin{example}\label{Heckecont}\label{exprecells1}
The set
 \[
\fd_n\sceq \left\{ z\in \h\ \left\vert\ |z|>1,\ |\Rea z| < \tfrac{\lambda_n}{2}\right.\right\}
\]
is an isometric fundamental domain for the Hecke triangle group $G_n$ (see Figure~\ref{FHecke}). 
\begin{figure}[h]
\begin{center}
\includegraphics*{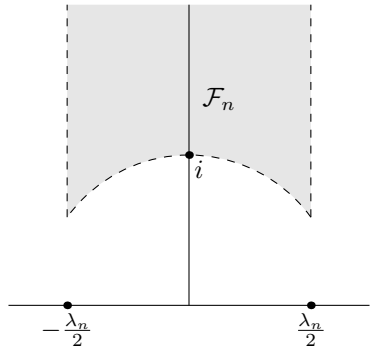} 
\end{center}
\caption{An isometric fundamental domain $\mathcal F_5$ for $G_5$ in $\h$.}\label{FHecke}
\end{figure}
Hecke triangle groups have only one precell in $\h$, up to equivalence under 
\[
 \big( G_n\big)_\infty = \left\langle T_n \right\rangle.
\]
It is indicated in Figure~\ref{precellHecke}.
\begin{figure}[h]
\begin{center}
\includegraphics*{Hecke5.17} 
\end{center}
\caption{The precell $\pch$ of $G_5$.}\label{precellHecke}
\end{figure}
\end{example}
Figure~\ref{decompHecke} provides an idea of the lifting and redistribution procedure.
\begin{figure}[h]
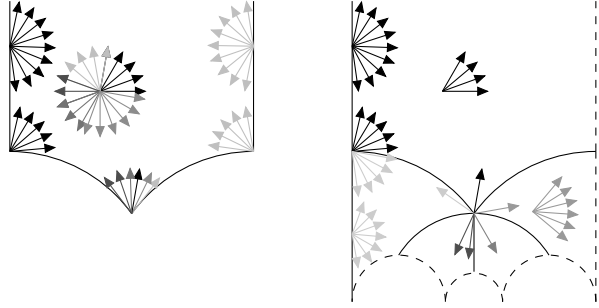

\begin{center}
\includegraphics*{Hecke5.18} \hspace{1cm}
\includegraphics*{Hecke5.19}
\end{center}
\caption{Lifting to $S\h$ and defining $\CS'$ for $G_5$.}\label{decompHecke}
\end{figure}

\subsection*{Discrete dynamical system on the boundary and associated transfer operator families}
The relation between isometric fundamental regions and the specific way of lifting to $S\h$ yields the following property of $\CS'$. Whenever $L$ is the side of some $\Gamma$-translate of some cell in $\h$ and $L'$ is the set of unit tangent vectors based on that side such that all vectors are pointing into this $\Gamma$-translate or all vectors pointing out of the $\Gamma$-translate, then there is a unique pair $(\alpha,g)\in A\times \Gamma$ such that $L'= g.\CS'_\alpha$. Both, $\alpha$ and $g$ can be determined algorithmically. Further, $\Gamma.\CS'$ is just the union of all such $L'$.

These properties allow to read off the induced discrete dynamical system $F$ on the boundary of $\h$. It is given by finitely many local diffeomorphisms of the form
\begin{align}\label{localdiffeo}
 \nonumber\big( (I_{\alpha_1}\times J_{\alpha_1}) \cap g^{-1}.(I_{\alpha_2}\times J_{\alpha_2}) \big) \times\{\alpha_1\} &\to \big( g.(I_{\alpha_1}\times J_{\alpha_1}) \cap (I_{\alpha_2}\times J_{\alpha_2})\big)\times\{\alpha_2\},
\\
 (x,y,\alpha_1) &\mapsto (g.x,g.y,\alpha_2),
\end{align}
where $\alpha_1, \alpha_2\in A$, $I_{\alpha_1}, I_{\alpha_2}, J_{\alpha_1}, J_{\alpha_2}$ are intervals, and $g\in\Gamma$ depends on $\alpha_1,\alpha_2$.

The associated transfer operator $\mc L_{F,\beta}$ with parameter $\beta\in\C$ is then
\[
 \mc L_{F,\beta}f(w) = \sum_{u\in F^{-1}(w)}\frac{f(u)}{|F'(u)|^\beta},
\]
defined on spaces of functions $f$ on the domain of $F$. A more explicit formula is provided in Section~\ref{sec_transop}.

\begin{example}
Figure~\ref{nextlastHecke} shows the translates of $\CS'$ for the Hecke triangle group $G_5$. Here, we set
\[
U_n \sceq T_nS = \bmat{\lambda_n}{-1}{1}{0}. 
\]
For general Hecke triangle groups $G_n$, the picture is similar. 
\begin{figure}[h]
\begin{center}
\includegraphics*{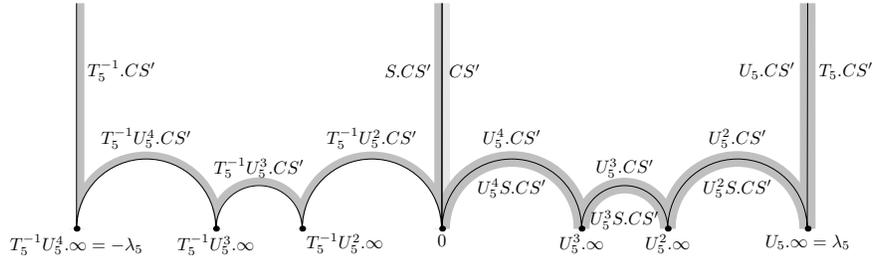} 
\end{center}
\caption{The shaded parts are translates of $\CS'$ (in unit tangent bundle) as indicated.}\label{nextlastHecke}
\end{figure}
If one restricts here to the forward-direction of the geodesic flow, then the discrete dynamical system for $G_n$ is  
\[
F\colon\R_{>0}\mminus G_n.\infty\to \R_{>0}\mminus G_n.\infty,
\]
given by the bijections
\[
 (g_k.0,g_k.\infty)\mminus G_n.\infty \to \R_{>0}\mminus G_n.\infty,\quad x\mapsto g_k^{-1}.x
\]
for $k=1,\ldots, n-1$ with
\[
 g_k \sceq U_n^kS.
\]
The associated transfer operator $\mc L_{F,\beta}$ with parameter $\beta\in\C$ is then
\[
 \mc L_{F,\beta}f(x) = \sum_{k=1}^{n-1} \big(g'_k(x)\big)^\beta f(g_k.x),
\]
defined on functions $f\colon \R_{>0}\mminus G_n.\infty \to \C$. For the modular group $\PSL_2(\Z) = G_3$ the transfer operator becomes
\[
 \mc L_{F,\beta}f(x) = f(x+1) + (x+1)^{-2\beta}f\left(\frac{x}{x+1}\right).
\]
The $1$-eigenfunctions of this transfer operator are characterized by the functional equation~\eqref{LZ}. In \cite{Moeller_Pohl}, the relation between eigenfunctions of this transfer operator and period functions, as well as the relation to Mayer's transfer operator and the factorization \eqref{factorMayer} are discussed in detail.
\end{example}

\subsection*{The second cross section and the reduced discrete dynamical system}
For some lattices $\Gamma$ it may happen that in some of the local diffeomorphisms of the form \eqref{localdiffeo} we have $g=\id$. To avoid that and at the same time to be able to eliminate the bits $\{\alpha_1\}$ and $\{\alpha_2\}$ from the domain and range of $F$, we will shrink the cross section $\wh\CS$ and the set of representatives $\CS'$ in an algorithmic way to deduce a new cross section $\CS_\rd$ with set of representatives $\CS'_\rd$. In terms of notions introduced only later, we will eliminate all vectors from $\wh\CS$ with $\id$ in the label.

\subsection*{A group that does not satisfy (A)}\label{example}
The previous examples as well as those from \cite{Hilgert_Pohl, Pohl_mcf_Gamma0p, Pohl_mcf_general} show that there are several geometrically finite Fuchsian groups with $\infty$ as cuspidal point and which satisfy the condition \eqref{A4}. We now provide an example of a geometrically finite Fuchsian group with $\infty$ as cuspidal point, which does not satisfy \eqref{A4}. 

Let $\Gamma$ be the subgroup of $\PSL_2(\R)$ which is generated by the matrices
\[
t \sceq \bmat{1}{\tfrac{17}{11}}{0}{1},\quad g_1 \sceq \bmat{18}{-5}{11}{-3} \quad\text{and}\quad g_2\sceq \bmat{3}{-1}{10}{-3}.
\]
Set 
\[
g_3 \sceq g_1g_2 = \bmat{4}{-3}{3}{-2},
\]
let $\fd_\infty \sceq \big(\tfrac{2}{11}, \tfrac{19}{11}\big) + i\R^+$ and (see Figure~\ref{vulakhF})
\[
\fd\sceq \fd_\infty \cap \Ext I(g_1) \cap \Ext I(g_1^{-1}) \cap \Ext I(g_2) \cap \Ext I(g_1g_2) \cap \Ext I\big( (g_1g_2)^{-1}\big).
\]
\begin{figure}[h]
\includegraphics*{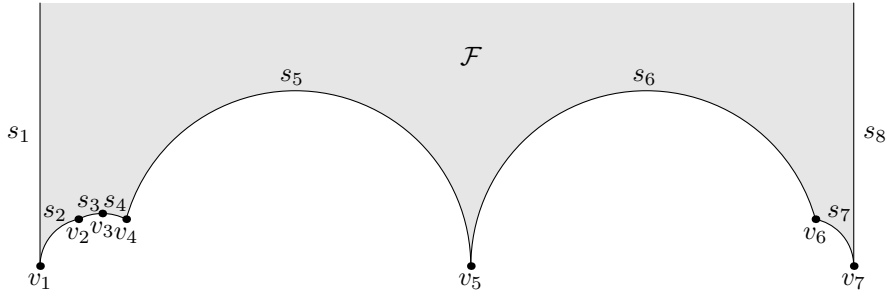} 
\caption{The fundamental domain $\mc F$}\label{vulakhF}
\end{figure}

We indicated the points
\[
\begin{array}{lll}
v_0 \sceq \infty & v_3 \sceq \tfrac{3}{10}+ \tfrac{i}{10}  & v_6 \sceq \tfrac{91}{55} + i\tfrac{\sqrt{24}}{55}
\\
v_1 \sceq \tfrac{2}{11} & v_4 \sceq \tfrac{19}{55} + i\tfrac{\sqrt{24}}{55} & v_7 \sceq \tfrac{19}{11}
\\
v_2\sceq \tfrac{14}{55} + i \tfrac{\sqrt{24}}{55} & v_5 \sceq 1
\end{array}
\]
and the geodesic segments
\[
\begin{array}{ccc}
s_1 \sceq [v_0,v_1] & s_4 \sceq [v_3,v_4] & s_7 \sceq [v_6,v_7]
\\
s_2 \sceq [v_1, v_2] & s_5\sceq [v_4,v_5] & s_8 \sceq [v_7,v_0].
\\
s_3 \sceq [v_2,v_3] & s_6 \sceq [v_5,v_6]
\end{array}
\]

\begin{prop}\label{Ffunddom}
The set $\fd$ is a fundamental domain for $\Gamma$ in $\h$.
\end{prop}

\begin{proof}
The sides of $\fd$ are $s_1,s_2, s_3\cup s_4, s_5,s_6,s_7$ and $s_8$. Further we have the side-pairings $ts_1 = s_8$, $g_1s_2 = s_7$, $g_2s_3 = s_4$ and $g_3s_5=s_6$, where $tv_0=v_0$, $tv_1=v_7$, $g_1v_1=v_7$, $g_1v_2=v_6$, $g_2v_2=v_4$, $g_2v_3=v_3$, $g_3v_4=v_6$ and $g_3v_5=v_5$. Now Poincar\'e's Fundamental Polyhedron Theorem (see e.g.\@ \cite{Maskit}) yields that $\fd$ is a fundamental domain for the group generated by $t,g_1,g_2$ and $g_3$. This group is exactly $\Gamma$.
\end{proof}

\begin{prop}
$\Gamma$ does not satisfy \eqref{A4}.
\end{prop}

\begin{proof}
Proposition~\ref{Ffunddom} states that $\fd$ is a fundamental domain for $\Gamma$ in $\h$. Its shape and side-pairings yield that it is an isometric fundamental domain. Therefore, the isometric sphere $I(g_1)$ is relevant and $s_1$ is its relevant part. The summit of $I(g_1)$ is $s\sceq \tfrac{3+i}{11}$. One easily calculates that $s\in \Int I(g_2)$. Thus, $\Gamma$ does not satisfy \eqref{A4}.
\end{proof}

In \cite{Vulakh}, Vulakh states that each geometrically finite subgroup of $\PSL_2(\R)$ for which $\infty$ is a cuspidal point satisfies \eqref{A4}. The previous example shows that this statement is not right. This property is crucial for the results in \cite{Vulakh}. Thus, Vulakh's constructions do not apply to such a huge class of groups as he claims.

\section{Precells in $\h$}\label{sec_preH}

Let $\Gamma$ be a geometric finite Fuchsian group with $\infty$ as cuspidal point and which satisfies Condition~\eqref{A4}.  To avoid empty statements suppose that $\Gamma\not=\Gamma_\infty$, which means that there are relevant isometric spheres. 

In this section we introduce the notion of precells in $\h$ and basal families of precells in $\h$. Moreover, we discuss their relation to fundamental regions and study some of their properties. We recall that 
\[
 \mc K \sceq \bigcap_{g\in\Gamma\mminus\Gamma_\infty} \Ext I(g).
\]

\subsection{The structure of $\mc K$}\label{sec_boundint}

We start with a short consideration of the vertex structure of $\mc K$. 

The set of all isometric spheres of $\Gamma$ need not be locally finite. For example, in the case of the modular group $\PSL_2(\Z)$, each neighborhood of $0$ in $\hg$ contains infinitely many isometric spheres. However, from $\Gamma$ being geometrically finite, it follows immediately that the set of \textit{relevant} isometric spheres is locally finite. In turn, the set of infinite vertices of $\mc K$ has no accumulation points. Moreover, if $v$ is an inner vertex of $\mc K$, then there are exactly two distinct relevant isometric spheres such that $v$ is a common endpoint of their relevant parts. If $v$ is an infinite vertex, then two situations can occur. If $v$ is an endpoint of the relevant parts of two distinct relevant isometric spheres, we call $v$ a \textit{two-sided infinite vertex}. Otherwise we call $v$ a \textit{one-sided infinite vertex}.

Straightforward geometric arguments prove the following statement on the local situation at one-sided infinite vertices of $\mc K$. For this let 
\[
\pr_\infty\colon \hg\mminus\{\infty\} \to \R,\quad \pr_\infty(z) \sceq z-\height(z) = \Rea z
\]
denote the geodesic projection from $\infty$ to $\bhg\h$. For $a,b\in\R$ we set
\[
 \langle a, b\rangle \sceq
\begin{cases}
[a,b] & \text{if $a\leq b$,}
\\
[b,a] & \text{otherwise.} 
\end{cases}
\]

\begin{prop}\label{boundint}
Let $v$ be a one-sided infinite vertex of $\mc K$. Then there exists a unique one-sided infinite vertex $w$ of $\mc K$ such that the strip $\pr_\infty^{-1}(\langle v, w\rangle) \cap \h$ is contained in $\mc K$. In particular, $\pr_\infty^{-1}(\langle v,w\rangle)$ does not intersect any isometric sphere in $\h$, and, of all vertices of $\mc K$, $\pr_\infty^{-1}(\langle v, w\rangle)$ contains only $v$ and $w$.
\end{prop}

For any one-sided infinite vertex $v$ of $\mc K$, we call the interval $\langle v, w\rangle$ from Proposition~\ref{boundint} a \textit{boundary interval}, and we call $w$ the one-sided infinite vertex \textit{adjacent} to $v$. Boundary intervals will be needed for the definition and investigation of strip precells defined below.

\begin{example}
Recall the group $\Gamma$ from Example~\ref{Gammaex}\eqref{Nichtcof}. The boundary intervals of $\mc K$ are the intervals $[1+4m,3+4m]$ for each $m\in\Z$. 
\end{example}

\subsection{Precells in $\h$ and basal families}\label{sec_summit}
We now define the notion of precells in $\h$. 

\begin{defi}\label{def_preH}
Let $v$ be a vertex of $\mathcal K$.  Suppose first that $v$ is an inner vertex or a two-sided infinite vertex. Then there are (exactly) two relevant isometric spheres $I_1$, $I_2$ with relevant parts $[a_1,v]$ resp.\@ $[v, b_2]$. Let $s_1$ resp.\@ $s_2$ be the summit of $I_1$ resp.\@ $I_2$. By Condition~\eqref{A4}, the summits $s_1$ and $s_2$ do not coincide with $v$.

If $v$ is a two-sided infinite vertex, then define $\pch_1$ to be the hyperbolic triangle\footnote{We consider the boundary of the triangle in $\h$ to belong to it.} with vertices $v$, $s_1$ and $\infty$, and  define $\pch_2$ to be the hyperbolic triangle with vertices $v$, $s_2$ and $\infty$. The sets $\pch_1$ and $\pch_2$ are the \textit{precells in $\h$} attached to $v$. Precells arising in this way are called \textit{cuspidal}. 

If $v$ is an inner vertex, then let $\pch$ be the hyperbolic quadrilateral with vertices $s_1$, $v$, $s_2$ and $\infty$. The set $\pch$ is the \textit{precell in $\h$} attached to $v$. Precells that are constructed in this way are called \textit{non-cuspidal}.

Suppose now that $v$ is a one-sided infinite vertex. Then there exist exactly one relevant isometric sphere $I$ with relevant part $[a,v]$ and a unique one-sided infinite vertex $w$ other than $v$ such that $\pr_\infty^{-1}(\langle v, w\rangle)$ does not contain vertices other than $v$ and $w$ (see Proposition~\ref{boundint}). Let $s$ be the summit of $I$.

Define $\pch_1$ to be the hyperbolic triangle with vertices $v$, $s$ and $\infty$, and define $\pch_2$ to be the vertical strip $\pr_\infty^{-1}(\langle v,w\rangle) \cap \h$. The sets $\pch_1$ and $\pch_2$ are the \textit{precells in $\h$} attached to $v$. The precell $\pch_1$ is called \textit{cuspidal}, and $\pch_2$ is called a \textit{strip precell}.
\end{defi}

\begin{example}\label{exprecells2}
The precells in $\h$ of the congruence group $\PGamma_0(5)$ from Example~\ref{Gammaex}\eqref{Gamma05} are indicated in Figure~\ref{precellsGamma05} up to $\PGamma_0(5)_\infty$-equivalence.
\begin{figure}[h]
\begin{center}
\includegraphics*{Gamma05.3} 
\end{center}
\caption{Precells in $\h$ of $\PGamma_0(5)$.}\label{precellsGamma05}
\end{figure}

The inner vertices of $\mc K$ are 
\[
 v_k = \frac{2k+1}{10} + i\frac{\sqrt{3}}{10}, \qquad k=1,2,3,
\]
and their translates under $\PGamma_0(5)_\infty$. The summits of the indicated isometric spheres are 
\[
 m_k = \frac{k}{5} + \frac{i}{5}, \qquad k=1,\ldots, 4.
\]
The group $\PGamma_0(5)$ has cuspidal as well as non-cuspidal precells in $\h$, but no strip precells.
\end{example}

\begin{example}\label{exprecells3}
The precells in $\h$ of the group $\Gamma$ from Example~\ref{Gammaex}\eqref{Nichtcof} are up to $\Gamma_\infty$-equivalence one strip precell $\pch_1$ and two cuspidal precells $\pch_2,\pch_3$ as indicated in Figure~\ref{precellsNichtcof}.
\begin{figure}[h]
\begin{center}
\includegraphics*{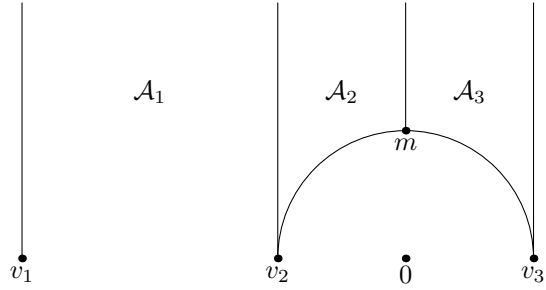} 
\end{center}
\caption{Precells in $\h$ of $\Gamma$.}\label{precellsNichtcof}
\end{figure}
Here, $v_1=-3$, $v_2=-1$, $v_3=1$ and $m=i$.
\end{example}

Condition~\eqref{A4} yields that the relation between different precells and between the precells and $\mc K$ are well-structured.

\begin{prop}\label{precellsum}
\begin{enumerate}[{\rm (i)}]
\item \label{precellsproj}
If $\pch$ is a precell in $\h$, then 
\[
\pch = \pr_\infty^{-1}\big(\pr_\infty(\pch)\big) \cap \overline{\mc K} \quad\text{and}\quad \pch^\circ = \pr_\infty^{-1}\big(\pr_\infty(\pch^\circ)\big) \cap \mc K.
\]
\item \label{boundarypreH}
If two precells in $\h$ have a common point, then either they are identical or they coincide exactly at a common vertical side. 
\item \label{decompK}
The set $\overline{\mathcal K}$ is the  essentially disjoint union of all precells in $\h$,
\[
 \overline{\mathcal K} = \bigcup \{ \pch \mid \text{$\pch$ precell in $\h$} \},
\]
and $\mc K$ contains the disjoint union of the interiors of all precells in $\h$, 
\[
 \bigcup\{ \pch^\circ \mid \text{$\pch$ precell in $\h$} \} \subseteq \mc K.
\]
\end{enumerate}
\end{prop}

\begin{proof}
We use the notation from Definition~\ref{def_preH} to discuss the relation between precells and $\mc K$. 

Suppose first that $\pch$ is a non-cuspidal precell in $\h$ attached to the inner vertex $v$. Condition~\eqref{A4} implies that the summits $s_1$ and $s_2$ are contained in the relevant parts of the relevant isometric spheres $I_1$ and $I_2$, respectively. Therefore they are contained in $\partial\mc K$. Hence $\pch$ is the subset of $\overline{\mc K}$ with the two vertical sides $[s_1,\infty]$ and $[s_2,\infty]$, and the two non-vertical ones $[s_1,v]$ and $[v,s_2]$. The geodesic projection of $\pch$ from $\infty$ is 
\[ \pr_\infty(\pch) = \langle \Rea s_1, \Rea s_2\rangle. \]

Suppose now that $\pch$ is a cuspidal precell in $\h$ attached to the infinite vertex $v$. Then $\pch$ has two vertical sides, namely $[v,\infty]$ and $[s,\infty]$, and a single non-vertical side, namely $[v,s]$. As for non-cuspidal precells we find that $[v,s]$ is contained in the relevant part of some relevant isometric sphere, and hence $\pch$ is a subset of $\overline{\mc K}$. The geodesic projection of $\pch$ from $\infty$ is 
\[
 \pr_\infty(\pch) = \langle\Rea s , v\rangle.
\]
Suppose finally that $\pch$ is the strip precell $\pr_\infty^{-1}( \langle v, w\rangle) \cap \h$. Then $\pch$ is attached to the two vertices $v$ and $w$. It has the two vertical sides $[v,\infty]$ and $[w,\infty]$ and no non-vertical ones. The geodesic projection of $\pch$ from $\infty$ is
\[
 \pr_\infty(\pch) = \langle v,w\rangle.
\]
By Proposition~\ref{boundint}, $\pch$ is contained in $\mc K$. From these observations, the statements follow easily.
\end{proof}

Before we investigate the relation between precells in $\h$ and fundamental regions in Theorem~\ref{precellsH} below, we state a few properties of isometric spheres. Their proofs are straightforward.

\begin{lemma}\label{isoprops}
Let $g\in\Gamma\mminus\Gamma_\infty$. 
\begin{enumerate}[{\rm (i)}]
\item We have $g.I(g) = I(g^{-1})$.
\item If $s$ the summit of $I(g)$, then $g.s$ is the summit of $I(g^{-1})$.
\item The geodesic projection $\pr_\infty(s)$ of the summit $s$ of $I(g)$ is the center $g^{-1}.\infty$ of $I(g)$.
\item If $I(g)$ is relevant with relevant part $[a,b]$, then $I(g^{-1})$ is relevant with relevant part $g.[a,b] = [g.a,g.b]$.
\end{enumerate}
\end{lemma}

Let $\Lambda$ be a Fuchsian group. A subset $\fd$ of $H$ is called a \textit{closed} fundamental region
for $\Lambda$ in $\h$ if $\fd$ is closed and $\fd^\circ$ is a fundamental region for $\Lambda$ in $\h$. If, in addition, $\fd^\circ$ is connected, then $\fd$ is said to be a \textit{closed fundamental domain} for $\Lambda$ in $\h$. Note that if $\fd$ is a non-connected fundamental region for $\Lambda$ in $\h$, then $\overline \fd$ can happen to be a closed fundamental domain.

Let $\{A_j\mid j\in J\}$ be a family of real submanifolds (possibly with boundary) of $\h$ or $\hg$, and let $n\sceq \max\{\dim A_j \mid j\in J\}$. We call the union 
\[
 \bigcup_{j\in J} A_j
\]
\textit{essentially disjoint} if for each $i,j\in J$, $i\not=j$, the intersection $A_i\cap A_j$ is contained in a real submanifold (possibly with boundary) of dimension $n-1$. 

\begin{thm}\label{precellsH}
There exists a set $\{ \pch_j \mid j\in J\}$, indexed by $J$, of precells in $\h$ such that the (essentially disjoint) union $\bigcup_{j\in J}\pch_j$ is a closed fundamental region for $\Gamma$ in $\h$. The set $J$ is finite and its cardinality does not depend on the choice of the specific set of precells. The set $\{ \pch_j \mid j\in J\}$ can be chosen such that $\bigcup_{j\in J} \pch_j$ is a closed fundamental domain for $\Gamma$ in $\h$. In each case, the (disjoint) union $\bigcup_{j\in J} \pch_j^\circ$ is a fundamental region for $\Gamma$ in $\h$.
\end{thm}

\begin{proof}
By Proposition~\ref{precellsum}\eqref{boundarypreH}, the union of each family of pairwise different precells in $\h$ is essentially disjoint. Let $r$ be the center of some relevant isometric sphere $I$. Let $r\sceq \Rea s$ for the summit $s$ of some relevant sphere of $\Gamma$ or let $r\sceq \Rea v$ for an infinite vertex $v$ of $\mc K$. Let $\lambda > 0$ be the unique positive number such that 
\[
 \Gamma_\infty = \left\langle \bmat{1}{\lambda}{0}{1}\right\rangle.
\]
From Lemma~\ref{isoprops} and Proposition~\ref{precellsum} one easily deduces that 
\[ 
\overline{\fd(r)} = \overline{\fd_\infty(r)} \cap \overline{\mc K} = \left( [r,r+\lambda]+i\R_{>0}\right) \cap \overline{\mc K}
\] 
decomposes into a finite set $\fpch\sceq \{ \pch_j\mid j\in J\}$ of precells in $\h$. Clearly, $\overline{\fd(r)}$ is a closed fundamental domain.

Let $\fpch_2\sceq \{ \pch_k\mid k\in K\}$ be a set of precells in $\h$ such that $\fd\sceq \bigcup_{k\in K}\pch_k$ is a closed fundamental region for $\Gamma$ in $\h$. Proposition~\ref{precellsum}\eqref{decompK} implies that 
\[
\overline{\mc K} = \bigcup\{ h.\pch \mid h\in\Gamma_\infty,\ \pch\in\fpch\}.
\]
Let $\pch_k\in\fpch_2$ and pick $z\in\pch_k^\circ$. Then there exists $h_k\in\Gamma_\infty$ and $j_k\in J$ such that $h_k.z\in \pch_{j_k}$. Therefore $h_k.\pch_k^\circ \cap \pch_{j_k} \not=\emptyset$. The $\Gamma_\infty$-invariance of $\mc K$ shows that $h_k.\pch_k$ is a precell in $\h$. Then Proposition~\ref{precellsum}\eqref{boundarypreH} implies that $h_k.\pch_k = \pch_{j_k}$, and in turn $h_k$ and $j_k$ are unique. We will show that the map $\varphi\colon K\to J$, $k\mapsto j_k$, is a bijection. To show that $\varphi$ is injective suppose that there are $l,k\in K$ such that $j_k=j_l\seqc j$. Then $h_l.\pch_l = \pch_j = h_k.\pch_k$, hence $h_k^{-1}h_l.\pch_l =\pch_k$. In particular, $h_k^{-1}h_l.\pch_l^\circ \cap \pch_k^\circ \not=\emptyset$. Since $\bigcup_{h\in K} \pch_h^\circ \subseteq \fd^\circ$ and $\fd^\circ$ is a fundamental region, it follows that $h_k^{-1}h_l=\id$ and $l=k$. Thus, $\varphi$ is injective. To show surjectivity let $j\in J$ and $z\in \pch_j^\circ$. Then there exists $g\in\Gamma$ and 
$k\in K$ such that $g.z\in \pch_k$. On the other hand, $\pch_k = h_k^{-1}.\pch_{j_k}$. Hence $h_kg.\pch_j^\circ \cap \pch_{j_k}\not=\emptyset$. Since $\pch_j$ and $\pch_{j_k}$ are convex polyhedrons, it follows that $h_kg.\pch_j^\circ \cap \pch_{j_k} \not=\emptyset$. Since $\fd(r)$ is a fundamental region and $\pch_j^\circ, \pch_{j_k}^\circ\subseteq \fd(r)$, we find that $h_kg=\id$ and $j=j_k$. Hence, $\varphi$ is surjective. It follows that $\# K = \# J$.

It remains to show that the disjoint union $P\sceq \bigcup_{k\in K} \pch_k^\circ$ is a fundamental region for $\Gamma$ in $\h$. Obviously, $P$ is open and contained in $\fd^\circ$. This shows that $P$ satisfies (\apref{F}{F1}{})  and (\apref{F}{F2}{}). Since $\overline{(\pch^\circ)} = \pch$ for each precell in $\h$ and $K$ is finite, it follows that 
\[
\overline{P} = \overline{\bigcup_{k\in K} \pch_k^\circ} = \bigcup_{k\in K} \pch_k = \fd.
\]
Hence, $P$ satisfies (\apref{F}{F3}{}) as well, and thus it is a fundamental region for $\Gamma$ in $\h$.
\end{proof}

Each set $\fpch \sceq \{ \pch_j \mid j\in J\}$, indexed by $J$, of precells in $\h$ with the property that $\fd\sceq \bigcup_{j\in J} \pch_j$ is a closed fundamental region is called a \textit{basal family of precells in $\h$} or a \textit{family of basal precells in $\h$}. If, in addition, $\fd$ is connected, then $\fpch$ is called a \textit{connected basal family of precells in $\h$} or a \textit{connected family of basal precells in $\h$}.

\begin{example}
Recall the Examples~\ref{exprecells1}, \ref{exprecells2} and \ref{exprecells3}.
The set $\{\pch\}$ of precells in $\h$ for $G_n$ is a connected basal family of precells in $\h$, as well as $\{\pch(v_0),\ldots, \pch(v_4)\}$ for $\PGamma_0(5)$, and $\{\pch_1,\pch_2,\pch_3\}$ for $\Gamma$.
\end{example}

The proof of Theorem~\ref{precellsH} shows the following statements. Let 
\[
 t_\lambda \sceq \bmat{1}{\lambda}{0}{1}
\]
with $\lambda>0$ be such that $\Gamma_\infty = \langle t_\lambda\rangle$.

\begin{cor}\label{props_preH}
Let $\fpch$ be a basal family of precells in $\h$.
\begin{enumerate}[{\rm (i)}]
\item For each precell $\pch$ in $\h$ there exists a unique pair $(\pch',m)\in \fpch\times \Z$ such that $t_\lambda^m.\pch' = \pch$.
\item For each $\pch\in\fpch$ choose an element $m(\pch)\in \Z$. Then $\big\{ t_\lambda^{m(\pch)}.\pch\ \big\vert\  \pch\in\fpch\big\}$ is a basal family of precells in $H$. For each $\pch\in\fpch$, the precell $t_\lambda^{m(\pch)}.\pch$ is of the same type as $\pch$.
\item The set $\overline{\mc K}$ is the essentially disjoint union $\bigcup\{ h.\pch \mid h\in\Gamma_\infty,\ \pch\in\fpch\}$.
\end{enumerate}
\end{cor}

\subsection{The tesselation of $\h$ by basal families of precells}\label{sec_tess}

The following proposition is crucial for the construction of cells in $\h$ from precells in $\h$. Note that the element $g\in\Gamma\mminus\Gamma_\infty$ in this proposition depends not only on $\pch$ and $b$ but also on the choice of the basal family $\fpch$ of precells in $\h$. In this section we will use the proposition as one ingredient for the proof that the family of $\Gamma$-translates of all precells in $\h$ is a tesselation of $\h$.

\begin{prop}\label{structureglue} Let $\fpch$ be a basal family of precells in $\h$.
Let $\pch\in\fpch$ be a basal precell that is not a strip precell, and suppose that $b$ is a non-vertical side of $\pch$. Then there is a unique element $g\in\Gamma\mminus\Gamma_\infty$ such that $b\subseteq I(g)$ and $g.b$ is the non-vertical side of some basal precell $\pch'\in\fpch$. If $\pch$ is non-cuspidal, then $\pch'$ is non-cuspidal, and, if $\pch$ is cuspidal, then $\pch'$ is cuspidal.
\end{prop}

\begin{proof} Let $I$ be the (relevant) isometric sphere with $b\subseteq I$. We will show first that there is a generator $g$ of $I$ such that $g.b$ is a side of some basal precell. Then $g.b\subseteq g.I(g) = I(g^{-1})$, which implies that $g.b$ is a non-vertical side.

Let $h\in\Gamma\mminus\Gamma_\infty$ be any generator of $I$, let $s$ be the summit of $I$ and $v$ the vertex of $\mathcal K$ that $\pch$ is attached to. Then $b=[v,s]$. Further, $b$ is contained in the relevant part of $I=I(h)$. By Lemma~\ref{isoprops}, the set $h.b=[h.v,h.s]$ is contained in the relevant part of the relevant isometric sphere $I(h^{-1})$, the point $h.v$ is a vertex of $\mathcal K$ and $h.s$ is the summit of $I(h^{-1})$. Thus, there is a unique precell $\pch_h$ with non-vertical side $h.b$. By Corollary~\ref{props_preH}, there is a unique basal precell $\pch'$ and a unique $m\in\Z$ such that 
\[
\pch_h = t_\lambda^m.\pch' = \pch' + m\lambda.
\]
Then $t_\lambda^{-m}h.b$ is a non-vertical side of $\pch'$, and $t_\lambda^{-m}h.b$ is contained in the relevant part of the relevant isometric sphere $I(h^{-1}) - m\lambda = I(h^{-1}t_\lambda^m) = I((t_\lambda^{-m}h)^{-1})$. Clearly, also $g\sceq t_\lambda^{-m}h$ is a generator of $I$. 

To prove the uniqueness of $g$, let $k$ be any generator of $I$. Then there exists a unique $n\in\Z$ such that $k=t_\lambda^n h$. Thus, $k.b = t_\lambda^nh.b = h.b+n\lambda$ and therefore $\pch_k = \pch_h + n\lambda$. Then 
\[
 \pch_k = \pch' + m\lambda + n\lambda = t_\lambda^{m+n}. \pch',
\]
and $t_\lambda^{-(m+n)}k$ is a generator of $I$ such that $t_\lambda^{-(m+n)}k.b$ is a side of some basal precell. Moreover,
\[
 t_\lambda^{-(m+n)}k = t_\lambda^{-m}t_\lambda^{-n}k = t_\lambda^{-m}h = g.
\]
This shows the uniqueness. 

The basal precell $\pch'$ cannot be a strip precell, since it has a non-vertical side. Finally, $\pch$ is cuspidal if and only if $v$ is an infinite vertex. This is the case if and only if $gv$ is an infinite vertex, which is equivalent to $\pch'$ being cuspidal. This completes the proof.
\end{proof}

\begin{lemma}\label{intervert}
Let $\pch$ be a precell in $\h$. Suppose that $S$ is a vertical side of $\pch$. Then there exists a precell $\pch'$ in $\h$ such that $S$ is a side of $\pch'$ and $\pch'\not=\pch$. In this case, $S$ is a vertical side of $\pch'$.
\end{lemma}

\begin{prop} \label{intersectionprecells}
Let $\pch_1, \pch_2$ be two precells in $\h$ and let $g_1,g_2 \in\Gamma$. Suppose that  $g_1.\pch_1 \cap g_2.\pch_2 \not=\emptyset$. Then we have either $g_1.\pch_1 = g_2.\pch_2$ and $g_1g_2^{-1} \in \Gamma_\infty$, or $g_1.\pch_1 \cap g_2.\pch_2$ is a common side of $g_1.\pch_1$ and $g_2.\pch_2$, or $g_1.\pch_1\cap g_2.\pch_2$ is a point which is the endpoint of some side of $g_1.\pch_1$ and some side of $g_2.\pch_2$. If $S$ is a common side of $g_1.\pch_1$ and $g_2.\pch_2$, then $g_1^{-1}.S$ is a vertical side of $\pch_1$ if and only if $g_2^{-1}.S$ is a vertical side of $\pch_2$.
\end{prop}

\begin{proof}
W.l.o.g.\@ $g_1=\id$. Let $\fpch$ be a basal family of precells in $\h$. Corollary~\ref{props_preH} shows that we may assume that $\pch_1\in\fpch$. Let $S_1=[a_1,\infty]$ and $S_2=[a_2,\infty]$ be the vertical sides of $\pch_1$. If $\pch_1$ has non-vertical sides, let these be $S_3=[b_1,b_2]$ resp.\@ $S_3=[b_1,b_2]$ and $S_4=[b_3,b_4]$. Lemma~\ref{intervert} shows that we find precells $\pch'_1$ and $\pch'_2$ such that $\pch'_1\not=\pch_1\not=\pch'_2$ and $S_1$ is a vertical side of $\pch'_1$ and $S_2$ is a vertical side of $\pch'_2$. Proposition~\ref{structureglue} shows that there exist $(h_3,\pch'_3)\in \Gamma\times\fpch$ resp.\@ $(h_3,\pch'_3), (h_4,\pch'_4) \in\Gamma\times\fpch$ such that $h_3.S_3$ is a non-vertical side of $\pch'_3$ and $h_4.S_4$ is a non-vertical side of $\pch'_4$ and $h_3.\pch_1\not=\pch'_3$ and $h_4.\pch_1\not=\pch'_4$. Recall that each precell in $\h$ is a convex polyhedron. Therefore $\pch_1\cup \pch'_1$ is a polyhedron with $(a_1,\infty)$ in its interior, and $\pch_1\cup h_3^{-1}.
\pch'_3$ is a polyhedron with $(b_1,b_2)$ in its interior. Likewise for $\pch_1\cup \pch'_2$ and $\pch_1\cup h_4^{-1}.\pch'_4$. 

Suppose first that $\pch_1^\circ\cap g_2.\pch_2\not=\emptyset$. By Corollary~\ref{props_preH} there exists a pair $(h,\pch') \in\Gamma_\infty\times\fpch$ such that $\pch_2=h.\pch'$. Then $\pch_1^\circ \cap g_2h.\pch' \not=\emptyset$. Since $\pch'$ is a convex polyhedron, $\pch_1^\circ \cap g_2h .(\pch')^\circ \not=\emptyset$. Now $\bigcup\big\{ (\wh\pch\,)^\circ \ \big\vert\  \wh\pch\in\fpch\big\}$ is a fundamental region for $\Gamma$ in $\h$ (see Theorem~\ref{precellsH}). Therefore, $g_2h=\id$ and, by Proposition~\ref{precellsum}\eqref{boundarypreH}, $\pch_1=\pch'$. Hence $g_2^{-1}\in\Gamma_\infty$ and $\pch_1 = g_2.\pch_2$.

Suppose now that $\pch_1^\circ \cap g_2.\pch_2 = \emptyset$. If $g_2.\pch_2\cap (a_1,\infty) \not=\emptyset$, then $g_2.\pch_2 \cap (\pch'_1)^\circ \not=\emptyset$ and the argument from above shows that $\pch'_1 = g_2.\pch_2$ and $g_2\in\Gamma_\infty$. From this it follows that $S_1$ is a vertical side of $\pch_2$. 

If $g_2.\pch_2 \cap (b_1,b_2) \not=\emptyset$, then $g_2.\pch_2 \cap h_3^{-1}.( \pch'_3)^\circ \not= \emptyset$. As before, $g_2h_3\in\Gamma_\infty$ and $g_2.\pch_2 = h_3^{-1}.\pch'_3$. Then $S_3$ is a non-vertical side of $\pch_2$. The argumentation for $g_2.\pch_2 \cap (a_2,\infty) \not=\emptyset$ and $g_2.\pch_2\cap (b_3,b_4) \not=\emptyset$ is analogous. 

It remains the case that $g_2.\pch_2$ intersects $\pch_1$ is an endpoint $v$ of some side of $\pch_1$. By symmetry of arguments, $v$ is an endpoint of some side of $\pch_2$. This completes the proof.
\end{proof}

We call a family $\{ S_j \mid j\in J\}$ of polyhedrons in $\h$ a \textit{tesselation} of $\h$ if
\begin{enumerate}[(T1)]
\item\label{T1} $\h= \bigcup_{j\in J}S_j$, and
\item\label{T2} if $S_j\cap S_k\not=\emptyset$ for some $j,k\in J$, then either $S_j=S_k$ or $S_j\cap S_k$ is a common side or vertex of $S_j$ and $S_k$.
\end{enumerate}

\begin{cor}\label{precellsHtess}
Let $\fpch$ be a basal family of precells in $\h$. Then 
\[
\Gamma.\fpch = \{ g.\pch \mid g\in\Gamma,\ \pch\in\fpch\}
\]
is a tesselation of $\h$ which satisfies in addition the property that if $g_1.\pch_1 = g_2.\pch_2$, then $g_1=g_2$ and $\pch_1=\pch_2$.
\end{cor}

\begin{proof}
Let $\fd\sceq \bigcup\{\pch\mid \pch\in\fpch\}$. Theorem~\ref{precellsH} states that $\fd$ is a closed fundamental region for $\Gamma$ in $\h$, hence $\bigcup_{g\in\Gamma}g.\fd=\h$. This proves (\apref{T}{T1}{}). Property (\apref{T}{T2}{}) follows directly from Proposition~\ref{intersectionprecells}. Now let $(g_1,\pch_1),(g_2,\pch_2)\in\Gamma\times\fpch$ with $g_1.\pch_1=g_2.\pch_2$. Then $g_1.\pch_1^\circ = g_2.\pch_2^\circ$. Recalling that $\fd^\circ$ is a fundamental region for $\Gamma$ in $\h$ and that $\pch_1^\circ,\pch_2^\circ\subseteq \fd^\circ$, we get that $g_1=g_2$ and $\pch_1=\pch_2$.
\end{proof}

\section{Cells in $\h$}\label{sec_cellsH}

Let $\Gamma$ be a geometrically finite subgroup of $\PSL_2(\R)$ of which $\infty$ is a cuspidal point and which satisfies \eqref{A4}. Suppose that the set of relevant isometric spheres is non-empty. Let $\fpch$ be a basal family of precells in $\h$. To each basal precell in $\h$ we assign a cell in $\h$, which is an essentially disjoint union of certain $\Gamma$-translates of certain basal precells. More precisely, using Proposition~\ref{structureglue} we define so-called cycles in $\fpch\times\Gamma$ in the same way as the cycles in the Poincar\'e Fundamental Polyhedron Theorem are defined (see e.g.\@ \cite{Maskit}). These are certain finite sequences of pairs $(\pch,h)\in\fpch\times\Gamma\mminus\Gamma_\infty$ such that each cycle is determined up to cyclic permutation by any pair which belongs to it. Moreover, if $(\pch, h_\pch)$ is an element of some cycle, then $h_\pch$ is an element in $\Gamma\mminus\Gamma_\infty$ assigned to $\pch$ by Proposition~\ref{structureglue} (or $h_\pch=\id$ if $\pch$ is a 
strip 
precell). Conversely, if $h_\pch$ is an element assigned to $\pch$ by Proposition~\ref{structureglue}, then $(\pch, h_\pch)$ determines a cycle in $\fpch\times\Gamma\mminus\Gamma_\infty$. 

One of the crucial properties of each cell in $\h$ is that it is a convex polyhedron with non-empty interior of which each side is a complete geodesic segment. This fact is mainly due to the condition \eqref{A4} of $\Gamma$. The other two important properties of cells in $\h$ are that each non-vertical side of a cell is a $\Gamma$-translate of some vertical side of some cell in $\h$ and that the family of $\Gamma$-translates of all cells in $\h$ is a tesselation of $\h$.

\subsection{Cycles in $\mathbb A\times\Gamma$}

Let $\pch\in\fpch$ be a non-cuspidal precell in $\h$. The definition of precells shows that $\pch$ is attached to a unique (inner) vertex $v$ of $\mathcal K$, and $\pch$ is the unique precell attached to $v$. Therefore we set $\pch(v) \sceq \pch$. Further, $\pch$ has two non-vertical sides $b_1$ and $b_2$. Let $\{k_1(\pch),k_2(\pch)\}$ be the two elements in $\Gamma\mminus\Gamma_\infty$ given by Proposition~\ref{structureglue} such that $b_j\in I(k_j(\pch))$ and $k_j(\pch).b_j$ is a non-vertical side of some basal precell. 
Necessarily, the isometric spheres $I(k_1(\pch))$ and $I(k_2(\pch))$ are different, therefore $k_1(\pch) \not= k_2(\pch)$. The set $\{ k_1(\pch), k_2(\pch)\}$ is uniquely determined by Proposition~\ref{structureglue}, the assignment $\pch\mapsto k_1(\pch)$ clearly depends on the enumeration of the non-vertical sides of $\pch$. Now $w\sceq k_j(\pch).v$ is an inner vertex. Let $\pch(w)$ be the (unique non-cuspidal) basal precell attached to $w$. Since one non-vertical side of $\pch(w)$ is $k_j(\pch).b_j$, which is contained in the relevant isometric sphere $I(k_j(\pch)^{-1})$, and $k_j(\pch)^{-1}k_j(\pch).b_j=b_j$ is a non-vertical side of some basal precell, namely of $\pch$, one of the elements in $\Gamma\mminus\Gamma_\infty$ assigned to $\pch(w)$ by Proposition~\ref{structureglue} is $k_j(\pch)^{-1}$.

\begin{construction}\label{elemsequences}
Let $\pch\in\fpch$ be a non-cuspidal precell and suppose that $\pch=\pch(v)$ is attached to the vertex $v$ of $\mathcal K$. We assign to $\pch$ two sequences $(h_j)$ of elements in $\Gamma\mminus\Gamma_\infty$ using the following algorithm:
\begin{itemize}
\item[(step 1)] Let $v_1 \sceq v$ and let $h_1$ be either $k_1(\pch)$ or $k_2(\pch)$. Set $g_1\sceq \id$, $g_2\sceq h_1$ and carry out (step 2).
\item[(step j)] Set $v_j \sceq g_j.v$ and $\pch_j\sceq \pch(v_j)$. Let $h_j$ be the element in $\Gamma\mminus\Gamma_\infty$ such that $\big\{ h_j, h_{j-1}^{-1}\big\} = \big\{ k_1(\pch_j), k_2(\pch_j)\big\}$. Set $g_{j+1}\sceq h_jg_{j}$.
If $g_{j+1} =\id$, then the algorithm stops. If $g_{j+1} \not=\id$, then carry out (step $j+1$).
\end{itemize}
\end{construction}

\begin{example}\label{exsequences}
Recall the Hecke triangle group $G_n$ and its basal family $\fpch=\{\pch\}$ of precells in $\h$ from Example~\ref{exprecells1}. 
The two sequences assigned to $\pch$ are $(U_n)_{j=1}^n$ and $\big(U_n^{-1}\big)_{j=1}^n$.
\end{example}

The statements of Propositions~\ref{propsequences}-\ref{stripseq} follow as in the Poincar\'e Fundamental Polyhedron Theorem, using Lemma~\ref{isoprops} and elementary convex geometry. For this reason we omit the detailed proofs.

\begin{prop}\label{propsequences}
Let $\pch=\pch(v)$ be a non-cuspidal basal precell. 
\begin{enumerate}[{\rm (i)}]
\item The sequences from Construction~\ref{elemsequences} are finite. In other words, the algorithm for the construction of the sequences always terminates.
\item Both sequences have same length, say $k\in \N$.
\item\label{psiii} Let $(a_j)_{j=1,\ldots, k}$ and $(b_j)_{j=1,\ldots, k}$ be the two sequences assigned to $\pch$. Then they are inverse to each other in the following sense: For each $j=1,\ldots, k$ we have $a_j = b_{k-j+1}^{-1}$.
\item\label{psiv} For $j=1,\ldots,k$ set $c_{j+1}\sceq a_ja_{j-1}\cdots a_2a_1$, $d_{j+1}\sceq b_jb_{j-1}\cdots b_2b_1$ and $c_1 \sceq \id \seqc d_1$. Then 
\[
\ch \sceq \bigcup_{j=1}^k c_j^{-1}.\pch\big(c_jv\big) = \bigcup_{j=1}^k d_j^{-1}.\pch\big(d_jv\big).
\]
Further, both unions are essentially disjoint, and $\ch$ is the polyhedron with the (pairwise distinct) vertices (in this order) $c_1^{-1}.\infty$, $c_2^{-1}.\infty$, $\ldots$, $c_k^{-1}.\infty$ resp.\@ $d_1^{-1}.\infty$, $d_2^{-1}.\infty$, $\ldots$, $d_k^{-1}.\infty$.
\end{enumerate}
\end{prop}

\begin{defi}
Let $\pch\in\fpch$ be a non-cuspidal precell and suppose that $\pch$ is attached to the vertex $v$ of $\mc K$. Let $h_\pch$ be one of the elements in $\Gamma\mminus\Gamma_\infty$ assigned to $\pch$ by Proposition~\ref{structureglue}. Let $(h_j)_{j=1,\ldots,k}$ be the sequence assigned to $\pch$ by Construction~\ref{elemsequences} with $h_1=h_\pch$. For $j=1,\ldots,k$ set $g_1\sceq \id$ and $g_{j+1}\sceq h_jg_j$. Then the (finite) sequence $\big( (\pch(g_jv), h_j) \big)_{j=1,\ldots,k}$ is called the \textit{cycle in $\fpch\times\Gamma$ determined by $(\pch,h_\pch)$}.

Let $\pch\in\fpch$ be a cuspidal precell. Suppose that $b$ is the non-vertical side of $\pch$ and let $h_\pch$ be the element in $\Gamma\mminus\Gamma_\infty$ assigned to $\pch$ by Proposition~\ref{structureglue}. Let $\pch'$ be the (cuspidal) basal precell with non-vertical side $h_\pch. b$. Then the (finite) sequence $\big( (\pch,h_\pch), (\pch',h_\pch^{-1}) \big)$ is called the \textit{cycle in $\fpch\times\Gamma$ determined by $(\pch,h_\pch)$}. 

Let $\pch\in\fpch$ be a strip precell. Set $h_\pch\sceq\id$. Then $\big( (\pch,h_\pch)\big)$ is called the \textit{cycle in $\fpch\times\Gamma$ determined by $(\pch,h_\pch)$}.
\end{defi}

\begin{example}\label{excycles}
\begin{enumerate}[{\rm (i)}]
\item Recall Example~\ref{exsequences}. The cycle in $\fpch\times G_n$ determined by $(\pch, U_n)$ is $\big( (\pch,U_n) \big)_{j=1}^n$.
\item Recall the group $\PGamma_0(5)$ and the basal family $\fpch = \{ \pch(v_0),\ldots, \pch(v_4)\}$ from Example~\ref{exprecells2}. The element in $\PGamma_0(5)\mminus \PGamma_0(5)_\infty$ assigned to $\pch(v_0)$ is $h\sceq \textmat{4}{-1}{5}{-1}$. The cycle in $\fpch\times \PGamma_0(5)$ determined by $(\pch(v_0), h)$ is 
\[
\big( (\pch(v_0), h), (\pch(v_4), h^{-1})\big).
\]
Further let $h_1\sceq h$, $h_2\sceq\textmat{3}{-2}{5}{-3}$ and $h_3\sceq\textmat{2}{-1}{5}{-2}$. The cycle in $\fpch\times\PGamma_0(5)$ determined by $(\pch(v_1),h_1)$ is 
\[
\big( (\pch(v_1),h_1), (\pch(v_3),h_2), (\pch(v_2),h_3)\big).
\]
\item Recall the group $\Gamma$ and the basal family $\fpch=\{\pch_1,\pch_2,\pch_3\}$  from Example~\ref{exprecells3}. The cycle in $\fpch\times\Gamma$ determined by $(\pch_2,S)$ is $\big( (\pch_2,S), (\pch_3,S)\big)$.
\end{enumerate}
\end{example}

\begin{prop}\label{cycleshift}
Let $\pch\in\fpch$ be a non-cuspidal precell in $\h$ and suppose that $h_\pch$ is one of the elements in $\Gamma\mminus\Gamma_\infty$ assigned to $\pch$ by Proposition~\ref{structureglue}. Let $\big( (\pch_j,h_j) \big)_{j=1,\ldots, k}$ be the cycle in $\fpch\times\Gamma$ determined by $(\pch,h_\pch)$. Let $j\in\{1,\ldots, k\}$ and define the sequence $\big( (\pch'_l, a_l) \big)_{l=1,\ldots, k}$ by 
\begin{align*}
a_l & \sceq
\begin{cases}
h_{l+j-1} & \text{for $l=1,\ldots, k-j+1$,}
\\
h_{l+j-1-k} & \text{for $l=k-j+2,\ldots,k$,}
\end{cases}
\intertext{and}
\pch'_l & \sceq
\begin{cases}
\pch_{l+j-1} & \text{for $l=1,\ldots, k-j+1$,}
\\
\pch_{l+j-1-k} & \text{for $l=k-j+2,\ldots,k$.}
\end{cases}
\end{align*}
Then $a_1=h_j$ is one of the elements in $\Gamma\mminus\Gamma_\infty$ assigned to $\pch_j$ by Proposition~\ref{structureglue} and $\big( (\pch'_l,a_l) \big)_{l=1,\ldots, k}$ is the cycle in $\fpch\times\Gamma$ determined by $(\pch_j,h_j)$.
\end{prop}

\begin{prop}\label{stripseq}
Let $\pch\in\fpch$ be a cuspidal precell in $\h$ and let $h_\pch$ be the element in $\Gamma\mminus\Gamma_\infty$ assigned to $\pch$ by Proposition~\ref{structureglue}. Let $\big( (\pch,h_\pch), (\pch', h_\pch^{-1}) \big)$ be the cycle in $\fpch\times\Gamma$ determined by $(\pch,h_\pch)$. Then $h_\pch^{-1}$ is the element in $\Gamma\mminus\Gamma_\infty$ assigned to $\pch'$ by Proposition~\ref{structureglue} and $\big( (\pch', h_\pch^{-1}), (\pch, h_\pch) \big)$ is the cycle in $\fpch\times\Gamma$ determined by $(\pch',h_\pch^{-1})$.
\end{prop}

\subsection{Cells in $\h$ and their properties}

Now we can define a cell in $\h$ for each basal precell in $\h$.

\begin{construction}\label{cellsH}
Let $\pch$ be a basal strip precell in $\h$. Then we set
\[
\ch(\pch) \sceq \pch.
\]
Let $\pch$ be a cuspidal basal precell in $\h$. Suppose that $g$ is the element in $\Gamma\mminus\Gamma_\infty$ assigned to $\pch$ by Proposition~\ref{structureglue} and let $\big( (\pch, g), (\pch',g^{-1})\big)$ be the cycle in $\fpch\times\Gamma$ determined by $(\pch, g)$.  Define
\[
\ch(\pch) \sceq \pch \cup g^{-1}.\pch'.
\]
The set $\ch(\pch)$ is  well-defined because $g$ is uniquely determined.

Let $\pch$ be a non-cuspidal basal precell in $\h$ and fix an element $h_\pch$ in $\Gamma\mminus\Gamma_\infty$ assigned to $\pch$ by Proposition~\ref{structureglue}. Let $\big( (\pch_j,h_j)\big)_{j=1,\ldots, k}$ be the cycle in $\fpch\times\Gamma$ determined by $(\pch, h_\pch)$. For $j=1,\ldots, k$ set $g_1\sceq \id$ and $g_{j+1}\sceq h_jg_j$. Set
\[
\ch(\pch) \sceq \bigcup_{j=1}^k g_j^{-1}.\pch_j.
\]
By Proposition~\ref{propsequences}, the set $\ch(\pch)$ does not depend on the choice of $h_\pch$.
The family $\fch \sceq \{ \ch(\pch) \mid \pch\in\fpch\}$ is called the \textit{family of cells in $\h$ assigned to $\fpch$}. Each element of $\fch$ is called a \textit{cell in $\h$}.

Note that the family $\fch$ of cells in $\h$ depends on the choice of $\fpch$. If we need to distinguish cells in $\h$ assigned to the basal family $\fpch$ of precells in $\h$ from those assigned to the basal family $\fpch'$ of precells in $\h$, we will call the first ones \textit{$\fpch$-cells} and the latter ones \textit{$\fpch'$-cells}.
\end{construction}

\begin{example} Recall the Example~\ref{excycles}.
For the Hecke triangle group $G_5$, Figure~\ref{cellHecke} shows the cell assigned to the family $\fpch=\{\pch\}$ from Example~\ref{exprecells1}.
\begin{figure}[h]
\begin{center}
\includegraphics*{Hecke5.11} 
\end{center}
\caption{The cell $\ch(\pch)$ for $G_5$.}\label{cellHecke}
\end{figure}
For the group $\PGamma_0(5)$, the family of cells in $\h$ assigned to $\fpch$ is indicated in Figure~\ref{cellGamma05}. 
\begin{figure}[h]
\begin{center}
\includegraphics*{Gamma05.9} 
\end{center}
\caption{The family of cells in $\h$ assigned to $\fpch$ for $\PGamma_0(5)$.}\label{cellGamma05}
\end{figure}
Figure~\ref{cellNichtcof} shows the family of cells in $\h$ assigned to the basal family $\fpch$ of precells of $\Gamma$.
\begin{figure}[h]
\begin{center}
\includegraphics*{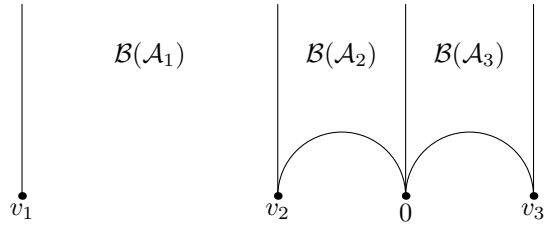} 
\end{center}
\caption{The family of cells in $\h$ assigned to $\fpch$ for $\Gamma$.}\label{cellNichtcof}
\end{figure}
\end{example}

In the series of the following six propositions we investigate the structure of cells and their relations to each other. This will allow to show that the family of $\Gamma$-translates of cells in $\h$ is a tesselation of $\h$, and it will be of interest for the labeling of the cross section.

\begin{prop} \label{nccells}
Let $\pch$ be a non-cuspidal basal precell in $\h$. Suppose that $h_\pch$ is an element in $\Gamma\mminus\Gamma_\infty$ assigned to $\pch$ by Proposition~\ref{structureglue} and let $\big( (\pch_j,h_j)\big)_{j=1,\ldots, k}$ be the cycle in $\fpch\times \Gamma$ determined by $(\pch, h_\pch)$. For $j=1,\ldots, k$ set $g_1\sceq\id$ and $g_{j+1}\sceq h_jg_j$. Then the following assertions hold true.
\begin{enumerate}[{\rm (i)}]
\item\label{ncc1} The set $\ch(\pch)$ is the convex polyhedron with vertices (in this order) $g_1^{-1}.\infty$, $g_2^{-1}.\infty$,$\ldots$, $g_k^{-1}.\infty$.
\item\label{ncc2} The boundary of $\ch(\pch)$ consists precisely of the union of the images of the vertical sides of $\pch_j$ under $g_j^{-1}$, $j=1,\ldots,k$. More precisely, if $s_j$ denotes the summit of $I(h_j)$ for $j=1,\ldots, k$, then 
\[
\partial \ch(\pch) = \bigcup_{j=1}^k g_j^{-1}.[s_j,\infty] \cup \bigcup_{j=2}^{k+1} g_j^{-1}.[h_{j-1}.s_{j-1},\infty].
\]
\item\label{ncc3} For each $j=1,\ldots, k$ we have $g_j. \ch(\pch) = \ch(\pch_j)$. In particular, the side $[g_j^{-1}.\infty, g_{j+1}^{-1}.\infty]$ of $\ch(\pch)$ is the image of the vertical side $[\infty, h_j^{-1}.\infty]$ of $\ch(\pch_j)$ under $g_j^{-1}$.
\item\label{ncc4} Let $\widehat \pch$ be a basal precell in $\h$ and let $h\in\Gamma$ such that $h.\widehat \pch \cap \ch(\pch)^\circ \not=\emptyset$. Then there exists a unique $j\in\{1,\ldots, k\}$ such that $h=g_j^{-1}$ and $\widehat \pch = \pch_j$. In particular, $\widehat \pch$ is non-cuspidal and $h.\ch(\widehat \pch) = \ch(\pch)$.
\end{enumerate}
\end{prop}

\begin{proof}
By Proposition~\ref{propsequences}, $\ch(\pch)$ is the polyhedron with vertices (in this order) $g_1^{-1}.\infty$, $g_2^{-1}.\infty,\ldots, g_k^{-1}.\infty$. Since each of its sides is a complete geodesic segment, $\ch(\pch)$ is convex. This shows \eqref{ncc1}. The statement \eqref{ncc2} follows from the proof of Proposition~\ref{propsequences}.

To prove \eqref{ncc3}, fix $j\in\{1,\ldots, k\}$ and recall from Proposition~\ref{cycleshift} the cycle $\big( (\pch'_l,a_l)\big)_{l=1,\ldots,k}$ in $\fpch\times\Gamma$ determined by $(\pch_j,h_j)$.  For $l=1,\ldots, k$ set $c_1\sceq\id$ and $c_{l+1}\sceq a_lc_l$. Then 
\[
c_l = 
\begin{cases}
g_{l+j-1}g_j^{-1} & \text{for $l=1,\ldots, k-j+1$}
\\
g_{l+j-k-1}g_j^{-1} & \text{for $l=k-j+2,\ldots, k$.}
\end{cases}
\]
Hence
\begin{align*}
\ch(\pch_j) & = \bigcup_{l=1}^k c_l^{-1}. \pch'_l = \bigcup_{l=1}^k g_j (c_lg_j)^{-1}. \pch'_l 
\\
& = g_j. \bigcup_{l=1}^{k-j+1} g_{l+j-1}^{-1}. \pch_{l+j-1} \cup g_j. \bigcup_{l=k-j+2}^k g_{l+j-k-1}^{-1}. \pch_{l+j-k-1}
\\
& = g_j. \bigcup_{l=1}^k g_l^{-1}.\pch_l = g_j. \ch(\pch).
\end{align*}
This immediately implies that the side $[g_j^{-1}.\infty, g_{j+1}^{-1}.\infty]$ of $\ch(\pch)$ maps to the side $g_j.[g_j^{-1}.\infty, g_{j+1}^{-1}.\infty] = [\infty, h_j^{-1}.\infty]$ of $\ch(\pch_j)$, which is vertical.

To prove \eqref{ncc4}, fix $z\in h.\widehat \pch\cap \ch(\pch)^\circ$. Then there exists $l\in\{1,\ldots, k\}$ such that $z\in h.\widehat \pch \cap g_l^{-1}.\pch_l$. Let $b\sceq h.\widehat \pch \cap g_l^{-1}.\pch_l$. By Proposition~\ref{intersectionprecells} there are three possibilities for $b$. 

If $b=h.\widehat \pch = g_l^{-1}.\pch_l$, then $g_lh.\widehat \pch = \pch_l$. Since $\widehat \pch$ and $\pch_l$ are both basal, it follows that $h=g_l^{-1}$ and $\widehat \pch = \pch_l$.

Suppose that $v$ is the vertex of $\mc K$ to which $\pch$ is attached. If $b$ is a common side of $h.\widehat \pch$ and $g_l^{-1}.\pch_l$, then, since $z\in \ch(\pch)^\circ$, $g_l.b$ must be a non-vertical side of $\pch_l$ (see \eqref{ncc2}). This implies that $v\in b$. In turn, there is a neighborhood $U$ of $v$ such that $U\subseteq \ch(\pch)$ and $U\cap h.(\widehat \pch)^\circ \not=\emptyset$. Hence, $h.(\widehat \pch)^\circ \cap \ch(\pch)^\circ \not=\emptyset$. Thus there exists $j\in \{1,\ldots, k\}$ such that $h.(\widehat \pch)^\circ \cap g_j^{-1}.\pch_j \not=\emptyset$. Proposition~\ref{intersectionprecells} implies that $\widehat \pch = \pch_j$ and $h=g_j^{-1}$.

If $b$ is a point, then $b=z$ must be the endpoint of some side of $g_l^{-1}.\pch_l$. From $z\in \ch(\pch)^\circ$ it follows that $z=v$. Now the previous argument applies. 

To show the uniqueness of $j\in\{1,\ldots,k\}$ with $\wh\pch=\pch_j$ and $h=g_j^{-1}$, suppose that there is $p\in\{1,\ldots,k\}$ with $\wh\pch=\pch_p$ and $h=g_p^{-1}$. Then $g_j=g_p$. By Proposition~\ref{propsequences}\eqref{psiv} $j=p$. The remaining parts of \eqref{ncc4} follow from \eqref{ncc3}.
\end{proof}

\begin{prop} \label{ccells}
Let $\pch$ be a cuspidal basal precell in $\h$ which is attached to the vertex $v$ of $\mc K$. Suppose that  $g$ is the element in $\Gamma\mminus\Gamma_\infty$ assigned to $\pch$ by Proposition~\ref{structureglue}. Let $\big( (\pch,g), (\pch',g^{-1})\big)$ be the cycle in $\fpch\times\Gamma$ determined by $(\pch,g)$. Then we have the following properties.
\begin{enumerate}[{\rm (i)}]
\item\label{cc1} The set $\ch(\pch)$ is the hyperbolic triangle with vertices $v$, $g^{-1}.\infty$, $\infty$.
\item\label{cc2} The boundary of $\ch(\pch)$ is the union of the vertical sides of $\pch$ with the images of the vertical sides of $\pch'$ under $g^{-1}$.
\item\label{cc3} The sets $g.\ch(\pch)$ and $\ch(\pch')$ coincide. In particular, the non-vertical side $[v,g^{-1}.\infty]$ of $\ch(\pch)$ is the image of the vertical side $[g.v,\infty]$ of $\pch'$ under $g^{-1}$.
\item\label{cc4} Suppose that $\widehat \pch$ is a basal precell in $\h$ and $h\in\Gamma$ such that $h.\widehat \pch\cap \ch(\pch)^\circ\not=\emptyset$. Then either $h=\id$ and $\widehat \pch=\pch$ or $h=g^{-1}$ and $\widehat \pch = \pch'$. In particular, $\widehat \pch$ is cuspidal and $h.\ch(\widehat \pch) = \ch(\pch)$.
\end{enumerate}
\end{prop}

\begin{prop} \label{stripcells}
Let $\pch$ be a basal strip precell in $\h$. Let $\widehat \pch$ be a basal precell in $\h$ and $h\in\Gamma$ such that $h.\widehat \pch \cap \ch(\pch)^\circ \not=\emptyset$. Then $h=\id$ and $\widehat \pch = \pch$.  
\end{prop}

\begin{cor}\label{ABbij}
The map $\fpch \to \fch$, $\pch\mapsto \ch(\pch)$ is a bijection.
\end{cor}

The following proposition is implied by following through the glueing procedure from precells to cells. The details are easily seen by straightforward deductions.

\begin{prop} \label{glues}
\begin{enumerate}[{\rm (1)}]
\item \label{glue_nccells}
Let $\pch$ be a non-cuspidal basal precell in $\h$. Suppose that $(h_j)_{j=1,\ldots, k}$ is a sequence in $\Gamma\mminus\Gamma_\infty$ assigned to $\pch$ by Construction~\ref{elemsequences}. For $j=1,\ldots, k$ set $g_1\sceq\id$ and $g_{j+1} \sceq h_jg_j$. Let $\pch'$ be a basal precell in $\h$ and $g\in\Gamma$ such that $\ch(\pch)\cap g.\ch(\pch') \not=\emptyset$. Then we have the following properties. 
\begin{enumerate}[{\rm (i)}]
\item Either $\ch(\pch) = g.\ch(\pch')$, or $\ch(\pch)\cap g.\ch(\pch')$ is a common side of $\ch(\pch)$ and $g.\ch(\pch')$.
\item If $\ch(\pch) = g.\ch(\pch')$, then $g=g_j^{-1}$ for a unique $j\in \{1,\ldots, k\}$. In particular, $\pch'$ is non-cuspidal.
\item If $\ch(\pch) \not= g.\ch(\pch')$, then $\pch'$ is cuspidal or non-cuspidal. If $\pch'$ is cuspidal and $k\in\Gamma\mminus\Gamma_\infty$ is the element assigned to $\pch'$ by Proposition~\ref{structureglue}, then we have  $\ch(\pch) \cap g.\ch(\pch') = g.[k^{-1}.\infty,\infty]$.
\end{enumerate}
\item  \label{glue_ccells}
Let $\pch$ be a cuspidal basal precell in $\h$ which is attached to the vertex $v$ of $\mc K$. Suppose that $h\in\Gamma\mminus\Gamma_\infty$ is the element assigned to $\pch$ by Proposition~\ref{structureglue}. Let $\pch'$ be a basal precell in $\h$ and $g\in\Gamma$ such that we have $\ch(\pch)\cap g.\ch(\pch') \not=\emptyset$. Then the following assertions hold true.
\begin{enumerate}[{\rm (i)}]
\item Either $\ch(\pch) = g.\ch(\pch')$, or $\ch(\pch) \cap g.\ch(\pch')$ is a common side of $\ch(\pch)$ and $g.\ch(\pch')$.
\item If $\ch(\pch) = g.\ch(\pch')$, then either $g=\id$ or $g=h^{-1}$. In particular, $\pch'$ is cuspidal.
\item If $\ch(\pch)\not= g.\ch(\pch')$, then $\pch'$ is cuspidal or non-cuspidal or a strip precell. If $\pch'$ is a strip precell, then $[h^{-1}.\infty,\infty] \not= \ch(\pch) \cap g.\ch(\pch')$. If $\pch'$ is a cuspidal precell attached to the vertex $w$ of $\mc K$ and $k\in\Gamma\mminus\Gamma_\infty$ is the element assigned to $\pch'$ by Proposition~\ref{structureglue}, then $\ch(\pch)\cap g.\ch(\pch')$ is either $[v,\infty] = g.[w,\infty]$ or $[v,\infty] = g.[w,k^{-1}.\infty]$ or $[v,h^{-1}.\infty] = g.[w,\infty]$ or $[v,h^{-1}.\infty] = g.[w,k^{-1}.\infty]$ or $[h^{-1}.\infty,\infty] = g.[k^{-1}.\infty,\infty]$. If $\pch'$ is a non-cuspidal precell, then we have $\ch(\pch)\cap g.\ch(\pch') = [h^{-1}.\infty,\infty]$.
\end{enumerate}
\item \label{glue_stripcells}
Let $\pch$ be a basal strip precell in $\h$. Let $\pch'$ be a basal precell in $\h$ and $g\in\Gamma$ such that $\ch(\pch) \cap g.\ch(\pch') \not=\emptyset$. Then the following statements hold.
\begin{enumerate}[{\rm (i)}]
\item Either $\ch(\pch) = g.\ch(\pch')$, or $\ch(\pch)\cap g.\ch(\pch')$ is a common side of $\ch(\pch)$ and $g.\ch(\pch')$.
\item If $\ch(\pch) = g.\ch(\pch')$, then $g=\id$ and $\pch=\pch'$.
\item If $\ch(\pch) \not= g.\ch(\pch')$, then $\pch'$ is a cuspidal or strip precell. If $\pch'$ is cuspidal and $k\in\Gamma\mminus\Gamma_\infty$ is the element assigned to $\pch'$ by Proposition~\ref{structureglue}, then we have $\ch(\pch) \cap g.\ch(\pch') \not= g.[k^{-1}.\infty,\infty]$.
\end{enumerate}
\end{enumerate}
\end{prop}

\begin{cor}\label{cellsHtess}
The family of $\Gamma$-translates of all cells provides a tesselation of $\h$. In particular, if $\ch$ is a cell in $\h$ and $S$ a side of $\ch$, then there exists a pair $(\ch',g) \in \fch \times \Gamma$ such that $S= \ch \cap g.\ch'$. Moreover, $(\ch',g)$ can be chosen such that $g^{-1}.S$ is a vertical side of $\ch'$. 
\end{cor}

\section{The base manifold of the cross sections}\label{sec_base}

Let $\Gamma$ be a geometrically finite subgroup of $\PSL_2(\R)$ of which $\infty$ is a cuspidal point and which satisfies \eqref{A4}, and suppose that there are relevant isometric spheres. In this section we define a set $\wh\CS$ which will turn out, in Section~\ref{sec_geomcross} below, to be a cross section for the geodesic flow on $Y=\Gamma\backslash \h$ \wrt to certain measures $\mu$, which will be characterized in Section~\ref{sec_geomcross} below. Here we will already see that $\wh\CS$ satisfies (\apref{C}{C2}{}) by showing that $\pr(\wh\CS)$ is a totally geodesic suborbifold of $Y$ of codimension one and that $\wh\CS$ is the set of unit tangent vectors based on $\pr(\wh\CS)$ but not tangent to it. To achieve this, we start at the other end. We fix a basal family $\fpch$ of precells in $\h$ and consider the family $\fch$ of cells in $\h$ assigned to $\fpch$. We define $\BS(\fch)$ to be the set of $\Gamma$-translates of the sides of these cells. Then we show that the set $\BS\sceq \BS(\fch)$ is in fact 
independent of the choice of $\fpch$. We proceed to prove that $\BS$ is a totally geodesic submanifold of $\h$ of codimension one and define $\CS$ to be the set of unit tangent vectors based on $\BS$ but not tangent to it. Then $\wh\CS\sceq \pi(\CS)$ is our (future) geometric cross section and $\pr(\wh\CS) = \wh\BS\sceq \pi(\BS)$. This construction shows in particular that the set $\wh\CS$ does not depend on the choice of $\fpch$. For future purposes we already define the sets $\NC(\fch)$ and $\bd(\fch)$ and show that also these are independent of the choice of $\fpch$. Throughout let
\[
 t_\lambda = \bmat{1}{\lambda}{0}{1}
\]
with $\lambda>0$ be such that $\Gamma_\infty = \langle t_\lambda\rangle$.

\begin{defi}
Let $\fpch$ be a basal family of precells in $\h$ and let $\fch$ be the family of cells in $\h$ assigned to $\fpch$. For $\ch\in\fch$ let $\Sides(\ch)$ be the set of sides of $\ch$. Then set 
\[
 \Sides(\fch) \sceq\bigcup_{\ch\in\fch}\Sides(\ch)
\]
and define
\[
\BS(\fch) \sceq \bigcup \Gamma.\Sides(\fch) = \bigcup \{ g.S \mid g\in\Gamma,\ S\in\Sides(\fch)\}.
\]
For $\ch\in \fch$ define $\bd(\ch) \sceq \bhg\ch$ and let $\NC(\ch)$ be the set of geodesics on $Y$ which have a representative on $H$ both endpoints of which are contained in $\bd(\ch)$. Further set 
\begin{align*}
\bd(\fch) & \sceq \bigcup_{g\in\Gamma}\bigcup_{\ch\in\fch} g.\bd(\ch)
\intertext{and}
\NC(\fch) & \sceq \bigcup_{\ch\in\fch} \NC(\ch).
\end{align*}
\end{defi}

\begin{prop}\label{choiceinvariant}
Let $\fpch$ and $\fpch'$ be two basal families of precells in $\h$ and suppose that $\fch$ resp.\@ $\fch'$ are the families of corresponding cells in $\h$ assigned to $\fpch$ resp.\@ $\fpch'$. There exists a unique map $\fpch \to \Z$, $\pch\mapsto m(\pch)$ such that 
\[
\psi\colon \fpch \to \{\text{precells in $H$}\},\quad \pch  \mapsto  t_\lambda^{m(\pch)}.\pch 
\]
is a bijection from $\fpch$ to $\fpch'$. Then 
\[
\chi\colon \fch  \to  \fch',\quad \ch(\pch)  \mapsto  t_\lambda^{m(\pch)}.\ch(\pch)
\]
is a bijection as well. Further we have that $\BS(\fch) = \BS(\fch')$, $\NC(\fch) = \NC(\fch')$ and $\bd(\fch) = \bd(\fch')$.
\end{prop}

\begin{proof}
This follows from Corollary~\ref{props_preH}, Theorem~\ref{precellsH}, Corollary~\ref{ABbij}, Proposition~\ref{structureglue} and Proposition~\ref{propsequences}.
\end{proof}

We set 
\[
\BS \sceq \BS(\fch),\quad \bd\sceq \bd(\fch) \quad\text{and}\quad \NC \sceq \NC(\fch)
\]
for the family $\fch$ of cells in $\h$ assigned to an arbitrary family $\fpch$ of precells in $\h$. Proposition~\ref{choiceinvariant} shows that $\BS$, $\bd$ and $\NC$ are well-defined.

\begin{prop}\label{BStotgeod}
The set $\BS$ is a totally geodesic submanifold of $\h$ of codimension one.
\end{prop}

\begin{proof}
The set $\BS$ is  a disjoint countable union of complete geodesic segments, which is locally finite. 
\end{proof}

Let $\CS$ \label{def_CS2} denote the set of unit tangent vectors in $S\h$ that are based on $\BS$ but not tangent to $\BS$. Recall that $Y$ denotes  the orbifold $\Gamma\backslash \h$ and recall the canonical projections $\pi\colon \h\to Y$, $\pi\colon S\h\to SY$ from Section~\ref{sec_prelims}. Set $\wh\BS \sceq \pi(\BS)$ and $\wh\CS \sceq \pi(\CS)$.

\begin{prop}\label{gcs1}
The set $\wh\BS$ is a totally geodesic suborbifold of $Y$ of codimension one, $\wh\CS$ is the set of unit tangent vectors based on $\wh\BS$ but not tangent to $\wh\BS$ and $\wh\CS$ satisfies (\apref{C}{C2}{}).
\end{prop}

\begin{proof}
Since $\BS$ is $\Gamma$-invariant by definition, we see that $\BS=\pi^{-1}(\wh\BS)$. Therefore, $\wh\BS$ is a totally geodesic suborbifold of $Y$ of codimension one. Moreover, $\CS=\pi^{-1}(\wh\CS)$ and hence $\wh\CS$ is indeed the set of unit tangent vectors based on $\wh\BS$ but not tangent to $\wh\BS$. Finally, $\pr(\wh\CS) = \wh\BS$. Therefore, $\wh\CS$ satisfies (\apref{C}{C2}{}).
\end{proof}

\begin{remark}\label{outlook}
Let $\NIC$ be the set of geodesics on $Y$ of which at least one endpoint is contained in $\pi(\bd)$. Here, $\pi\colon \hg\to \Gamma\backslash\hg$ denotes the extension of the canonical projection $\h\to Y$ to $\hg$. In Section~\ref{sec_geomcross} we will show that $\wh\CS$ is a cross section for the geodesic flow on $Y$ \wrt any measure $\mu$ on the space of geodesics on $Y$ for which $\mu(\NIC) = 0$. 
\end{remark}

We end this section with a short explanation of the acronyms. Obviously, $\CS$ stands for ``cross section'' and $\BS$ for ``base of (cross) section''. Then $\bd$ is for ``boundary'' in sense of geodesic boundary, and $\bd(\ch)$ is the geodesic boundary of the cell $\ch$. Moreover, which will become more sense in Section~\ref{sec_geomcodseq} (see Remark~\ref{whynaming}), $\NC$ stands for ``not coded'' and $\NC(\ch)$ for ``not coded due to the cell $\ch$''. Finally, $\NIC$ is for ``not infinitely often coded''.

\section{Precells and cells in $S\h$}\label{sec_preSH}

Let $\Gamma$ be a geometrically finite subgroup of $\PSL_2(\R)$ which satisfies \eqref{A4}. Suppose that $\infty$ is a cuspidal point of $\Gamma$ and that the set of relevant isometric spheres is non-empty. In this section we define the precells and cells in $S\h$ and study their properties. The purpose of precells and cells in $S\h$ is to get very detailed information about the set $\wh\CS$ from Section~\ref{sec_base} and its relation to the geodesic flow on $Y$, see Section~\ref{sec_geomcross}.

\begin{defi}
Let $U$ be a subset of $\h$ and $z\in \overline{U}$. A unit tangent vector $v$ at $z$ is said to \textit{point into $U$} if the geodesic $\gamma_v$ determined by $v$ runs into $U$, \ie if there exists $\eps > 0$ such that $\gamma_v( (0,\eps) ) \subseteq U$. The unit tangent vector $v$ is said to \textit{point along the boundary of $U$} if there exists $\eps>0$ such that $\gamma_v( (0,\eps) ) \subseteq \partial U$. It is said to \textit{point out of $U$} if it points into $H\mminus U$. 
\end{defi}

\begin{defi}\label{def_preSH}
Let $\pch$ be a precell in $\h$. Define $\wt\pch$ to be the set of unit tangent vectors that are based on $\pch$ and point into $\pch^\circ$. The set $\widetilde \pch$ is called the \textit{precell in $S\h$ corresponding to $\pch$}. If $\pch$ is attached to the vertex $v$ of $\mc K$, we call $\widetilde \pch$ a \textit{precell in $S\h$ attached to $v$}. 
\end{defi}

Recall the projection $\pr\colon S\h\to \h$ on base points.

\begin{remark}\label{bijec} Let $\pch$ be a precell in $\h$ and $\wt\pch$ the corresponding precell in $S\h$. Since $\pch$ is a convex polyhedron with non-empty interior, $\pr(\widetilde \pch)$ is the precell in $\h$ to which $\widetilde \pch$ corresponds.
\end{remark}

\begin{lemma}\label{predisj}
Let $\pch_1,\pch_2$ be two different precells in $\h$. Then the precells $\widetilde{\pch_1}$ and $\widetilde{\pch_2}$ in $S\h$ are disjoint.
\end{lemma}

\begin{proof}
This is an immediate consequence of Proposition~\ref{precellsum}\eqref{boundarypreH} and Definition~\ref{def_preSH}.
\end{proof}

\begin{defi}\label{def_visual}
Let $\pch$ be a precell in $\h$ and $\widetilde \pch$ the corresponding precell in $S\h$. The set $\vb(\widetilde \pch)$ of unit tangent vectors based on $\partial \pch$ and pointing along $\partial \pch$ is called the \textit{visual boundary of $\widetilde \pch$}. Further, $\vc(\widetilde \pch) \sceq \widetilde \pch \cup \vb(\widetilde \pch)$ is said to be the \textit{visual closure of $\widetilde \pch$}.
\end{defi}

\begin{figure}[h]
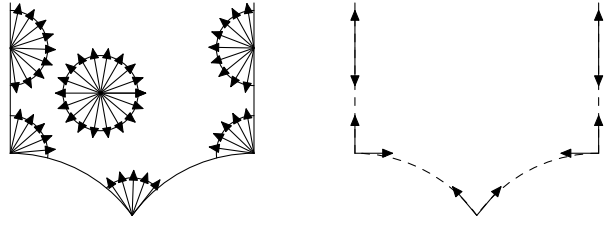

\begin{center}
\includegraphics*{Hecke5.5} \hspace{1cm}
\includegraphics*{Hecke5.6}
\end{center}
\caption{The precell in $S\h$ and its visual boundary for a non-cuspidal precell in $\h$.}
\end{figure}

The next lemma is clear from the definitions.

\begin{lemma}\label{boundarypreSH}
Let $\pch$ be a precell in $\h$ and $\widetilde \pch$ be the corresponding precell in $S\h$. Then $\vc(\widetilde \pch)$ is the disjoint union of $\widetilde \pch$ and $\vb(\widetilde \pch)$.
\end{lemma}

We remark that the visual boundary and the visual closure of a precell $\wt\pch$ in $S\h$ is a proper subset of the topological boundary resp.\@ closure of $\wh\pch$ in $S\h$.

\begin{prop}\label{precellsSH}
Let $\{\pch_j\mid j\in J\}$ be a basal family of precells in $\h$ and let $\{ \widetilde{\pch_j} \mid j\in J\}$ be the set of corresponding precells in $S\h$. Then there is a fundamental set $\widetilde \fd$ for $\Gamma$ in $S\h$ such that 
\begin{equation}\label{fundsetSH}
\bigcup_{j\in J} \widetilde{\pch_j} \subseteq \widetilde \fd \subseteq \bigcup_{j\in J} \vc\big( \widetilde{\pch_j} \big).
\end{equation}
Moreover, $\pr(\widetilde \fd) = \bigcup_{j\in J}\pch_j$. If $\wt\fd$ is a fundamental set for $\Gamma$ in $S\h$ such that $\wt\fd\subseteq \bigcup_{j\in J} \vc(\wt{\pch_j})$, then $\wt\fd$ satisfies \eqref{fundsetSH}.
Conversely, if $\{ \widetilde \pch_j \mid j\in J\}$ is a set, indexed by $J$, of precells in $S\h$ such that \eqref{fundsetSH} holds for some fundamental set $\wt\fd$ for $\Gamma$ in $S\h$, then the family $\{ \pr(\widetilde \pch_j) \mid j\in J\}$ is a basal family of precells in $\h$. 
\end{prop}

\begin{proof}
This is a slight extension of a similar statement in \cite{Hilgert_Pohl}.
\end{proof}

\begin{remark}
Recall from Theorem~\ref{precellsH} that each basal family of precells in $\h$ contains the same finite number of precells, say $m$. Proposition~\ref{precellsSH} shows that if $\{ \widetilde \pch_k\mid k\in K\}$ is a set of precells in $S\h$, indexed by $K$, such that \eqref{fundsetSH} holds for some fundamental set $\wt\fd$ for $\Gamma$ in $S\h$, then $\# K = m$.
\end{remark}

\begin{defi}\label{def_wtfpch}
Let $\fpch \sceq \{ \pch_j\mid j\in J\}$ be a basal family of precells in $\h$. Then the set $\widetilde\fpch \sceq \{ \widetilde{\pch_j} \mid j\in J\}$ of corresponding precells in $S\h$ is called a \textit{basal family of precells in $S\h$} or a \textit{family of basal precells in $S\h$}. If $\fpch$ is a connected family of basal precells in $\h$, then $\widetilde\fpch$ is said to be a \textit{connected family of basal precells in $S\h$} or a \textit{connected basal family of precells in $S\h$}. 
\end{defi}

Let $\fpch\sceq \{ \pch_j\mid j\in J\}$ be a basal family of precells in $\h$ and let $\widetilde\fpch \sceq \{ \widetilde{\pch_j} \mid j\in J\}$ be the corresponding basal family of precells in $S\h$.

\begin{defirem}
We call two cycles $c_1,c_2$ in $\fpch\times\Gamma$ \textit{equivalent} if there exists a basal precell $\pch\in\fpch$ and elements $g_1,g_2\in\Gamma\mminus\Gamma_\infty$ such that $(\pch, g_1)$ is an element of $c_1$ and $(\pch, g_2)$ is an element of $c_2$. Obviously, equivalence of cycles is an equivalence relation (on the set of all cycles). If $[c]$ is an equivalence class of cycles in $\fpch\times\Gamma$, then each element $(\pch, h_\pch)\in\fpch\times\Gamma$ in any representative $c$ of $[c]$ is called a \textit{generator} of $[c]$. 
\end{defirem}

\begin{lemma}\label{cyclic}
Let $\pch$ be a non-cuspidal basal precell in $\h$ and suppose that $h_\pch$ is an element in $\Gamma\mminus\Gamma_\infty$ assigned to $\pch$ by Proposition~\ref{structureglue}. Let $\big( (\pch_j,h_j)\big)_{j=1,\ldots, k}$ be the cycle in $\fpch\times\Gamma$ determined by $(\pch, h_\pch)$. If $\pch=\pch_l$ for some $l\in \{2,\ldots, k\}$, then $h_l=h_\pch$. Moreover, if
\[
 q \sceq \min\big\{ l\in\{1,\ldots,k-1\} \ \big\vert\  \pch_{l+1} = \pch \big\}
\]
exists, then $q$ does not depend on the choice of $h_\pch$, $k$ is a multiple of $q$, and $(\pch_{l+q}, h_{l+q}) = (\pch_l, h_l)$ for $l\in\{1,\ldots, k-q\}$.
\end{lemma}

\begin{proof}
This is an obvious consequence of the definition of cycles.
\end{proof}

\begin{defi}
Let $\pch$ be a non-cuspidal basal precell in $\h$ and let $h_\pch$ be an element in $\Gamma\mminus\Gamma_\infty$ assigned to $\pch$ by Proposition~\ref{structureglue}. Suppose that $\big( (\pch_j, h_j) \big)_{j=1,\ldots, k}$ is the cycle in $\fpch\times\Gamma$ determined by $(\pch, h_\pch)$. We set
\[
\cyl(\pch) \sceq \min \big(\big\{ l\in\{1,\ldots,k-1\} \ \big\vert\  \pch_{l+1} = \pch  \big\} \cup \{k\} \big).
\]
Lemma~\ref{cyclic} shows that $\cyl(\pch)$ is well-defined. Moreover, it implies that $\cyl(\pch)$ does not depend on the choice of the generator $(\pch, h_\pch)$ of an equivalence class of cycles.
For a cuspidal basal precell $\pch$ in $\h$ we set 
\[
 \cyl(\pch) \sceq 3,
\]
and for a basal strip precell $\pch$  in $\h$ we define
\[
 \cyl(\pch) \sceq 2.
\]
\end{defi}

\begin{example} 
Recall Example~\ref{excycles}. For the Hecke triangle group $G_n$ and its basal precell $\pch$ in $\h$  we have $\pch=\pch_2$ and hence $\cyl(\pch) = 1$. In contrast, the basal precell $\pch(v_1)$ in $H$ of the congruence group $\PGamma_0(5)$ appears only once in the cycle in $\fpch\times \PGamma_0(5)$ and therefore $\cyl\big(\pch(v_1)\big) = 3$.
\end{example}

\begin{constrdefi}\label{cellsSH}
Set $\fd\sceq \bigcup_{j\in J}\pch_j$. Pick a fundamental set $\widetilde \fd$ for $\Gamma$ in $S\h$ such that
\[
 \bigcup_{j\in J} \widetilde{\pch_j} \subseteq \widetilde \fd \subseteq \bigcup_{j\in J} \vc\big( \widetilde{\pch_j}\big),
\]
which is possible by Proposition~\ref{precellsSH}.
For each basal precell $\pch\in\fpch$ and each $z\in \fd$ let $\widetilde E_z(\pch)$ denote the set of unit tangent vectors in $\widetilde \fd\cap \vc({\widetilde \pch})$ based at $z$. Fix any enumeration of the index set $J$ of $\fpch$, say $J=\{j_1,\ldots, j_k\}$. For $z\in\fd$ and $l\in\{1,\ldots,k\}$ set
\[
 \wt\fd_z\big(\pch_{j_1}\big) \sceq \wt E_z\big(\pch_{j_1}\big) \quad\text{and}\quad \wt\fd_z\big(\pch_{j_l}\big) \sceq \wt E_z\big(\pch_{j_l}\big)\mminus\bigcup_{m=1}^{l-1} \wt E_z\big(\pch_{j_m}\big).
\]
Further set 
\[
 \wt\fd(\pch) \sceq \bigcup_{z\in\pch}\wt\fd_z(\pch)
\]
for $\pch\in\fpch$.
Recall from Proposition~\ref{precellsSH} that $\pr(\widetilde \fd) = \fd$. Thus, 
\begin{equation}\label{superunion}
\bigcup_{z\in \fd} \bigcup_{\pch\in\fpch} \widetilde \fd_z(\pch) = \widetilde \fd,
\end{equation}
and the union is disjoint.
For each equivalence class of cycles in $\fpch\times\Gamma$ fix a generator and let $\choices$ denote the set of chosen generators. Let $(\pch,h_\pch)\in\choices$.

Suppose that $\pch$ is a non-cuspidal precell in $\h$ and let $v$ be the vertex of $\mc K$ to which $\pch$ is attached. Let $\big( (\pch_j,h_j) \big)_{j=1,\ldots,k}$ be the cycle in $\fpch\times\Gamma$ determined by $(\pch, h_\pch)$. For $j=1,\ldots, k$ set $g_1\sceq \id$ and $g_{j+1} \sceq h_jg_j$, and let $s_j$ be the summit of $I(h_j)$. Further, for convenience, set $h_0\sceq h_k$ and $s_0\sceq s_k$. In the following we partition certain $\wt\fd(\pch_j)$ into $k$ subsets. More precisely, we partition each element of the set $\{ \wt\fd(\pch_j)\mid j=1,\ldots, k\}$ into $k$ subsets. Let $j\in \{1,\ldots, \cyl(\pch)\}$.

For each $z\in \pch_j^\circ \cup (h_{j-1}.s_{j-1}, g_j.v] \cup [g_j.v,s_j)$ we pick any partition of $\widetilde \fd_z(\pch_j)$ into $k$ non-empty disjoint subsets $W^{(1)}_{j,z},\ldots, W^{(k)}_{j,z}$.

For $z\in [s_j,\infty)$ we set $W^{(1)}_{j,z}\sceq \widetilde \fd_z(\pch_j)$ and $W^{(2)}_{j,z} = \ldots = W^{(k)}_{j,z} \sceq \emptyset$.

For $z\in [h_{j-1}.s_{j-1},\infty)$ we set $W^{(1)}_{j,z}\sceq \emptyset$, $W^{(2)}_{j,z} \sceq \widetilde \fd_z(\pch_j)$ and $W^{(3)}_{j,z} = \ldots = W^{(k)}_{j,z} \sceq \emptyset$.

For $m\in\{1,\ldots, k\}$ and $j\in\{1,\ldots,\cyl(\pch)\}$ we set 
\[
\widetilde \pch_{j,m} \sceq \bigcup_{z\in \pch_j} W^{(m)}_{j,z}
\]
and
\[
\widetilde \ch_j(\pch,h_\pch) \sceq \bigcup_{l=1}^k g_jg_l^{-1}. \widetilde \pch_{l, l-j+1}
\]
where the first part ($l$) of the subscript of $\widetilde \pch_{l,l-j+1}$ is calculated modulo $\cyl(\pch)$ and the second part ($l-j+1$) is calculated modulo $k$.

Suppose that $\pch$ is a cuspidal precell in $\h$. Let $\big( (\pch_1,h_1), (\pch_2,h_2) \big)$ be the cycle in $\fpch\times\Gamma$ determined by $(\pch, h_\pch)$. Set $g\sceq h_1 = h_\pch$, $g_1\sceq \id$ and $g_2\sceq h_1=g$. Suppose that $v$ is the vertex of $\mc K$ to which  $A$ is attached, and let $s$ be the summit of $I(g)$. Let $j\in\{1,2\}$. We partition $\widetilde\fd(\pch_j)$ into three subsets as follows.

For $z\in \pch_j^\circ \cup (g_j.v,g_j.s)$ we pick any partition of $\widetilde \fd_z(\pch_j)$ into three non-empty disjoint subsets $W^{(1)}_{j,z}$, $W^{(2)}_{j,z}$ and $W^{(3)}_{j,z}$. 

For $z\in (g_j.v,\infty)$ we set $W^{(1)}_{j,z} \sceq \widetilde \fd_z(\pch_j)$ and $W^{(2)}_{j,z} = W^{(3)}_{j,z} \sceq \emptyset$.

For $z\in [g_j.s,\infty)$ we set $W^{(1)}_{j,z} = W^{(2)}_{j,z} \sceq \emptyset$ and $W^{(3)}_{j,z} \sceq \widetilde \fd_z(\pch_j)$.

For $m\in \{1,2,3\}$ and $j\in\{1,2\}$ we set
\[
\widetilde \pch_{j,m} \sceq \bigcup_{z\in \pch_j} W^{(m)}_{j,z}.
\]
Then we define
\begin{align*}
\widetilde \ch_1(\pch,h_\pch) &\sceq \widetilde \pch_{1,1} \cup g^{-1}. \widetilde \pch_{2,2},
\\
\widetilde \ch_2(\pch,h_\pch) & \sceq g.\widetilde \pch_{1,2} \cup \widetilde \pch_{2,1},
\\
\widetilde \ch_3(\pch,h_\pch) & \sceq \widetilde \pch_{1,3} \cup g^{-1}.\widetilde \pch_{2,3}.
\end{align*}

Suppose that $\pch$ is a strip precell in $\h$.  Let $v_1,v_2$ be the two (infinite) vertices of $\mc K$ to which $\pch$ is attached and suppose that $v_1 < v_2$. We partition $\widetilde\fd(\pch)$ into two subsets as follows. 

For $z\in \pch^\circ$ we pick any partition of $\widetilde \fd_z(\pch)$ into two non-empty disjoint subsets $W^{(1)}_z$ and $W^{(2)}_z$. 

For $z\in (v_1,\infty)$ we set $W^{(1)}_z \sceq \widetilde \fd_z(\pch)$ and $W^{(2)}_z \sceq \emptyset$. For $z\in (v_2,\infty)$ we set $W^{(1)}_z \sceq \emptyset$ and $W^{(2)}_z \sceq \widetilde \fd_z(\pch)$.

For $m\in \{1,2\}$ we define 
\[
\widetilde \ch_1(\pch,h_\pch) \sceq \bigcup_{z\in \pch} W^{(1)}_z \quad\text{and}\quad \widetilde \ch_2(\pch,h_\pch) \sceq \bigcup_{z\in \pch} W^{(2)}_z.
\]
The set $\choices$ (``selection'') is called a \textit{set of choices associated to $\fpch$}. 
The family 
\[
\widetilde{\fch}_{\choices} \sceq \left\{ \wt\ch_j(\pch,h_\pch) \left\vert\  (\pch,h_\pch)\in\choices,\ j=1,\ldots,\cyl(\pch) \vphantom{\wt\ch_j(\pch,h_\pch)}       \right.\right\}
\]
 is called the \textit{family of cells in $S\h$ associated to $\fpch$ and $\choices$}. 
The elements in this family are subject to the choice of the fundamental set $\wt\fd$, the enumeration of $J$ and some choices for partitions in unit tangent bundle. However, all further applications of $\widetilde{\fch}_{\choices}$ are invariant under these choices. This justifies to call $\widetilde{\fch}_{\choices}$ \textit{the} family of cells in $S\h$ associated to $\fpch$ and $\choices$. Each element in $\widetilde{\fch}_{\choices}$ is called a \textit{cell in $S\h$}. 
\end{constrdefi}

\begin{example}\label{HeckecellSH}
For the Hecke triangle group $G_5$ from Example~\ref{excycles} we choose $\choices = \{ (\pch, U_5) \}$. Here we have $k=5$ and $\cyl(\pch) = 1$. The first figure in Figure~\ref{decompA} indicates 
\begin{figure}[h]
\begin{center}
\includegraphics*{Hecke5.12} \hspace{1cm}
\includegraphics*{Hecke5.13}
\end{center}
\caption{A partition of $\wt\fd(\pch)$ and the cell $\wt\ch_1(\pch, U_5)$ in  $S\h$.}\label{decompA}
\end{figure}
a possible partition of $\wt\fd(\pch)$ into the sets $\wt\pch_{1,1}, \wt\pch_{1,2}, \ldots, \wt\pch_{1,5}$. The unit tangent vectors in black belong to $\wt\pch_{1,1}$, those in very light grey to $\wt\pch_{1,2}$, those in light grey to $\wt\pch_{1,3}$, those in middle grey to $\wt\pch_{1,4}$, and those in dark grey to $\wt\pch_{1,4}$. The second figure in Figure~\ref{decompA} shows the cell $\wt\ch_1(\pch,U_5)$ in $S\h$. 
\end{example}

\begin{figure}[h]
\begin{center}
\includegraphics*{Nichtcof.5} \hspace{1cm}
\includegraphics*{Nichtcof.6}
\end{center}
\caption{The cells in $S\h$ arising from $(\pch_1,\id)$.}\label{SHstripNicht}
\end{figure}

\begin{example}
For the group $\Gamma$ from Example~\ref{Gammaex}\eqref{Nichtcof} we choose as set of choices $\choices = \{ (\pch_1,\id), (\pch_2,S) \}$ (cf.\@ Example~\ref{excycles}). The cells in $S\h$ which arise from $(\pch_1,\id)$ are shown in Figure~\ref{SHstripNicht}.
The cells in $S\h$ arising from $(\pch_2,S)$ are indicated in Figure~\ref{SHcuspidalNicht}.
\begin{figure}[h]
\begin{center}
\includegraphics*{Nichtcof.7} \hspace{1cm}
\includegraphics*{Nichtcof.8} \hspace{1cm}
\includegraphics*{Nichtcof.9} 
\end{center}
\caption{The cells in $S\h$ arising from $(\pch_1,\id)$.}\label{SHcuspidalNicht}
\end{figure}
\end{example}

Let $\choices$ be a set of choices associated to $\fpch$. 

\begin{prop}\label{BfundsetSH}
The union $\bigcup_{\widetilde \ch \in \widetilde{\fch}_{\choices}} \widetilde \ch$ is disjoint and a fundamental set for $\Gamma$ in $S\h$.
\end{prop}

\begin{proof}
Construction~\ref{cellsSH} picks a fundamental set $\wt\fd$ for $\Gamma$ in $S\h$ and chooses a family $\mc P\sceq \{ \wt\fd_z(\pch) \mid z\in\fd,\ \pch\in\fpch\}$ of subsets of it. Since the union in \eqref{superunion} is disjoint, $\mc P$ is a partition of $\wt\fd$. 
Recall the notation from Construction~\ref{cellsSH}. One considers the family
\[
\mc P_1 \sceq \left\{ \wt\fd_z(\pch_j)\ \left\vert\ z\in\fd,\ (\pch,h_\pch)\in\choices,\ j=1,\ldots, \cyl(\pch) \vphantom{\wt\fd_z(\pch_j)} \right.\right\}.
\]
The elements of $\mc P_1$ are pairwise disjoint and each element of $\mc P$ is contained in $\mc P_1$. Hence, $\mc P_1$ is a partition of $\wt\fd$. The next step is to partition each element of $\mc P_1$ into a finite number of subsets. Thus, $\wt\fd$ is partitioned into some family $\mc P_2$ of subsets of $\wt\fd$.  Then each element $W$ of $\mc P_2$ is translated by some element $g(W)$ in $\Gamma$ to get the family $\mc P_3 \sceq \{ g(W).W\mid W\in\mc P_2\}$. Since $\wt\fd$ is a fundamental set for $\Gamma$ in $S\h$, the elements of $\mc P_3$ are pairwise disjoint and $\bigcup \mc P_3$ is a fundamental set for $\Gamma$ in $S\h$. Now $\mc P_3$ is partitioned into certain subsets, say into the subsets $\mc Q_l$, $l\in L$. Each cell $\wt\ch$ in $SH$ is the union of the elements in some $\mc Q_{l(\wt\ch)}$ such that $l(\wt\ch_1)\not=l(\wt\ch_2)$ if $\wt\ch_1\not=\wt\ch_2$. Therefore, the union $\bigcup \big\{ \wt\ch\ \big\vert\ \wt\ch\in\wt\fch_\choices\big\}$ is disjoint and a fundamental set for $\Gamma$ in 
$S\h$.
\end{proof}

For each $\widetilde \ch \in \widetilde{\fch}_{\choices}$ set $b(\widetilde \ch) \sceq \pr(\widetilde \ch) \cap \partial\pr(\widetilde \ch)$ and let $\CS'(\widetilde \ch)$ be the set of unit tangent vectors in $\widetilde \ch$ that are based on $b(\widetilde \ch)$ but do not point along $\partial \pr(\widetilde \ch)$.

\begin{example}\label{choicesGamma05}
Recall the congruence subgroup $\PGamma_0(5)$ and its cycles in $\fpch\times\PGamma_0(5)$ from Example~\ref{excycles}. We choose $\choices \sceq \big\{ \big(\pch(v_4),h^{-1}\big), \big(\pch(v_1), h_1\big)\big\}$ as set of choices associated to $\fpch$ and set 
\begin{align*}
\wt\ch_1 & \sceq \wt\ch_1\big(\pch(v_4), h^{-1}\big), & \wt\ch_4 & \sceq \wt\ch_1\big(\pch(v_1), h_1\big),
\\
\wt\ch_2 & \sceq \wt\ch_2\big(\pch(v_4), h^{-1}\big), & \wt\ch_5 & \sceq \wt\ch_2\big(\pch(v_1), h_1\big),
\\
\wt\ch_3 & \sceq \wt\ch_3\big(\pch(v_4), h^{-1}\big), & \wt\ch_6 & \sceq \wt\ch_3\big(\pch(v_1), h_1\big),
\end{align*}
as well as
\[
 \CS'_j \sceq \CS'\rueck\big(\wt\ch_j\big)
\]
for $j=1,\ldots, 6$. Figure~\ref{CSGamma05} shows the sets $\CS'_j$.
\begin{figure}[h]
\begin{center}
\includegraphics*{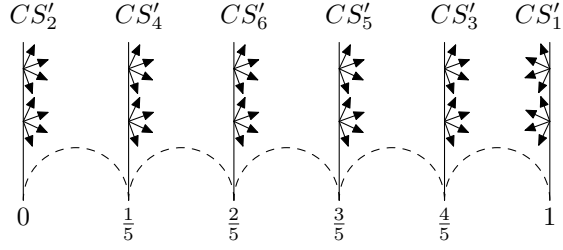} 
\end{center}
\caption{The sets $\CS'_j$.}\label{CSGamma05}
\end{figure}
\end{example}

\begin{lemma}\label{tesselationSH}
Let $\pch_1,\pch_2$ be two basal precells in $\h$ and let $g\in \Gamma$ such that $g.\vc(\widetilde{\pch_1}) \cap \vc(\widetilde{\pch_2}) \not=\emptyset$. Suppose that $\pch_1\not=\pch_2$ or $g\not=\id$. Then
\[
g.\vc\big(\widetilde{\pch_1}\big) \cap \vc\big(\widetilde{\pch_2}\big) \subseteq g.\vb\big(\widetilde{\pch_1}\big) \cap \vb\big(\widetilde{\pch_2}\big).
\]
Moreover, suppose that $\pch_1$ is cuspidal or non-cuspidal and that there is a unit tangent vector $w\in\vc(\widetilde{\pch_1})$ pointing into a non-vertical side $S_1$ of $\pch_1$ such that $gw\in \vc(\widetilde{\pch_2})$. Then $g.w$ points into a non-vertical side $S_2$ of $\pch_2$ and $g.S_1=S_2$.
\end{lemma}

\begin{proof}
We have $\pr(g.\vc(\widetilde{\pch_1})) = g.\pr(\vc(\widetilde{\pch_1})) = g.\pch_1$ and $\pr(g.\vb(\widetilde{\pch_1})) = g.\partial \pch_1$, and likewise for $\pr(\vc(\widetilde{\pch_2})) = \pch_2$ and $\pr(\vb(\widetilde{\pch_2}))=\partial\pch_2$. From 
\[
g.\vc(\wt{\pch_1})\cap \vc(\wt{\pch_2})\not=\emptyset
\]
then follows that $g.\pch_1\cap\pch_2\not=\emptyset$. By Proposition~\ref{intersectionprecells}, either $g.\pch_1=\pch_2$ and $g\in\Gamma_\infty$ or 
$g.\pch_1\cap \pch_2 \subseteq g.\partial \pch_1 \cap \partial \pch_2$. 
Assume for contradiction that $g.\pch_1=\pch_2$ with $g\in\Gamma_\infty$. Since $\pch_1$ and $\pch_2$ are basal, Corollary~\ref{props_preH} shows that $g=\id$ and $\pch_1=\pch_2$. This contradicts the hypotheses of the lemma. Hence $g.\pch_1\cap\pch_2\subseteq g.\partial\pch_1 \cap \partial\pch_2$ and therefore 
\[
g.\vc(\widetilde{\pch_1}) \cap \vc(\widetilde{\pch_2}) \subseteq g.\vb(\widetilde{\pch_1}) \cap \vb(\widetilde{\pch_2}).
\] 
Let $\pch_1$ and $w$ be as in the claim. Further let $\gamma$ be the geodesic determined by $w$. Then $g.\gamma$ is the geodesic determined by $g.w$. By definition there exists $\eps>0$ such that $\gamma( (0,\eps) ) \subseteq S_1$ and $g.\gamma( (0,\eps) )\subseteq \pch_2$. Then 
\[
\gamma\big( (0,\eps) \big) \subseteq S_1 \cap g^{-1}. \pch_2\subseteq \pch_1\cap g^{-1}.\pch_2.
\]
Since the sets $\pch_1$ and $g^{-1}.\pch_2$ intersect in more than one point and $\pch_1\not= g^{-1}.\pch_2$, Proposition~\ref{intersectionprecells} states that $\pch_1 \cap g^{-1}.\pch_2$ is a common side of $\pch_1$ and $g^{-1}.\pch_2$. Necessarily, this side is $S_1$. Proposition~\ref{intersectionprecells} shows further that $g.S_1$ is a non-vertical side of $\pch_2$. Thus, $g.w$ points along the non-vertical side $g.S_1$ of $\pch_2$.
\end{proof}

\begin{prop}\label{ncSH}
Let $(\pch,h_\pch)\in\choices$ and suppose that $\pch$ is a non-cuspidal precell in $\h$. Let $\big((\pch_j,h_j)\big)_{j=1,\ldots,k}$ be the cycle in $\fpch\times\Gamma$ determined by $(\pch,h_\pch)$. For each $m=1,\ldots, \cyl(\pch)$ we have $b(\widetilde \ch_m(\pch,h_\pch)) = (h_m^{-1}.\infty,\infty)$ and 
\[
\pr\big( \widetilde \ch_m(\pch,h_\pch)\big) = \ch(\pch_m)^\circ \cup (h_m^{-1}.\infty,\infty).
\]
Moreover, $\CS'(\widetilde \ch_m(\pch,h_\pch))$ is the set of unit tangent vectors based on $(h_m^{-1}.\infty,\infty)$ that point into $\ch(\pch_m)^\circ$, and $\pr(\CS'(\widetilde \ch_m(\pch,h_\pch))) = b(\widetilde \ch_m(\pch,h_\pch))$.
\end{prop}

\begin{proof}
We use the notation from Construction~\ref{cellsSH}. Let $j\in\{1,\ldots, \cyl(\pch)\}$ and $z\in\pch_j$. At first we show that $\wt\fd_z(\pch_j)\not=\emptyset$. For each choice of $\wt\fd$ we have $\wt{\pch_j}\subseteq \wt\fd\cap \vc(\wt{\pch_j})$. Remark~\ref{bijec} states that $\pr(\wt{\pch_j})=\pch_j$. Hence $\wt E_z(\pch_j)\cap \wt\pch_j\not=\emptyset$. More precisely, if $\big(\wt{\pch_j}\big)_z$ denotes the set of unit tangent vectors based on $z$ that point into $\pch_j^\circ$, then $\big(\widetilde{\pch_j}\big)_z = \wt E_z(\pch_j) \cap \wt{\pch_j}$. The set $\big(\wt{\pch_j}\big)_z$ is non-empty, since $\pch_j$ is convex with non-empty interior. Let $k\in J$ such that $\pch_k\not=\pch_j$. Then 
\[
\wt E_z(\pch_k)\cap \wt E_z(\pch_j) \subseteq \vc\big(\wt{\pch_k}\big) \cap \vc\big(\wt{\pch_j}\big) \subseteq \vb\big(\wt{\pch_k}\big)\cap \vb\big(\wt{\pch_j}\big),
\]
where the last inclusion follows from Lemma~\ref{tesselationSH}. Since $\wt{\pch_j} \cap \vb(\wt{\pch_j}) = \emptyset$ by Lemma~\ref{boundarypreSH}, it follows that 
\[
\big(\wt{\pch_j}\big) \cap \wt E_z(\pch_k) = \wt{\pch_j}\cap \wt E_z(\pch_j) \cap \wt E_z(\pch_k) \subseteq\wt{\pch_j} \cap \vb\big(\wt{\pch_j}\big) = \emptyset.
\]
Hence 
\begin{equation}\label{haveelem}
\big(\wt{\pch_j}\big)_z \subseteq \wt\fd_z(\pch_j).
\end{equation}
Let $j\in\{1,\ldots, \cyl(\pch)\}$ set $\widetilde \ch_j \sceq \widetilde \ch_j(\pch,h_\pch)$ and
\[
T_j \sceq \pch_j^\circ \cup (h_{j-1}.s_{j-1}, g_j.v] \cup [g_j.v, s_j).
\]
Let $m\in\{1,\ldots, k\}$. Then
\begin{align*}
\pr\big( \widetilde \pch_{j,m}\big) & = \bigcup_{z\in \pch_j} \pr\big(W_{j,z}^{(m)}\big)
\\
& = \bigcup_{z\in T_j} \pr\big(W_{j,z}^{(m)}\big) \cup \bigcup_{z\in [s_j,\infty)} \pr\big(W_{j,z}^{(m)}\big) \cup \bigcup_{ z\in [h_{j-1}.s_{j-1}, \infty)} \pr\big(W_{j,z}^{(m)}\big)
\\
& = 
\begin{cases}
T_j & \text{for $m\notin\{1,2\}$,}
\\
T_j \cup [s_j,\infty) & \text{for $m=1$,}
\\
T_j \cup [h_{j-1}.s_{j-1},\infty) & \text{for $m=2$.}
\end{cases}
\end{align*}
Note that necessarily $k\geq 3$. Then 
\begin{align*}
\widetilde \ch_j & = \bigcup_{l=1}^k g_jg_l^{-1}. \widetilde \pch_{l,l-j+1}
\\
& = \bigcup_{l=1}^{j-1} g_jg_l^{-1}. \widetilde \pch_{l,l-j+1} \cup g_jg_j^{-1}. \widetilde \pch_{j,1} \cup g_jg_{j+1}^{-1}. \widetilde \pch_{j+1,2} \cup \bigcup_{l=j+2}^k g_jg_l^{-1}. \widetilde \pch_{l,l-j+1}.
\end{align*}
Since $l-j+1 \not\equiv 1,2 \mod k$ for $l\in\{1,\ldots, j-1\} \cup \{ j+2,\ldots, k\}$, it follows that 
\begin{align*}
\pr\big(\widetilde \ch_j\big) & = \bigcup_{l=1}^{j-1} g_jg_l^{-1}. \pr\big(\widetilde \pch_{l,l-j+1}\big) \cup \pr\big(\widetilde \pch_{j,1}\big) \cup h_j^{-1}.\pr\big(\widetilde \pch_{j+1,2}\big) 
\\
& \quad\ \cup \bigcup_{l=j+2}^k g_jg_l^{-1}. \pr\big( \widetilde \pch_{l,l-j+1}\big)
\\
& = \bigcup_{l=1}^{j-1} g_jg_l^{-1}. T_l \cup T_j \cup [s_j,\infty) \cup h_j^{-1}. T_{j+1} \cup h_j^{-1}. [h_j.s_j,\infty) \cup \bigcup_{l=j+2}^k g_jg_l^{-1}. T_l
\\
& = \bigcup_{l=1}^k g_jg_l^{-1}.T_l \cup (h_j^{-1}.\infty, \infty).
\end{align*}
For the last equality we use that $\pr_\infty(s_j) = h_j^{-1}.\infty$ by Lemma~\ref{isoprops}. Hence $s_j$ is contained in the geodesic segment $\pr_\infty^{-1}(h_j^{-1}.\infty)\cap \h = (h_j^{-1}.\infty,\infty)$, which shows that the union of the two geodesic segments $[s_j,\infty)$ and $[s_j,h_j^{-1}.\infty)$ is indeed $(h_j^{-1}.\infty,\infty)$.
Proposition~\ref{nccells} implies that 
\[
\pr\big(\widetilde \ch_j\big) = \ch(\pch_j)^\circ \cup (h_j^{-1}.\infty,\infty).
\]
This shows that $b(\widetilde \ch_j) = (h_j^{-1}.\infty,\infty)$. The set of unit tangent vectors in $\widetilde \ch_j$ that are based on $b(\widetilde \ch_j)$ is the disjoint union 
\begin{align*}
D'_j & \sceq \bigcup_{z\in [s_j,\infty)} W_{j,z}^{(1)} \cup h_j^{-1}. \bigcup_{z\in [h_j.s_j,\infty)} W_{j+1,z}^{(2)} 
\\
& = \bigcup_{z\in [s_j,\infty)} \widetilde \fd_z(\pch_j) \cup h_j^{-1}. \bigcup_{z\in [h_j.s_j,\infty)} \widetilde \fd_z(\pch_{j+1}).
\end{align*}
To show that $\CS'(\widetilde \ch_j)$ is the set of unit tangent vectors based on $b(\widetilde \ch_j)$ that point into $\ch(\pch_j)^\circ$ we have to show that $D'_j$ contains all unit tangent vectors based on $[s_j,\infty)$ that point into $\pch_j^\circ$ and all unit tangent vectors based on $(h_j^{-1}.\infty, s_j]$ that point into $h_j^{-1}. \pch_{j+1}^\circ$ and the unit tangent vector which is based at $s_j$ and points into $[s_j,g_jv]$.  If $w$ is a unit tangent vector based on $[h_j.s_j,\infty)$ that points into $\pch_{j+1}^\circ$, then, clearly, $h_j^{-1}.w$ is a unit tangent vector based on $[s_j,h_j^{-1}.\infty)$ that points into $h_j^{-1}.\pch_{j+1}^\circ$. Hence, \eqref{haveelem} shows that $D'_j$ contains all unit tangent vectors of the first two kinds mentioned above. Let $w$ be the unit tangent vector with $\pr(w) = s_j$ which points into $[s_j,g_j.v]$. 

Suppose first that $w\in\wt\fd$. Then $w\in \vc(\wt{\pch_j})\cap\wt\fd$ and therefore $w\in \wt E_{s_j}(\pch_j)$. Let $k\in J$ with $\pch_k\not=\pch_j$. Assume for contradiction that $w\in \vc(\wt{\pch_k})$. Lemma~\ref{tesselationSH} implies that $[s_j,g_j.v]$ is a non-vertical side of $\pch_k$, which is a contradiction. Hence $w\notin \wt E_{s_j}(\pch_k)$. Therefore, $w\in \wt\fd_{s_j}(\pch_j)$ and hence $w\in D'_j$.

Suppose now that $w\notin\widetilde\fd$. Then there exists a unique $g\in\Gamma\mminus\{\id\}$ such that $gw\in\widetilde \fd$. Let $\pch$ be a basal precell in $\h$ such that $g.w\in\vc(\widetilde \pch) \cap \widetilde \fd$. Lemma~\ref{tesselationSH} shows that $g.[s_j,g_j.v]$ is a non-vertical side $S$ of $\pch$. Thus, $g^{-1}.\pch \cap \ch(\pch_j)^\circ \not= \emptyset$. By Proposition~\ref{nccells}\eqref{ncc4} there is a unique $l\in \{1,\ldots,k\}$ such that $g=g_lg_j^{-1}$ and $\pch=\pch_l$. Now $[s_j,g_j.v]$ is mapped by $h_j$ to the non-vertical side $[h_j.s_j, g_{j+1}.v]$ of $\pch_{j+1}$. Thus, $g=h_j$ and $\pch = \pch_{j+1}$. Then $h_j.w \in \vc(\wt{\pch_{j+1}})$. As before we see that $w\in D'_j$.  Moreover, $\pr(\CS'(\widetilde \ch_j)) = b(\widetilde \ch_j)$.
\end{proof}

Analogously to Proposition~\ref{ncSH} one proves the following two propositions.

\begin{prop}\label{ccSH}
Let $(\pch,h_\pch)\in\choices$ and suppose that $\pch$ is a cuspidal precell in $\h$. Let $v$ be the vertex of $\mc K$ to which  $\pch$ is attached and let $\big( (\pch, g), (\pch', g^{-1})\big)$ be the  cycle in $\fpch\times\Gamma$ determined by $(\pch,h_\pch)$. Then 
\[
b\big(\widetilde \ch_1(\pch,h_\pch)\big) = (v,\infty), \ b\big(\widetilde \ch_2(\pch,h_\pch)\big) = (g.v,\infty),\ b\big(\widetilde \ch_3(\pch,h_\pch)\big) = (g^{-1}.\infty,\infty)
\]
and
\begin{align*}
\pr\big(\widetilde \ch_1(\pch,h_\pch)\big) & = \ch(\pch)^\circ \cup (v,\infty),
\\
\pr\big(\widetilde \ch_2(\pch,h_\pch)\big) & = \ch(\pch')^\circ \cup (g.v,\infty),
\\
\pr\big(\widetilde \ch_3(\pch,h_\pch)\big) & = \ch(\pch)^\circ \cup (g^{-1}.\infty,\infty).
\end{align*}
Moreover, $\CS'(\widetilde \ch_m(\pch,h_\pch))$ is the set of unit tangent vectors based on $b(\widetilde \ch_m(\pch,h_\pch))$ that point into $\pr(\widetilde \ch_m(\pch,h_\pch))^\circ$, and $\pr(\CS'(\widetilde \ch_m(\pch,h_\pch))) = b(\widetilde \ch_m(\pch,h_\pch))$ for $m=1,2,3$.
\end{prop}

\begin{prop}\label{stripSH}
Let $(\pch,h_\pch)\in\choices$ and suppose that $\pch$ is a strip precell. Let $v_1,v_2$ be the two (infinite) vertices of $\mc K$ to which  $\pch$ is attached and suppose that $v_1 < v_2$. For $m=1,2$ we have $b(\widetilde \ch_m(\pch,h_\pch)) = (v_m,\infty)$ and 
\[
\pr\big( \widetilde \ch_m(\pch,h_\pch) \big) = \ch(\pch)^\circ \cup (v_m,\infty) = \pch^\circ \cup (v_m,\infty).
\]
Moreover, $\CS'(\widetilde \ch_m(\pch,h_\pch))$ is the set of unit tangent vectors based on $(v_m,\infty)$ that point into $\ch(\pch)^\circ$, and $\pr(\CS'(\widetilde \ch_m(\pch,h_\pch))) = b(\widetilde \ch_m(\pch,h_\pch))$.
\end{prop}

\begin{cor}\label{iscellH}
Let $\wt\ch\in \wt\fch_\choices$. Then $\ch\sceq\cl(\pr(\wt\ch))$ is a cell in $\h$ and $b(\wt\ch)$ a side of $\ch$. Moreover, $\pr(\wt\ch)=\ch^\circ\cup b(\wt\ch)$ and $\pr(\wt\ch)^\circ=\ch^\circ$.
\end{cor}

The development of a symbolic dynamics for the geodesic flow on $Y$ via the family $\wt\fch_\choices$ of cells in $S\h$ is based on the following properties of the cells $\wt\ch$ in $S\h$: It uses that $\cl(\pr(\wt\ch))$ is a convex polyhedron of which each side is a complete geodesic segment and that each side is the image under some element $g\in\Gamma$ of the complete geodesic segment $b(\wt\ch')$ for some cell $\wt\ch'$ in $S\h$. It further uses that $\bigcup \wt\fch_\choices$ is a fundamental set for $\Gamma$ in $S\h$ and that $\{ g.\cl(\pr(\wt\ch)) \mid g\in\Gamma,\ \wt\ch\in\wt\fch_\choices\}$ is a tesselation of $\h$. Moreover, one needs that $b(\wt\ch)$ is a vertical side of $\pr(\wt\ch)$ and that $\CS'(\wt\ch)$ is the set of unit tangent vectors based on $b(\wt\ch)$ that point into $\pr(\wt\ch)^\circ$. It does not use that $\{ \cl(\pr(\wt\ch)) \mid \wt\ch\in\wt\fch_\choices\}$ is the set of all cells in $\h$ nor does one need that for some cells $\wt\ch_1,\wt\ch_2\in\wt\fch_\choices$ one has $\cl(\
pr(\wt\ch_1))=\ch(\pr(\wt\ch_2))$. This means that one has the freedom to perform (horizontal) translations of single cells in $S\h$ by elements in $\Gamma_\infty$. The following definition is motivated by this fact. We will see that in some situations the family of shifted cells in $S\h$ will induce a symbolic dynamics which has a generating function for the future part while the symbolic dynamics that is constructed from the original family of cells in $S\h$ has not.

\begin{defi}\label{defi_shmap}
Each map $\shmap\colon \wt\fch_\choices\to \Gamma_\infty$ (``translation'') is called a \textit{shift map} for $\wt\fch_\choices$. 
The family
\[
\wt\fch_{\choices,\shmap} \sceq \big\{ \shmap\big(\wt\ch\big).\wt\ch \ \big\vert\  \wt\ch\in\wt\fch_\choices\big\}
\]
is called the \textit{family of cells in $S\h$ associated to $\fpch$, $\choices$ and $\shmap$.} Each element of $\wt\fch_{\choices,\shmap}$ is called a \textit{shifted cell in $S\h$}. 

For each $\wt\ch\in\wt\fch_{\choices,\shmap}$ define $b(\wt\ch) \sceq \pr(\wt\ch)\cap\partial\pr(\wt\ch)$ and let $\CS'(\wt\ch)$ be the set of unit tangent vectors in $\wt\ch$ that are based on $b(\wt\ch)$ but do not point along $\partial\pr(\wt\ch)$.
\end{defi}

Let $\shmap$ be a shift map for $\wt\fch_\choices$.

\begin{remark}\label{remain}
The results of Propositions~\ref{BfundsetSH} and \ref{ncSH}-\ref{stripSH} remain true for $\wt\fch_{\choices,\shmap}$ after the obvious changes. More precisely, the union $\bigcup \wt\fch_{\choices,\shmap}$ is disjoint and a fundamental set for $\Gamma$ in $S\h$, and if $\wt\ch\in \wt\fch_\choices$, then $\pr\big( \shmap(\wt\ch).\wt\ch\big) = \shmap(\wt\ch).\pr(\wt\ch)$ and $b\big(\shmap(\wt\ch).\wt\ch\big) = \shmap(\wt\ch). b(\wt\ch)$. Then $\CS'\big(\shmap(\wt\ch).\wt\ch\big)$ is the set of unit tangent vectors based on $\shmap(\wt\ch). b(\wt\ch)$ that point into $\shmap(\wt\ch).\pr(\wt\ch)^\circ$.
\end{remark}

\section{Geometric symbolic dynamics}\label{sec_geomsymdyn}

Let $\Gamma$ be a geometrically finite subgroup of $\PSL_2(\R)$ of which $\infty$ is a cuspidal point and which satisfies \eqref{A4}. Suppose that the set of relevant isometric spheres is non-empty.  Let $\fpch$ be a basal family of precells in $\h$ and denote the family of cells in $\h$ assigned to $\fpch$ by $\fch$. Suppose that $\choices$ is a set of choices associated to $\fpch$ and let $\wt\fch_\choices$ be the family of cells in $S\h$ associated to $\fpch$ and $\choices$. Fix a shift map $\shmap$ for $\wt\fch_\choices$ and denote the family of cells in $S\h$ associated to $\fpch,\choices$ and $\shmap$ by $\wt\fch_{\choices,\shmap}$. Recall the set $\CS'(\wt\ch)$ for $\wt\ch\in\wt\fch_{\choices, \shmap}$ from Definition~\ref{defi_shmap}. We set 
\[
\CS'\rueck\big(\wt\fch_{\choices,\shmap}\big) \sceq \bigcup_{\wt\ch\in\wt\fch_{\choices,\shmap}} \CS'\rueck\big(\wt\ch\big) \quad\text{and}\quad \wh\CS\big(\wt\fch_{\choices,\shmap}\big) \sceq \pi\big( \CS'\rueck\big(\wt\fch_{\choices,\shmap}\big)\big).
\]
In Section~\ref{sec_geomcross}, we will use the results from Section~\ref{sec_preSH} to show that $\wh\CS$ satisfies (\apref{C}{C1}{}) and hence is a cross section for the geodesic flow on $Y$ \wrt certain measures $\mu$. It will turn out that the measures $\mu$ are characterized by the condition that $\NIC$ (see Remark~\ref{outlook}) be a $\mu$-null set.

\subsection{Geometric cross section}\label{sec_geomcross}

Recall the set $\BS$ from Section~\ref{sec_base}.

\begin{lemma}\label{propsSH}
Let $\wt\ch\in\wt\fch_{\choices,\shmap}$. Then $\pr(\wt\ch)$ is a convex polyhedron and $\partial\pr(\wt\ch)$ consists of complete geodesic segments. Moreover, we have that $\pr(\wt\ch)^\circ\cap \BS=\emptyset$ and $\partial\pr(\wt\ch) \subseteq\BS$ and $\pr(\wt\ch)\cap\BS=b(\wt\ch)$ and that $b(\wt\ch)$ is a connected component of $\BS$.
\end{lemma}

\begin{prop}\label{CS=CShat} 
We have $\wh\CS = \wh\CS(\wt\fch_{\choices,\shmap})$. Moreover, the union 
\[ 
\CS'\rueck\big(\wt\fch_{\choices,\shmap}\big) = \bigcup\big\{ \CS'\rueck\big(\wt\ch\big) \ \big\vert\  \wt\ch\in\wt\fch_{\choices,\shmap}\big\}
\]
is disjoint and $\CS'(\wt\fch_{\choices,\shmap})$ is a set of representatives for $\wh\CS$.
\end{prop}

\begin{proof}
We start by showing that $\wh\CS(\wt\fch_{\choices,\shmap})\subseteq\wh\CS$. Let $\wt\ch\in \wt\fch_{\choices,\shmap}$. Then there exists a (unique) $\wt\ch_1\in\wt\fch_{\choices}$ such that $\wt\ch=\shmap(\wt\ch_1).\wt\ch_1$. Lemma~\ref{propsSH} shows that $b(\wt\ch_1)$ is a connected component of $\BS$. The set $\CS'(\wt\ch_1)$ consists of unit tangent vectors based on $b(\wt\ch_1)$ which are not tangent to it. Therefore, $\CS'(\wt\ch_1)\subseteq \CS$. Now $b(\wt\ch)=\shmap(\wt\ch_1).\wt\ch_1$ and $\CS'(\wt\ch) = \shmap(\wt\ch_1).\CS'(\wt\ch_1)$ with $\shmap(\wt\ch_1)\in\Gamma$. Thus, we see that $\pi(b(\wt\ch)) \subseteq\pi(\BS)=\wh\BS$ and $\pi(\CS'(\wt\ch))\subseteq\pi(\CS)=\wh\CS$. This shows that $\wh\CS(\wt\fch_{\choices,\shmap})\subseteq\wh\CS$. 

Conversely, let $\wh v\in\wh\CS$. We will show that there is a unique $\wt\ch\in\wt\fch_{\choices,\shmap}$ and a unique $v\in \CS'(\wt\ch)$ such that $\pi(v)=\wh v$. Pick any $w\in\pi^{-1}(v)$. Remark~\ref{remain} shows that the set $\mc P\sceq \bigcup\{\wt\ch\mid \wt\ch\in\wt\fch_{\choices,\shmap}\}$ is a fundamental set for $\Gamma$ in $S\h$. Hence there exists a unique pair $(\wt\ch, g)\in \wt\fch_{\choices,\shmap}\times\Gamma$ such that $v\sceq g.w\in \wt\ch$. Note that $\pi^{-1}(\wh\CS)=\CS$. Thus, $v\in\CS$ and hence $\pr(v)\in\pr(\wt\ch)\cap \BS$. Lemma~\ref{propsSH} shows that $\pr(v)\in b(\wt\ch)$. Therefore, $v\in \pi^{-1}(b(\wt\ch))\cap \wt\ch$. Since $v\in \CS$, it does not point along $b(\wt\ch)$. Hence $v$ does not point along $\partial\pr(\wt\ch)$, which shows that $v\in \CS'(\wt\ch)$. This proves that $\wh\CS\subseteq\wh\CS(\wt\fch_{\choices,\shmap})$. 

To see the uniqueness of $\wt\ch$ and $v$ suppose that $w_1\in\pi^{-1}(\wh v)$. Let $(\wt\ch_1,g_1)\in\wt\fch_{\choices,\shmap}\times\Gamma$ be the unique pair such that $g_1.w_1\in\wt\ch_1$. There exists a unique element $h\in\Gamma$ such that $h.w=w_1$. Then $g_1hg^{-1}.v=g.w_1$ and $v,g_1hg^{-1}.v\in\mc P$. Now $\mc P$ being a fundamental set shows that $g_1hg^{-1}=\id$, which proves that $g_1.w_1=g_1h.w=g.w=v$ and $\wt\ch_1=\wt\ch$. This completes the proof.
\end{proof}

\begin{cor}
Let $\wh\gamma$ be a geodesic on $Y$ which intersects $\wh\CS$ in $t$. Then there is a unique geodesic $\gamma$ on $H$ which intersects $\CS'(\wt\fch_{\choices,\shmap})$ in $t$ such that $\pi(\gamma)=\wh\gamma$.
\end{cor}

\begin{defi}
Let $\wh\gamma$ be a geodesic on $Y$ which intersects $\wh\CS$ in $\wh\gamma'(t_0)$. If 
\[
 s\sceq \min \big\{ t>t_0 \ \big\vert\  \wh\gamma'(t) \in \wh\CS \big\}
\]
exists, we call $s$ the \textit{first return time} of $\wh\gamma'(t_0)$ and $\wh\gamma'(s)$ the \textit{next point of intersection of $\wh\gamma$ and $\wh\CS$}. 
Let $\gamma$ be a geodesic on $H$. If $\gamma'(t)\in \CS$, then we say that $\gamma$ \textit{intersects $\CS$ in $t$}. 
If there is a sequence $(t_n)_{n\in\N}$ with $\lim_{n\to\infty} t_n = \infty$ and $\gamma'(t_n)\in \CS$ for all $n\in\N$, then $\gamma$ is said to \textit{intersect $\CS$ infinitely often in future}. Analogously, if we find a sequence $(t_n)_{n\in\N}$ with $\lim_{n\to\infty} t_n = -\infty$ and $\gamma'(t_n)\in \CS$ for all $n\in\N$, then $\gamma$ is said to \textit{intersect $\CS$ infinitely often in past}. 
Suppose that $\gamma$ intersects $\CS$ in $t_0$. If 
\[
s \sceq \min\big\{t>t_0 \ \big\vert\  \gamma'(t) \in\CS \big\}
\]
exists, we call $s$ the \textit{first return time} of $\gamma'(t_0)$ and $\gamma'(s)$ the \textit{next point of intersection of $\gamma$ and $\CS$}. Analogously, we define the \textit{previous point of intersection} of $\wh\gamma$ and $\wh\CS$ resp.\@ of $\gamma$ and $\CS$. 
\end{defi}

\begin{remark}\label{charinter}
A geodesic $\widehat \gamma$ on $Y$ intersects $\widehat\CS$ if and only if some (and hence any) representative of $\widehat\gamma$ on $\h$ intersects $\pi^{-1}(\widehat\CS)$. Recall that $\CS=\pi^{-1}(\widehat \CS)$, and that $\CS$ is the set of unit tangent vectors based on $\BS$ but which are not tangent to $\BS$. Since $\BS$ is a totally geodesic submanifold of $\h$ (see Proposition~\ref{BStotgeod}), a geodesic $\gamma$ on $\h$ intersects $\CS$ if and only if $\gamma$ intersects $\BS$ transversely. Again because $\BS$ is totally geodesic, the geodesic $\gamma$ intersects $\BS$ transversely if and only if $\gamma$ intersects $\BS$ and is not contained in $\BS$.  Therefore, a geodesic $\widehat\gamma$ on $Y$ intersects $\widehat\CS$ if and only if some (and thus any) representative $\gamma$ of $\widehat\gamma$ on $\h$  intersects $\BS$ and $\gamma(\R) \not\subseteq \BS$.

A similar argument simplifies the search for previous and next points of intersection. To make this precise, suppose that $\widehat\gamma$ is a geodesic on $Y$ which intersects $\widehat\CS$ in $\widehat\gamma'(t_0)$ and that $\gamma$ is a representative of $\widehat\gamma$ on $\h$. Then $\gamma'(t_0) \in \CS$. There is a next point of intersection of $\widehat\gamma$ and $\widehat\CS$ if and only if there is a next point of intersection of $\gamma$ and $\CS$. The hypothesis that $\gamma'(t_0)\in\CS$ implies that $\gamma(\R)$ is not contained in $\BS$. Hence each intersection of $\gamma$ and $\BS$ is transversal. Then there is a next point of intersection of $\gamma$ and $\CS$ if and only if $\gamma( (t_0,\infty) )$ intersects $\BS$. Suppose that there is a next point of intersection. If  $\gamma'(s)$ is the next point of intersection of $\gamma$ and $\CS$, then and only then $\widehat\gamma'(s)$ is the next point of intersection of $\widehat\gamma$ and $\widehat\CS$. In this case, $s=\min\{ t>t_0 \mid \
gamma(t)\in\BS\}$.

Likewise, there was a previous point of intersection of $\widehat\gamma$ and $\widehat\CS$ if and only if there was a previous point of intersection of $\gamma$ and $\CS$. Further, there was a previous point of intersection of $\gamma$ and $\CS$ if and only if $\gamma( (-\infty, t_0) )$ intersects $\BS$. If there was a previous point of intersection, then $\gamma'(s)$ is the previous point of intersection of $\gamma$ and $\CS$ if and only if $\widehat\gamma'(s)$ was the previous point of intersection of $\widehat\gamma$ and $\widehat\CS$. Moreover, $s=\max\{ t<t_0\mid \gamma(t)\in\BS\}$.
\end{remark}

Proposition~\ref{CS1} provides a characterization of the geodesics on $Y$ which intersect $\wh\CS$ at all. Using Remark~\ref{charinter}, its proof consists of elementary arguments in convex geometry. We omit it here.

\begin{prop}\label{CS1}
Let $\widehat\gamma$ be a geodesic on $Y$. Then $\widehat\gamma$ intersects $\widehat\CS$ if and only if $\widehat\gamma \notin \NC$.
\end{prop}

Suppose that we are given a geodesic $\wh\gamma$ on $Y$ which intersects $\wh\CS$ in $\wh\gamma'(t_0)$ and suppose that $\gamma$ is the unique geodesic on $\h$ which intersects $\CS'(\wt\fch_{\choices,\shmap})$ in $\gamma'(t_0)$ and which satisfies $\pi(\gamma)=\wh\gamma$. We now characterize when there is a next point of intersection of $\wh\gamma$ and $\wh\CS$ resp.\@ of $\gamma$ and $\CS$, and, if there is one, where this point is located. Further we will do analogous investigations on the existence and location of previous points of intersections. To this end we need the following preparations.

\begin{defi}\label{def_codint}
Let $\wt\ch\in\wt\fch_{\choices,\shmap}$ and suppose that $b(\wt\ch)$ is the complete geodesic segment $(a,\infty)$ with $a\in\R$. We assign to $\wt\ch$ two intervals $I(\wt\ch)$ and $J(\wt\ch)$ which are given as follows:
\begin{align*}
I\big(\wt\ch\big) & \sceq 
\begin{cases}
(a,\infty) & \text{if $\pr(\wt\ch)\subseteq \{ z\in \h\mid \Rea z\geq a\}$,}
\\
(-\infty, a) & \text{if $\pr(\wt\ch) \subseteq \{ z\in \h \mid \Rea z \leq a\}$,}
\end{cases}
\intertext{and}
J\big(\wt\ch\big) & \sceq 
\begin{cases}
(-\infty, a) & \text{if $\pr(\wt\ch)\subseteq \{ z\in \h\mid \Rea z \geq a\}$,}
\\
(a,\infty) & \text{if $\pr(\wt\ch) \subseteq \{ z\in \h \mid \Rea z \leq a\}$.}
\end{cases}
\end{align*}
\end{defi}

We note that the combination of Remark~\ref{remain} with Propositions~\ref{nccells}\eqref{ncc1} and \ref{ncSH} resp.\@ with Propositions~\ref{ccells}\eqref{cc1} and \ref{ccSH} resp.\@ Proposition~\ref{stripSH} shows that indeed each $\wt\ch\in\wt\fch_{\choices,\shmap}$ gets assigned a pair $\big( I(\wt\ch), J(\wt\ch) \big)$ of intervals.

\begin{lemma}\label{char_intervals}
Let $\wt\ch\in\wt\fch_{\choices,\shmap}$. For each $v\in\CS'(\wt\ch)$ let $\gamma_v$ denote the geodesic on $\h$ determined by $v$. If $v\in \CS'(\wt\ch)$, then $(\gamma_v(\infty), \gamma_v(-\infty))\in I(\wt\ch)\times J(\wt\ch)$. Conversely, if $(x,y)\in I(\wt\ch)\times J(\wt\ch)$, then there exists a unique element $v$ in $\CS'(\wt\ch)$ such that $(\gamma_v(\infty), \gamma_v(-\infty)) = (x,y)$.
\end{lemma}

Let $\wt\ch\in\wt\fch_{\choices,\shmap}$ and $g\in\Gamma$. Suppose that $I(\wt\ch) =(a,\infty)$. Then
\begin{align*}
g.I\big(\wt\ch\big) & = 
\begin{cases}
(g.a,g.\infty) & \text{if $g.a<g.\infty$,}
\\
(g.a,\infty] \cup (-\infty, g.\infty) & \text{if $g.\infty < g.a$,}
\end{cases}
\intertext{and}
g.J\big(\wt\ch\big) & =
\begin{cases}
(g.a,\infty] \cup (-\infty, g.\infty) & \text{if $g.a < g.\infty$,}
\\
(g.a,g.\infty) &\text{if $g.\infty < g.a$.}
\end{cases}
\end{align*}
Here, the interval $(b,\infty]$ denotes the union of the interval $(b,\infty)$ with the point $\infty\in \bhg \h$. 
Hence, the set $I\sceq (b,\infty]\cup (-\infty, c)$ is connected as a subset of $\bhg \h$. The interpretation of $I$ is more eluminating in the ball model: Via the Cayley transform $\mc C$ the set $\bhg \h$ is homeomorphic to the unit sphere $S^1$. Let $b'\sceq \mc C(b)$, $c'\sceq \mc C(c)$ and $I'\sceq \mc C(I)$. Then $I'$ is the connected component of $S^1\mminus\{b',c'\}$ which contains $\mc C(\infty)$.

Suppose now that $I(\wt\ch) = (-\infty, a)$. Then 
\begin{align*}
g.I\big(\wt\ch\big) & =
\begin{cases}
(-\infty, g.a) \cup (g.(-\infty), \infty] & \text{if $g.a<g.(-\infty)$,}
\\
(g.(-\infty), g.a) & \text{if $g.(-\infty)<g.a$,}
\end{cases}
\intertext{and}
g.J\big(\wt\ch\big) & =
\begin{cases}
( g.(-\infty), g.a) & \text{if $g.a < g.(-\infty)$,}
\\
(-\infty, g.a) \cup (g.(-\infty), \infty] & \text{if $g.(-\infty)<g.a$.}
\end{cases}
\end{align*}
Note that for $g=\textbmat{\alpha}{\beta}{\gamma}{\delta}$ we have
\[
g.(-\infty) = \lim_{t\searrow -\infty} \frac{\alpha t + \beta}{\gamma t + \delta} = \lim_{s\nearrow 0} \frac{\alpha + \beta s}{\gamma + \delta s} = \lim_{s\searrow 0} \frac{\alpha + \beta s}{\gamma + \delta s} = g.\infty.
\]
In particular, $\id.(-\infty) = \infty$.

Let $a,b\in\overline \R$. For abbreviation we set  $(a,b)_+\sceq (\min(a,b),\max(a,b))$ and $(a,b)_-\sceq (\max(a,b),\infty] \cup (-\infty, \min(a,b))$.

\begin{prop}\label{adjacentSH}
Let $\wt\ch\in\wt\fch_{\choices,\shmap}$ and suppose that $S$ is a side of $\pr(\wt\ch)$. Then there exist exactly two pairs $(\wt\ch_1,g_1),(\wt\ch_2,g_2)\in \wt\fch_{\choices,\shmap}\times\Gamma$ such that $S=g_j.b(\wt\ch_j)$. Moreover, $g_1. \cl( \pr(\wt\ch_1) ) = \cl(\pr(\wt\ch))$ and $g_2.\cl(\pr(\wt\ch_2))\cap \cl(\pr(\wt\ch)) = S$ or vice versa. The union $g_1.\CS'(\wt\ch_1) \cup g_2.\CS'(\wt\ch_2)$ is disjoint and equals the set of all unit tangent vectors in $\CS$ that are based on $S$. Let $a,b\in\bhg \h$ be  the endpoints of $S$. Then $g_1.I(\wt\ch_1)\times g_1.J(\wt\ch_1) = (a,b)_+ \times (a,b)_-$ and $g_2.I(\wt\ch_2)\times g_2.J(\wt\ch_2) = (a,b)_-\times (a,b)_+$ or vice versa.
\end{prop}

\begin{proof}
Let $D'$ denote the set of unit tangent vectors in $\CS$ that are based on $S$. By Lemma~\ref{propsSH}, $S$ is a connected component of $\BS$. Hence $D'$ is the set of unit tangent vectors based on $S$ but not tangent to $S$. The complete geodesic segment $S$ divides $\h$ into two open half-spaces $H_1,H_2$ such that $\h$ is the disjoint union $H_1\cup S\cup H_2$. Moreover, $\pr(\wt\ch)^\circ$ is contained in $H_1$ or $H_2$, say $\pr(\wt\ch)^\circ\subseteq H_1$. Then $D'$ decomposes into the disjoint union $D'_1\cup D'_2$ where $D'_j$ denotes the non-empty set of elements in $D'$ that point into $H_j$. For $j=1,2$ pick $v_j\in D'_j$. Since $\CS'(\wt\fch_{\choices,\shmap})$ is a set of representatives for $\wh\CS=\pi(\CS)$ (see Proposition~\ref{CS=CShat}), there exists a uniquely determined pair $(\wt\ch_j, g_j)\in\wt\fch_{\choices,\shmap}\times\Gamma$ such that $v_j\in g_j.\CS'(\wt\ch_j)$. We will show that $S=g_j.b(\wt\ch_j)$. Assume for contradiction that $S\not=g_j.b(\wt\ch_j)$. Since $S$ and $g_j.b(\wt\
ch_j)$ are complete geodesic segments, the intersection of $S$ and $g_j.b(\wt\ch_j)$ in $\pr(v_j)$ is transversal. Recall that $S\subseteq\partial\pr(\wt\ch)$ and $b(\wt\ch_j)\subseteq\partial\pr(\wt\ch_j)$ and that $\partial\pr(\wt\ch')^\circ = \partial\pr(\wt\ch')$ for each $\wt\ch'\in\wt\fch_{\choices,\shmap}$. Then there exists $\eps>0$ such that $B_\eps(\pr(v_j))\cap \pr(\wt\ch)^\circ = B_\eps(\pr(v_j))\cap H_1$ and
\[
 B_\eps\big(\pr(v_j)\big) \cap g_j.\pr\big(\wt\ch_j\big)^\circ \cap H_1 \not=\emptyset.
\]
Hence $\pr(\wt\ch)^\circ\cap g_j.\pr(\wt\ch_j)^\circ \not=\emptyset$. Proposition~\ref{glues} in combination with Remark~\ref{remain} shows that $\cl(\pr(\wt\ch)) = g_j.\cl(\pr(\wt\ch_j))$. But then 
\[
\partial\pr(\wt\ch) = g_j.\partial\pr(\wt\ch_j),
\]
which implies that $S=g_j.b(\wt\ch_j)$. This is a contradiction, and hence $S=g_j.b(\wt\ch_j)$. Then Lemma~\ref{char_intervals} implies that $g_j.I(\wt\ch_j)\times g_j.J(\wt\ch_j)$ equals $(a,b)_+\times (a,b)_-$ or $(a,b)_-\times (a,b)_+$. On the other hand 
\[
\bhg H_1 \times \bhg H_2 = \big\{ \big(\gamma_v(\infty), \gamma_v(-\infty)\big) \ \big\vert\ v\in D'_1\big\} = \big\{ \big(\gamma_v(-\infty),\gamma_v(\infty)\big) \ \big\vert\ v\in D'_2\big\}
\]
equals $(a,b)_+\times (a,b)_-$ or $(a,b)_-\times (a,b)_+$. Therefore, again by Lemma~\ref{char_intervals}, we have $g_j.\CS'(\wt\ch_j)=D'_j$. This shows that the union $g_1.\CS'(\wt\ch_1)\cup g_2.\CS'(\wt\ch_2)$ is disjoint and equals $D'$. 

We have  $\cl(\pr(\wt\ch))\subseteq \overline H_1$ and $g_1.\cl(\pr(\wt\ch_1))\subseteq \overline H_1$ with $S\subseteq \partial\pr(\wt\ch) \cap g_1.\partial \pr(\wt\ch_1)$. Let $z\in S$. Then there exists $\eps>0$ such that 
\[
 B_\eps(z) \cap \pr\big(\wt\ch\big)^\circ = B_\eps(z) \cap H_1 = B_\eps(z) \cap g_1.\pr\big(\wt\ch_1\big)^\circ.
\]
Hence $\pr(\wt\ch)^\circ \cap g_1.\pr(\wt\ch_1)^\circ\not=\emptyset$. As above we find that $\cl(\pr(\wt\ch)) = g_1.\cl(\pr(\wt\ch_1))$. Finally, $g_2.\cl(\pr(\wt\ch_2))\subseteq \overline H_2$ with 
\[
 S\subseteq g_2.\cl\big(\pr\big(\wt\ch_2\big)\big) \cap \overline H_1 \subseteq \overline H_2 \cap \overline H_1 = S.
\]
Hence $\cl(\pr(\wt\ch))\cap g_2.\cl(\pr(\wt\ch_2)) = S$.
\end{proof}

Let $\wt\ch\in\wt\fch_{\choices,\shmap}$ and suppose that $S$ is a side of $\pr(\wt\ch)$. Let $(\wt\ch_1,g_1), (\wt\ch_2,g_2)$ be the two elements in $\wt\fch_{\choices,\shmap}\times\Gamma$ such that $S=g_j.b(\wt\ch_j)$ and $g_1.\cl(\pr(\wt\ch_1)) = \cl(\pr(\wt\ch))$ and  $g_2.\cl(\pr(\wt\ch_2))\cap \cl(\pr(\wt\ch))=S$. Then we define
\[
 p\big(\wt\ch,S\big) \sceq \big(\wt\ch_1,g_1\big)\quad\text{and}\quad n\big(\wt\ch,S\big)\sceq \big(\wt\ch_2,g_2\big).
\]

\begin{remark}\label{sides_eff}
Let $\wt\ch\in\wt\fch_{\choices,\shmap}$ and suppose that $S$ is a side of $\pr(\wt\ch)$. We will show how one effectively finds the elements $p(\wt\ch,S)$ and $n(\wt\ch,S)$. Let 
\[
\big(\wt\ch_1,k_1\big)\sceq p\big(\wt\ch,S\big) \quad\text{and}\quad \big(\wt\ch_2,k_2\big)\sceq n\big(\wt\ch,S\big).
\]
Suppose that $\wt\ch'$ is the (unique) element in $\wt\fch_\choices$ such that $\shmap(\wt\ch').\wt\ch'=\wt\ch$ and suppose further that $\wt\ch'_j\in\wt\fch_\choices$ such that $\shmap(\wt\ch'_j).\wt\ch'_j = \wt\ch_j$ for $j=1,2$. Set $S'\sceq \shmap(\wt\ch')^{-1}.S$. Then $S'$ is a side of $\pr(\wt\ch')$. For $j=1,2$ we have
\[
S'=\shmap\big(\wt\ch'\big)^{-1}.S = \shmap\big(\wt\ch'\big)^{-1} k_j .b\big(\wt\ch_j\big) = \shmap\big(\wt\ch'\big)^{-1} k_j \shmap\big(\wt\ch'_j\big).b\big(\wt\ch'_j\big)
\]
and
\[
 k_j.\cl\big(\pr\big(\wt\ch_j\big)\big) = k_j\shmap\big(\wt\ch'_j\big).\cl\big(\pr\big(\wt\ch'_j\big)\big).
\]
Moreover, $\cl(\pr(\wt\ch)) = \shmap(\wt\ch').\cl(\pr(\wt\ch'))$. Then $k_1.\cl(\pr(\wt\ch_1)) = \cl(\pr(\wt\ch))$ is equivalent to
\[
\shmap\big(\wt\ch'\big)^{-1}k_1\shmap\big(\wt\ch'_1\big).\cl\big(\pr\big(\wt\ch'_1\big)\big) = \cl\big(\pr\big(\wt\ch'\big)\big),
\]
and $k_2.\cl(\pr(\wt\ch_2))\cap\cl(\pr(\wt\ch)) = S$ is equivalent to
\[
\shmap\big(\wt\ch'\big)^{-1}k_2\shmap\big(\wt\ch'_2\big).\cl\big(\pr\big(\wt\ch'_2\big)\big) \cap \cl\big(\pr\big(\wt\ch'\big)\big) = S'.
\]
Therefore, $(\wt\ch_1,k_1) = p(\wt\ch,S)$ if and only if $(\wt\ch'_1,\shmap(\wt\ch')^{-1}k_1\shmap(\wt\ch'_1)) = p(\wt\ch',S')$, and $(\wt\ch_2,k_2) = n(\wt\ch,S)$ if and only if $(\wt\ch'_2,\shmap(\wt\ch')^{-1}k_2\shmap(\wt\ch'_2)) = n(\wt\ch',S')$.

By Corollary~\ref{iscellH}, the sets $\ch'\sceq \cl(\pr(\wt\ch'))$ and $\ch'_j\sceq \cl(\pr(\wt\ch'_j))$ are $\fpch$-cells in $\h$. Suppose first that $\ch'$ arises from the non-cuspidal basal precell $\pch'$ in $\h$. Then there is a unique element $(\pch, h_\pch)\in \choices$ such that for some $h\in\Gamma$ the pair $(\pch',h)$ is contained in the cycle in $\fpch\times\Gamma$ determined by $(\pch, h_\pch)$. Necessarily, $\pch$ is non-cuspidal. Let $\big( (\pch_j, h_j) \big)_{j=1,\ldots, k}$ be the cycle in $\fpch\times\Gamma$ determined by $(\pch,h_\pch)$. Then $\pch'=\pch_m$ for some $m\in\{1,\ldots, \cyl(\pch)\}$ and hence $\ch'=\ch(\pch_m)$ and $\wt\ch' = \wt\ch_m(\pch,h_\pch)$. For $j=1,\ldots, k$ set $g_1\sceq\id$ and $g_{j+1}\sceq h_jg_j$. Proposition~\ref{nccells}\eqref{ncc3} states that $\ch(\pch_m)=g_m.\ch(\pch)$ and Proposition~\ref{nccells}\eqref{ncc1} shows that $S'$ is the geodesic segment $[g_mg_j^{-1}.\infty, g_mg_{j+1}^{-1}.\infty]$ for some $j\in \{1,\ldots, k\}$. Then $g_jg_m^{-1}.S'=[\
infty, h_j^{-1}.\infty]$.  Let $n\in\{1,\ldots, \cyl(\pch)\}$ such that $n\equiv j \mod \cyl(\pch)$. Then $h_n=h_j$ by Lemma~\ref{cyclic}. Proposition~\ref{ncSH} shows that 
$b(\wt\ch_n(\pch, h_\pch)) = [\infty, h_n^{-1}.\infty] = g_jg_m^{-1}.S'$. We claim that $(\wt\ch_j(\pch,h_\pch), g_mg_j^{-1}) = p(\wt\ch',S')$. For this is remains to show that $g_mg_j^{-1}.\cl(\pr(\wt\ch_j(\pch,h_\pch))) = \cl(\pr(\wt\ch'))$. Proposition~\ref{ncSH} shows that $\cl(\pr(\wt\ch_j(\pch, h_\pch))) = \ch(\pch_n)$ and Lemma~\ref{cyclic} implies that $\ch(\pch_n) = \ch(\pch_j)$. Let $v$ be the vertex of $\mc K$ to which $\pch$ is attached. Then $g_j.v\in \ch(\pch_j)^\circ$ and $g_mg_j^{-1}g_j.v=g_m.v\in\ch(\pch_m)^\circ$. Therefore $g_mg_j^{-1}.\ch(\pch_j)^\circ \cap \ch(\pch_m)^\circ \not=\emptyset$. From Proposition~\ref{glues}\eqref{glue_nccells} it follows that $g_mg_j^{-1}.\ch(\pch_j) = \ch(\pch_m)$. Recall that $\ch(\pch_m) = \cl(\pr(\wt\ch'))$. Hence $( \wt\ch_j(\pch,h_\pch), g_mg_j^{-1}) = p(\wt\ch',S')$ and 
\[
\big( \shmap\big(\wt\ch_j(\pch,h_\pch)\big) .\wt\ch_j\big(\pch, h_\pch\big), \shmap\big(\wt\ch'\big) g_mg_j^{-1} \shmap\big( \wt\ch_j(\pch, h_\pch)\big)^{-1} \big) = p\big(\wt\ch, S\big).
\]
Analogously one proceeds if $\pch'$ is cuspidal or a strip precell. 

Now we show how one determines $(\wt\ch_2,k_2)$. Suppose again that $\ch'$ arises from the non-cuspidal basal precell $\pch'$ in $\h$. We use the notation from the determination of $p(\wt\ch',S')$. By Corollary~\ref{props_preH} there is a unique pair $(\wh\pch,s)\in\fpch\times\Z$ such that $b(\wt\ch_n(\pch, h_\pch)) \cap t_\lambda^s.\wh\pch \not=\emptyset$ and $t_\lambda^s.\wh\pch\not=\pch_n$. Then $t_\lambda^{-s}g_jg_m^{-1}.S'$ is a side of the cell $\ch(\wh\pch)$ in $\h$. As before, we determine $(\wt\ch_3,k_3)\in\wt\fch_\choices\times\Gamma$ such that $k_3.b(\wt\ch_3) = t_\lambda^{-s}g_jg_m^{-1}.S'$ and $k_3.\cl(\pr(\wt\ch_3)) = \ch(\wh\pch)$. Recall that $g_jg_m^{-1}.S'$ is a side of $\ch(\pch_j) = \ch(\pch_n)$. We have
\begin{align*}
g_mg_j^{-1}t_\lambda^s k_3. \cl\big(\pr\big(\wt\ch_3\big)\big) \cap \cl\big(\pr\big(\wt\ch'\big)\big) & = g_mg_j^{-1}t_\lambda^s. \ch(\wh\pch) \cap \ch(\pch_m) 
\\
& = g_mg_j^{-1}t_\lambda^s. \ch(\wh\pch) \cap g_mg_j^{-1}.\ch(\pch_j) 
\\
& = g_mg_j^{-1}. \big( t_\lambda^s. \ch(\wh\pch) \cap \ch(\pch_j) \big) 
\\
& = g_mg_j^{-1}. \big( t_\lambda^s. \ch(\wh\pch) \cap \ch(\pch_n) \big) 
\\
& = g_mg_j^{-1}g_jg_m^{-1}.S'
\\
& = S'.
\end{align*}
Thus $n(\wt\ch',S') = (\wt\ch_3, g_mg_j^{-1}t_\lambda^s k_3)$ and 
\[
n\big(\wt\ch, S\big) = \big( \shmap\big(\wt\ch_3\big).\wt\ch_3, \shmap\big(\wt\ch'\big)g_mg_j^{-1}t_\lambda^s k_3 \shmap\big(\wt\ch_3\big)^{-1}\big).
\]
If $\ch'$ arises from a cuspidal or strip precell in $\h$, then the construction of $n(\wt\ch,S)$ is analogous.
\end{remark}

\begin{prop}\label{CS2}
Let $\wh\gamma$ be a geodesic on $Y$ and suppose that $\wh\gamma$ intersects $\wh\CS$ in $\wh\gamma'(t_0)$. Let $\gamma$ be the unique geodesic on $\h$ such that  $\gamma'(t_0)\in \CS'(\wt\fch_{\choices,\shmap})$ and $\pi(\gamma'(t_0))=\wh\gamma'(t_0)$. Let $\wt\ch\in\wt\fch_{\choices,\shmap}$ be the unique shifted cell in $S\h$ for which we have $\gamma'(t_0)\in \CS'(\wt\ch)$. 
\begin{enumerate}[{\rm (i)}]
\item\label{CS2i} There is a next point of intersection of $\gamma$ and $\CS$ if and only if $\gamma(\infty)$ does not belong to $\bhg\pr(\wt\ch)$.
\item\label{CS2ii} Suppose that $\gamma(\infty)\notin\bhg\pr(\wt\ch)$. Then there is a unique side $S$ of $\pr(\wt\ch)$ intersected by $\gamma( (t_0,\infty) )$. Suppose that $(\wt\ch_1,g) = n(\wt\ch,S)$. The next point of intersection is on $g.\CS'(\wt\ch_1)$.
\item\label{CS2iii} Let $(\wt\ch',h) = n(\wt\ch, b(\wt\ch))$. Then there was a previous point of intersection of $\gamma$ and $\CS$ if and only if $\gamma(-\infty)\notin h.\bhg\pr(\wt\ch')$.
\item\label{CS2iv} Suppose that $\gamma(-\infty)\notin h.\bhg\pr(\wt\ch')$. Then there is a unique side $S$ of $h.\pr(\wt\ch')$ intersected by $\gamma( (-\infty, t_0) )$. Let $(\wt\ch_2, h^{-1}k) = p(\wt\ch',h^{-1}S)$. Then the previous point of intersection was on $k.\CS'(\wt\ch_2)$.
\end{enumerate}
\end{prop}

\begin{proof}
We start by proving \eqref{CS2i}. Recall from Remark~\ref{charinter} that there is a next point of intersection of $\gamma$ and $\CS$ if and only if $\gamma( (t_0,\infty) )$ intersects $\BS$. Since $\gamma'(t_0)\in \CS'(\wt\ch)$, Proposition~\ref{ncSH} resp.\@ \ref{ccSH} resp.\@ \ref{stripSH} in combination with Remark~\ref{remain} shows that $\gamma'(t_0)$ points into $\pr(\wt\ch)^\circ$. Lemma~\ref{propsSH} states that $\pr(\wt\ch)^\circ \cap \BS =\emptyset$ and $\partial\pr(\wt\ch)\subseteq \BS$. Hence $\gamma( (t_0,\infty))$ does not intersect $\BS$ if and only if $\gamma( (t_0,\infty) )\subseteq \pr(\wt\ch)^\circ$. In this case, 
\[
\gamma(\infty)\in \chg( \pr(\wt\ch)) \cap \bhg H = \bhg \pr(\wt\ch).
\]
Conversely, if $\gamma(\infty) \in \bhg\pr(\wt\ch)$, then $\gamma( (t_0,\infty) )\subseteq \pr(\wt\ch)^\circ$ or $\gamma( (t_0,\infty) )\subseteq \partial\pr(\wt\ch)$. In the latter case, Lemma~\ref{propsSH} shows that $\gamma( (t_0,\infty) )\subseteq\BS$. Hence, if $\gamma(\infty)\in\bhg\pr(\wt\ch)$, then $\gamma( (t_0,\infty) )\subseteq \pr(\wt\ch)^\circ$.

Suppose now that $\gamma(\infty)\notin\bhg\pr(\wt\ch)$. The previous argument shows that the geodesic segment $\gamma( (t_0,\infty) )$ intersects $\partial\pr(\wt\ch)$, say $\gamma(t_1)\in\partial\pr(\wt\ch)$ with $t_1\in (t_0,\infty)$. If there was an element $t_2\in (t_0,\infty)\mminus\{t_1\}$ with $\gamma(t_2)\in\partial\pr(\wt\ch)$, then there is a side $S$ of $\pr(\wt\ch)$ such that $\gamma(\R)=S$, where the equality follows from the fact that $S$ is a complete geodesic segment. But then, by Lemma~\ref{propsSH}, $\gamma(\R)\subseteq \BS$, which contradicts to $\gamma'(t_0)\in \CS$. Thus, $\gamma(t_1)$ is the only intersection point of $\partial\pr(\wt\ch)$ and $\gamma( (t_0,\infty) )$. Since $\gamma( (t_0,t_1) )\subseteq \pr(\wt\ch)^\circ$, $\gamma'(t_0)$ is the next point of intersection of $\gamma$ and $\CS$. Moreover, $\gamma'(t_1)$ points out of $\pr(\wt\ch)$, since otherwise $\gamma( (t_1,\infty) )$ would intersect $\partial\pr(\wt\ch)$ which would lead to a contradiction as before.  Proposition~\
ref{CS=CShat} states that there is a unique pair $(\wt\ch_1,g)\in\wt\fch_{\choices,\shmap}\times\Gamma$ such that $\gamma'(t_1)\in g.\CS'(\wt\ch_1)$. Then $\gamma'(t_1)$ points into $g.\pr(\wt\ch_1)^\circ$. Let $S$ be the side of $\pr(\wt\ch)$ with $\gamma(t_1)\in S$. By Proposition~\ref{adjacentSH}, either $g.\cl(\pr(\wt\ch_1)) = \cl(\pr(\wt\ch))$ or $g.\cl(\pr(\wt\ch_1))\cap \cl(\pr(\wt\ch)) = S$. In the first case, $\gamma'(t_1)$ points into $\pr(\wt\ch)^\circ$, which is a contradiction. Therefore 
\[
g.\cl(\pr(\wt\ch_1))\cap \cl(\pr(\wt\ch)) = S,
\]
which shows that $(\wt\ch_1,g) = n(\wt\ch,S)$. This completes the proof of \eqref{CS2ii}.

Let $(\wt\ch',h)=n(\wt\ch, b(\wt\ch))$. Since $\gamma(t_0) \in b(\wt\ch)$ and $\gamma'(t_0)\in \CS'(\wt\ch)$, Proposition~\ref{adjacentSH} implies that $\gamma(t_0)\in h.b(\wt\ch')$ and $\gamma'(t_0) \notin h.\CS'(\wt\ch')$. Since  $\gamma(\R)\not\subseteq h.\partial\pr(\wt\ch')$, the unit tangent vector $\gamma'(t_0)$ points out of $\pr(\wt\ch')$. Because the intersection of $\gamma(\R)$ and $h.b(\wt\ch')$ is transversal and $\pr(\wt\ch')$ is a convex polyhedron with non-empty interior, $\gamma( (t_0-\eps,t_0) )\cap h.\pr(\wt\ch)^\circ \not= \emptyset$ for each $\eps>0$. As before we find that there was a previous point of intersection of $\gamma$ and $\CS$ if and only if $\gamma( (-\infty, t_0) )$ intersects $\partial\pr(\wt\ch')$ and that this is the case if and only if $\gamma(-\infty) \notin h.\bhg\pr(\wt\ch')$. 

Suppose that $\gamma(-\infty)\notin h.\bhg\pr(\wt\ch')$. As before, there is a unique $t_{-1}\in (-\infty, t_0)$ such that $\gamma(t_{-1})\in h.\bhg\pr(\wt\ch')$. Let $S$ be the side of $h.\pr(\wt\ch')$ with $\gamma(t_{-1})\in S$. Necessarily, $\gamma( (t_{-1},t_0 ) \subseteq h.\pr(\wt\ch')^\circ$, which shows that $\gamma'(t_{-1})$ points into $h.\pr(\wt\ch')^\circ$ and that $\gamma'(t_{-1})$ is the previous point of intersection. Let $(\wt\ch_2, k) \in \wt\fch_{\choices,\shmap}\times\Gamma$ be the unique pair such that $\gamma'(t_{-1})\in k.\CS'(\wt\ch_2)$ (see Proposition~\ref{CS=CShat}). By Proposition~\ref{adjacentSH}, we have either $k.\cl(\pr(\wt\ch_2)) = h.\cl(\pr(\wt\ch'))$ or $k.\cl(\pr(\wt\ch_2)) \cap h.\cl(\pr(\wt\ch')) = S$. In the latter case, $\gamma'(t_{-1})$ points out of $h.\pr(\wt\ch')^\circ$ which is a contradiction. Hence $h^{-1}k.\cl(\pr(\wt\ch_2) = \cl(\pr(\wt\ch'))$, which shows that $(\wt\ch_2,h^{-1}k) = p(\wt\ch',h^{-1}S)$.
\end{proof}

\begin{cor}\label{lastinter}
Let $\wh\gamma$ be a geodesic on $Y$ and suppose that $\wh\gamma$ does not intersect $\wh\CS$ infinitely often in future. If $\wh\gamma$ intersects $\wh\CS$ at all, then there exists $t\in\R$ such that $\wh\gamma'(t)\in\wh\CS$ and $\wh\gamma( (t,\infty) )\cap \wh\BS =\emptyset$. Analogously, suppose that $\wh\eta$ is a geodesic on $Y$ which does not intersect $\wh\CS$ infinitely often in past. If $\wh\eta$ intersects $\wh\CS$ at all, then there exists $t\in\R$ such that $\wh\eta'(t)\in\wh\CS$ and $\wh\eta( (-\infty, t) ) \cap \wh\BS = \emptyset$.
\end{cor}

\begin{example}\label{intersectionGamma05}
Recall the setting of Example~\ref{choicesGamma05}. We consider the two shift maps $\shmap_1 \equiv \id$, and
\[
\shmap_2\big(\wt\ch_1\big) \sceq \mat{1}{-1}{0}{1}\quad\text{and}\quad \shmap_2\big(\wt\ch_j\big) \sceq \id \quad\text{for $j=2,\ldots, 6$.}
\]
For simplicity set $\wt\ch_{-1}\sceq \shmap_2\big(\wt\ch_1\big). \wt\ch_1$ and $\CS'_{-1}\sceq \shmap_2\big(\wt\ch_1\big).\CS'_1$. Further we set
\begin{align*}
g_1 & \sceq \mat{1}{0}{5}{1}, & g_2 & \sceq \mat{2}{-1}{5}{-2}, & g_3 & \sceq \mat{3}{-2}{5}{-3}, & g_4 & \sceq \mat{4}{-1}{5}{-1},
\\
g_5 & \sceq \mat{4}{-5}{5}{-6}, & g_6 & \sceq \mat{1}{1}{0}{1},  & g_7 & \sceq \mat{-1}{0}{5}{-1}.
\end{align*}
Figure~\ref{forward1} shows the translates of the sets $\CS'_j$ which are necessary to determine the location of the next point of intersection if the shift map is $\shmap_1$, and Figure~\ref{forward2} those if $\shmap_2$ is the chosen shift map.
\begin{figure}[h]
\begin{center}
\includegraphics*{Gamma05.10} 
\end{center}
\caption{The translates of $\CS'$ relevant for determination of the location next point of intersection for the shift map $\shmap_1$.}\label{forward1}
\end{figure}
\begin{figure}[h]
\begin{center}
\includegraphics*{Gamma05.5} 
\end{center}
\caption{The translates of $\CS'$ relevant for determination of the location next point of intersection for the shift map $\shmap_2$.}\label{forward2}
\end{figure}
\end{example}

Recall the set $\bd$ from Section~\ref{sec_base}. The following characterization is now obvious.

\begin{prop}\label{CS3}
Let $\widehat\gamma$ be a geodesic on $Y$. 
\begin{enumerate}[{\rm (i)}]
\item\label{CS3i} $\widehat\gamma$ intersects $\widehat\CS$ infinitely often in future if and only if $\widehat\gamma(\infty) \notin \pi(\bd)$.
\item\label{CS3ii} $\widehat\gamma$ intersects $\widehat\CS$ infinitely often in past if and only if $\widehat\gamma(-\infty) \notin \pi(\bd)$.
\end{enumerate}
\end{prop}

Recall the set $\NIC$ from Remark~\ref{outlook}.

\begin{thm}\label{geomcross}
Let $\mu$ be a measure on the space of geodesics on $Y$. Then $\wh\CS$ is a cross section \wrt $\mu$ for the geodesic flow on $Y$ if and only if $\mu(\NIC) = 0$.
\end{thm}

\begin{proof}
Proposition~\ref{gcs1} shows that $\wh\CS$ satisfies (\apref{C}{C2}{}). Let $\wh\gamma$ be a geodesic on $Y$. Then Proposition~\ref{CS3} implies that $\wh\gamma$ intersects $\wh\CS$ infinitely often in past and future if and only if $\wh\gamma\notin\NIC$.  This completes the proof.
\end{proof}

Let $\mc E$ denote the set of unit tangent vectors to the geodesics in $\NIC$  and set $\wh\CS_\st \sceq \wh\CS \mminus \mc E$.

\begin{cor} Let $\mu$ be a measure on the space of geodesics on $Y$ such that $\mu(\NIC) = 0$. Then $\wh\CS_\st$ is the maximal strong cross section \wrt $\mu$ contained in $\wh\CS$.
\end{cor}

\subsection{Geometric coding sequences and geometric symbolic dynamics}\label{sec_geomcodseq}

A \textit{label} of a unit tangent vector in $\wh\CS$ or $\CS$ is a symbol which is assigned to this vector. The \textit{labeling} 
of $\wh\CS$ resp.\@ $\CS$ is the assigment of the labels to its elements. The set of labels is commonly called the \textit{alphabet} of the arising symbolic dynamics. 

We establish a labeling of $\wh\CS$ and $\CS$ in the following way: Let $v\in\CS'(\wt\fch_{\choices,\shmap})$ and suppose that $\wt\ch\in\wt\fch_{\choices,\shmap}$ is the unique shifted cell in $S\h$ such that $v\in\CS'(\wt\ch)$. Let $\gamma$ be the geodesic on $\h$ determined by $v$. 

Suppose first that $\gamma(\infty)\notin \bhg\pr(\wt\ch)$. Proposition~\ref{CS2}\eqref{CS2ii} states that there is a unique side $S$ of $\pr(\wt\ch)$ intersected by $\gamma( (0,\infty) )$ and that the next point of intersection of $\gamma$ and $\CS$ is on $g.\CS'(\wt\ch_1)$ if $(\wt\ch_1,g)=n(\wt\ch,S)$. We assign to $v$ the label $(\wt\ch_1,g)$.

Suppose now that $\gamma(\infty)\in \bhg\pr(\wt\ch)$. Proposition~\ref{CS2}\eqref{CS2i} shows that there is no next point of intersection of $\gamma$ and $\CS$. Let $\eps$ be an abstract symbol which is not contained in $\Gamma$. Then we label $v$ by $\eps$ (``end'' or ``empty word''). 

Let $\wh v\in\wh\CS$. By Proposition~\ref{CS=CShat} there is a unique $v\in\CS'(\wt\fch_{\choices,\shmap})$ such that $\pi(v) = \wh v$. We endow $\wh v$ and each element in $\pi^{-1}(\wh v)$ with the labels of $v$.

The following proposition is the key result for the determination of the set of labels.

\begin{prop}\label{char_bound}
Let $\wt\ch\in\wt\fch_{\choices,\shmap}$ and suppose that $S_j$, $j=1,\ldots,k$, are the sides of $\pr(\wt\ch)$. For $j=1,\ldots, k$ set $(\wt\ch_j,g_j)\sceq n(\wt\ch, S_j)$. Let $v\in \CS'(\wt\ch)$ and suppose that $\gamma$ is the geodesic determined by $v$. Then $\gamma(\infty) \in g_j.I(\wt\ch_j)$ if and only if $\gamma( (0,\infty) )$ intersects $S_j$. Moreover, if $S_j\not= b(\wt\ch)$, then $g_j.I(\wt\ch_j) \subseteq I(\wt\ch)$. If $S_j=b(\wt\ch)$, then $g_j.I(\wt \ch_j) = J(\wt\ch)$.
\end{prop}

\begin{proof}
This follows from Propositions~\ref{adjacentSH} and \ref{CS2} and Lemma~\ref{char_intervals} with similar arguments as in the proof of Proposition~\ref{adjacentSH}. 
\end{proof}

For $\wt\ch\in\wt\fch_{\choices,\shmap}$ let $\Sides(\wt\ch)$ denote the set of sides of $\pr(\wt\ch)$.

\begin{cor}\label{decomposition}
Let $\wt\ch\in\wt\fch_{\choices,\shmap}$. For each $S\in\Sides(\wt\ch)$ set $(\wt\ch_S,g_S)\sceq n(\wt\ch,S)$. Then $I(\wt\ch)$ is the disjoint union 
\[
I(\wt\ch) = \left( I(\wt\ch) \cap \bhg\pr(\wt\ch) \right) \cup \bigcup_{S\in\Sides(\wt\ch)\mminus b(\wt\ch)} g_S.I(\wt\ch_S).
\]
\end{cor}

Let $\Sigma$ denote the set of labels.

\begin{cor}\label{prop_labels}
The set $\Sigma$ of labels is given by
\begin{align*}
\Sigma & = \left\{\eps\right\} \cup 
\\
& \quad\cup\left\{ (\wt\ch,g)\in\wt\fch_{\choices,\shmap}\times\Gamma \left\vert\ \exists\, \wt\ch'\in\wt\fch_{\choices,\shmap}\ \exists\, S\in\Sides(\wt\ch')\mminus b(\wt\ch')\colon (\wt\ch,g)=n(\wt\ch',S)\right.\right\}.
\end{align*}
Moreover, $\Sigma$ is finite.
\end{cor}

\begin{proof}
Note that for each $\wt\ch'\in\wt\fch_{\choices,\shmap}$ we have $\bhg\pr(\wt\ch')\cap I(\wt\ch')\not=\emptyset$. Thus, $\eps$ is a label. Then the claimed equality follows immediately from Corollary~\ref{decomposition}. Since $\wt\fch_{\choices,\shmap}$ is finite and each shifted cell in $S\h$ has only finitely many sides, $\Sigma$ is finite.
\end{proof}

\begin{example}\label{labelsHecke}
For the Hecke triangle group $G_n$ let $\fpch = \{\pch\}$, $\choices = \{ (\pch, U_n)\}$ and $\shmap \equiv \id$. Then $\wt\fch_{\choices,\shmap} = \{ \wt\ch_1(\pch, U_n) \}$. Set $\wt\ch\sceq \wt\ch_1(\pch,U_n)$.  Then the set of labels is (see Figure~\ref{nextlastHecke})
\[
\Sigma = \left\{ \eps, \big(\wt\ch, U_n^kS\big)\left\vert\  k\in\{1,\ldots, n-1\} \vphantom{\big(\wt\ch, U_n^kS\big)}  \right.\right\}.
\]
\end{example}

\begin{example}\label{labelsGamma05}
Recall Example~\ref{intersectionGamma05}. If the shift map is $\shmap_1$, then the set of labels is
\begin{align*}
\Sigma = \big\{ \eps,  \big(\wt\ch_2,g_1\big), \big(\wt\ch_4,\id\big), \big(\wt\ch_5, g_2\big), \big(\wt\ch_6, \id\big), \big(\wt\ch_3,g_3\big), &\big(\wt\ch_5, \id\big), \big(\wt\ch_6,g_4\big), \big(\wt\ch_3,\id\big),
\\ & \big(\wt\ch_1, g_5\big), \big(\wt\ch_2,g_6\big), \big(\wt\ch_4,g_4\big) \big\}.
\end{align*}
With the shift map $\shmap_2$, the set of labels is
\begin{align*}
\Sigma = \big\{ \eps, \big(\wt\ch_4,g_7\big), \big(\wt\ch_{-1},g_7\big), & \big(\wt\ch_2,g_1\big),  \big(\wt\ch_4,\id\big),  \big(\wt\ch_5,g_2\big), \big(\wt\ch_6,\id\big), \big(\wt\ch_3,g_3\big), 
\\ & \big(\wt\ch_5,\id\big),  \big(\wt\ch_6,g_4\big), \big(\wt\ch_3,\id\big), \big(\wt\ch_{-1},g_4\big), \big(\wt\ch_2,g_6\big) \big\}.
\end{align*}
\end{example}

\begin{defirem}
Let $v\in\CS'(\wt\fch_{\choices,\shmap})$ and suppose that $\gamma$ is the geodesic on $\h$ determined by $v$. Proposition~\ref{CS2} implies that there is a unique sequence $(t_n)_{n\in J}$ in $\R$ which satisfies the following properties:
\begin{enumerate}[(i)]
\item $J = \Z \cap (a,b)$ for some interval $(a,b)$ with $a,b\in\Z\cup\{\pm\infty\}$ and $0\in (a,b)$,
\item the sequence $(t_n)_{n\in J}$ is increasing,
\item $t_0=0$,
\item for each $n\in J$ we have $\gamma'(t_n)\in\CS$ and 
\[
\gamma'( (t_n,t_{n+1}) )\cap \CS = \emptyset\quad\text{and}\quad \gamma'( (t_{n-1},t_n) ) \cap \CS = \emptyset
\]
where we set $t_b\sceq \infty$ if $b< \infty$ and $t_a\sceq -\infty$ if $a>-\infty$.
\end{enumerate}
The sequence $(t_n)_{n\in J}$ is said to be the \textit{sequence of intersection times of $v$ (with respect to $\CS$)}. %
\index{sequence of intersection times}%

Let $\wh v\in \wh\CS$ and set $v\sceq \left(\pi\vert_{\CS'(\wt\fch_{\choices,\shmap})}\right)^{-1}(\wh v)$. Then the \textit{sequence of intersection times (\wrt $\CS$) of $\wh v$} and \textit{of} each $w\in \pi^{-1}(\wh v)$ is defined to be the sequence of intersection times of $v$.
\end{defirem}

Now we define the geometric coding sequences.

\begin{defi}\label{def_CSs}
For each $s\in\Sigma$ we set
\begin{align*}
\wh\CS_s & \sceq \big\{ \wh v\in \wh\CS \ \big\vert\  \text{$\wh v$ is labeled with $s$} \big\}
\intertext{and}
\CS_s & \sceq \pi^{-1}\big(\wh\CS_s\big) = \big\{ v\in \CS \ \big\vert\  \text{$v$ is labeled with $s$} \big\}.
\end{align*}
Let $\wh v\in\wh\CS$ and let $(t_n)_{n\in J}$ be the sequence of intersection times of $\wh v$. Suppose that $\wh\gamma$ is the geodesic on $Y$ determined by $\wh v$. The \textit{geometric coding sequence} %
\index{geometric coding sequence}\index{coding sequence!geometric}%
of $\wh v$ is the sequence $(a_n)_{n\in J}$ in $\Sigma$ defined by 
\[
a_n\sceq s\quad\text{if and only if}\quad \wh\gamma'(t_n)\in\wh\CS_s
\]
for each $n\in J$.

Let $v\in\CS$. The \textit{geometric coding sequence} of $v$ is defined to be the geometric coding sequence of $\pi(v)$.
\end{defi}

\begin{prop}\label{locint}
Let $v\in\CS'$. Suppose that $(t_n)_{n\in J}$ is the sequence of intersection times of $v$ and that $(a_n)_{n\in J}$ is the geometric coding sequence of $v$. Let $\gamma$ be the geodesic on $\h$ determined by $v$. Suppose that $J=\Z \cap (a,b)$ with $a,b\in \Z \cup \{\pm \infty\}$.
\begin{enumerate}[{\rm (i)}]
\item\label{locinti} If $b=\infty$, then $a_n\in\Sigma\mminus\{\eps\}$ for each $n\in J$.
\item\label{locintii} If $b<\infty$, then $a_n\in\Sigma\mminus\{\eps\}$ for each $n\in (a,b-2]\cap\Z$ and $a_{b-1}=\eps$.
\item\label{locintiii} Suppose that $a_n= (\wt\ch_n,h_n)$ for $n\in (a,b-1)\cap\Z$ and set 
\begin{align*}
g_0 &\sceq h_0  && \text{if $b\geq 2$,}
\\
g_{n+1} & \sceq g_nh_{n+1} && \text{for $n\in [0,b-2)\cap\Z$,}
\\
g_{-1} & \sceq \id,
\\
g_{-(n+1)} & \sceq g_{-n}h_{-n}^{-1} && \text{for $n\in [1,-(a+1))\cap \Z$.}
\end{align*}
Then $\gamma'(t_{n+1}) \in g_n.\CS'(\wt\ch_n)$ for each $n\in (a,b-1)\cap\Z$.
\end{enumerate}
\end{prop}

\begin{proof} We start with some preliminary considerations which will prove \eqref{locinti} and \eqref{locintii} and simplify the argumentation for \eqref{locintiii}.
Let $n\in J$ and consider $w\sceq \gamma'(t_n)$. The definition of geometric coding sequences shows that $\gamma'(t_n)\in \CS_{a_n}$. Since $\CS$ is the disjoint union $\bigcup_{k\in\Gamma} k.\CS'(\wt\fch_{\choices,\shmap})$ (see Proposition~\ref{CS=CShat}), there is a unique $k\in\Gamma$ such that $k^{-1}.w\in\CS'(\wt\fch_{\choices,\shmap})$. The label of $k^{-1}.w$ is $a_n$. Let $\eta$ be the geodesic on $\h$ determined by $k^{-1}.w$. Note that $\eta(t) \sceq k^{-1}.\gamma(t+t_n)$ for each $t\in\R$. The definition of labels shows that $a_n=\eps$ if and only if there is no next point of intersection of $\eta$ and $\CS$. In this case $\gamma'( (t_n,\infty) ) \cap \CS=\emptyset$ and hence $b=n+1$. This shows \eqref{locinti} and \eqref{locintii}.
Suppose now that $a_n=(\wt\ch,g)$. Then there is a next point of intersection of $\eta$ and $\CS$, say $\eta'(s)$, and this is on $g.\CS'(\wt\ch)$. Then $k^{-1}.\gamma'(s+t_n)\in g.\CS'(\wt\ch)$ and $k^{-1}.\gamma'( (t_n,s+t_n) )\cap \CS = \emptyset$. Hence $t_{n+1}= s+t_n$ and $\gamma'(t_{n+1}) \in kg.\CS'(\wt\ch)$.

Now we show \eqref{locintiii}. Suppose that $b\geq 2$. Then $v=\gamma'(t_0)$ is labeled with $(\wt\ch_0,h_0)$. Hence for the next point of intersection $\gamma'(t_1)$ of $\gamma$ and $\CS$ we have
\[
 \gamma'(t_1) \in h_0.\CS'\rueck\big(\wt\ch_0\big) = g_0. \CS'\rueck\big(\wt\ch_0\big).
\]
Suppose that we have already shown that 
\[
 \gamma'(t_{n+1}) \in g_n.\CS'\rueck\big(\wt\ch_n\big)
\]
for some $n\in [0,b-1)\cap \Z$ and that $b\geq n+3$. By \eqref{locinti} resp.\@ \eqref{locintii}, $\gamma'( (t_{n+1},\infty) ) \cap \CS \not= \emptyset$ and hence $\gamma'(t_{n+1})$ is labeled with $(\wt\ch_{n+1},h_{n+1})$. Our preliminary considerations show that 
\[
 \gamma'(t_{n+2}) \in g_nh_{n+1}.\CS'\rueck\big(\wt\ch_{n+1}\big) = g_{n+1}.\CS'\rueck\big(\wt\ch_{n+1}\big).
\]
Therefore $\gamma'(t_{n+1}) \in g_n.\CS'(\wt\ch_n)$ for each $n\in [0,b-1)\cap\Z$. 

Suppose that $a\leq -2$. The element $\gamma'(t_{-1})$ is labeled with $(\wt\ch_{-1},h_{-1})$. Since $\gamma'(t_{-1}) \in k.\CS'(\wt\fch_{\choices,\shmap})$ for some $k\in\Gamma$, our preliminary considerations show that $\gamma'(t_0)\in kh_{-1}.\CS'(\wt\ch_{-1})$. Because $\gamma'(t_0) = v \in \CS'(\wt\fch_{\choices,\shmap})$, Proposition~\ref{CS=CShat} implies that $k=h_{-1}^{-1}$ and 
\[
\gamma'(t_0) \in \CS'\rueck\big(\wt\ch_{-1}\big) = g_{-1}.\CS'\rueck\big(\wt\ch_{-1}\big)\ \text{and}\ \gamma'(t_{-1}) \in h_{-1}^{-1}.\CS'\rueck\big(\wt\fch_{\choices,\shmap}\big) = g_{-2}.\CS'\rueck\big(\wt\fch_{\choices,\shmap}\big).
\]
Suppose that we have already shown that 
\[
\gamma'(t_{-(n-1)}) \in g_{-n}.\CS'\rueck\big(\wt\ch_{-n}\big)\quad \text{and}\quad \gamma'(t_{-n}) \in g_{-(n+1)}.\CS'\rueck\big(\wt\fch_{\choices,\shmap}\big)
\]
for some $n\in [1,-a)\cap \Z$ and suppose that $a\leq -n-2$. Then $\gamma'(t_{-n-1})$ exists and is labeled with $(\wt\ch_{-n-1},h_{-n-1})$. Since $\gamma'(t_{-n-1}) \in h.\CS'(\wt\fch_{\choices,\shmap})$ for some $h\in\Gamma$, we know that $\gamma'(t_{-n}) \in hh_{-n-1}.\CS'(\wt\ch_{-n-1})$. Therefore
\[
\gamma'(t_{-n}) \in hh_{-(n+1)}.\CS'\rueck\big(\wt\ch_{-(n+1)}\big) \cap g_{-(n+1)}.\CS'\rueck\big(\wt\fch_{\choices,\shmap}\big).
\]
Proposition~\ref{CS=CShat} implies that $hh_{-(n+1)} = g_{-(n+1)}$, $\gamma'(t_{-n}) \in g_{-(n+1)}.\CS'(\wt\ch_{-(n+1)})$ and 
\[
\gamma'(t_{-(n+1)}) \in g_{-(n+1)}h_{-(n+1)}^{-1}. \CS'\rueck\big(\wt\fch_{\choices,\shmap}\big) = g_{-(n+2)}.\CS'\rueck\big(\wt\fch_{\choices,\shmap}\big).
\]
Therefore $\gamma'(t_{n+1}) \in g_n.\CS'(\wt\ch_n)$ for each $n\in (a,-1]\cap \Z$. This completes the proof.
\end{proof}

Let $\Lambda$ denote the set of geometric coding sequences and let $\Lambda_\sigma$ be the subset of $\Lambda$ which contains the geometric coding sequences $(a_n)_{n\in (a,b)\cap\Z}$ with $a,b\in\Z\cup\{\pm\infty\}$ for which $b\geq 2$. Let $\Sigma^\all$ denote the set of all finite and one- or two-sided infinite sequences in $\Sigma$. The left shift $\sigma\colon \Sigma^\all \to \Sigma^\all$,
\[
\sigma\big( (a_n)_{n\in J} \big)_k \sceq a_{k+1} \quad\text{for all $k\in J$}
\]
induces a partially defined map $\sigma\colon\Lambda\to\Lambda$ resp.\@ a map $\sigma\colon \Lambda_\sigma \to \Lambda$. Suppose  that $\Seq\colon \wh\CS\to \Lambda$ is the map which assigns to $\wh v\in \wh \CS$ the geometric coding sequence of $\wh v$. Recall the first return map $R$ from Section~\ref{sec_symdyn}.

\begin{prop}
The diagram
\[
\xymatrix{
\wh\CS \ar[r]^R \ar[d]_{\Seq} & \wh\CS\ar[d]^{\Seq}
\\
\Lambda \ar[r]^{\sigma} & \Lambda
}
\]
commutes. In particular, for $\wh v\in\wh\CS$, the element $R(\wh v)$ is defined if and only if $\Seq(\wh v)\in\Lambda_\sigma$.
\end{prop}

\begin{proof}
This follows immediately from the definition of geometric coding sequences, Propositions~\ref{CS2} and \ref{locint}.
\end{proof}

Set\label{def_CSst} $\CS_\st\sceq \pi^{-1}(\wh\CS_\st)$ and $\CS'_\st(\wt\fch_{\choices,\shmap})\sceq \CS_\st\cap \CS'(\wt\fch_{\choices,\shmap})$ and let $\Lambda_\st$ denote the set of two-sided infinite geometric coding sequences. 

\begin{remark}
The set of geometric coding sequences of elements in $\CS_\st$ (or only $\CS'_\st(\wt\fch_{\choices,\shmap})$) is $\Lambda_\st$. Moreover, $\Lambda_\st\subseteq \Lambda_\sigma$. 
\end{remark}

In the following we will show that $(\Lambda_\st,\sigma)$ is a symbolic dynamics for the geodesic flow on $\wh\Phi$. Elementary convex geometry proves the following lemma.

\begin{lemma}\label{findsec}
Suppose that $x,y\in\bhg H\mminus\bd$, $x<y$. Then there exists a connected component $S=[a,b]$ of $\BS$ with $a,b\in\R$, $a<b$, such that $x<a<y<b$.
\end{lemma}

For the proof of the following proposition we recall that each connected component of $\BS$ is a complete geodesic segment and that it is of the form $\pr(g.\CS'(\wt\ch))$ for some pair $(\wt\ch,g)\in \wt\fch_{\choices,\shmap}\times\Gamma$. Conversely, for each pair $(\wt\ch,g)\in \wt\fch_{\choices,\shmap}$, the set $\pr(g.\CS'(\wt\ch))$ is a connected component of $\BS$.

\begin{prop}\label{geomsequnique}
Let $v,w\in\CS'_\st(\wt\fch_{\choices,\shmap})$. If the geometric coding sequences of $v$ and $w$ are equal, then $v=w$.
\end{prop}

\begin{proof}
Let $\big( (\wt\ch_j, h_j) \big)_{j\in\Z}$ be the geometric coding sequence of $v$ and assume that $\big( (\wt\ch'_j,k_j) \big)_{j\in\Z}$ is that of $w$. Suppose that $v\not=w$. Suppose first that 
\[
 \big( \gamma_v(\infty), \gamma_v(-\infty) \big) = \big( \gamma_w(\infty), \gamma_w(-\infty) \big).
\]
Proposition~\ref{locint} shows that $v\in \CS'(\wt\ch_{-1})$ and $w\in\CS'(\wt\ch'_{-1})$. Lemma~\ref{char_intervals} implies that $\wt\ch_{-1}\not=\wt\ch'_{-1}$, which shows that the geometric coding sequences of $v$ and $w$ are different.

Suppose now that 
\[
 \big( \gamma_v(\infty),\gamma_v(-\infty) \big) \not= \big( \gamma_w(\infty), \gamma_w(-\infty) \big).
\]
Assume for contradiction that $\big( (\wt\ch_j,h_j) \big)_{j\in\Z} = \big( (\wt\ch'_j,k_j) \big)_{j\in\Z}$. Let $(t_n)_{n\in\Z}$ be the sequence of intersection times of $v$ and $(s_n)_{n\in\Z}$ be that of $w$. Prop~\ref{locint}\eqref{locintiii} implies that for each $n\in\Z$, the elements $\pr(\gamma'_v(t_n))$ and $\pr(\gamma'_w(s_n))$ are on the same connected component of $\BS$. For each connected component $S$ of $\BS$ let $H_{1,S}, H_{2,S}$ denote the open convex half spaces such that $\h$ is the disjoint union 
\[
\h=H_{1,S}\cup S\cup H_{2,S}.
\]
Suppose first that $\gamma_v(\infty) \not= \gamma_w(\infty)$. Proposition~\ref{CS3} shows that 
\[
\gamma_v(\infty),\gamma_w(\infty)\notin\bd.
\]
By Lemma~\ref{findsec} we find a connected component $S$ of $\BS$ such that $\gamma_v(\infty)\in\bhg H_{1,S}\mminus\bhg S$ and $\gamma_w(\infty)\in \bhg H_{2,S}\mminus\bhg S$ (or vice versa). Since $\BS$ is a manifold, each connected component of $\BS$ other than $S$ is either contained in $H_{1,S}$ or in $H_{2,S}$. In particular, we may assume that $\pr(v),\pr(w)\in H_{1,S}$. Then 
\[
\gamma_v([0,\infty))\subseteq H_{1,S}\quad\text{and}\quad \gamma_w( (t,\infty) )\subseteq H_{2,S}
\]
for some $t>0$. Hence there is $n\in\N$ such that $\pr(\gamma'_w(s_n))\in H_{2,S}$, which implies that $\pr(\gamma'_v(t_n))$ and $\pr(\gamma'_w(s_n))$ are not on the same connected component of $\BS$.

Suppose now that $\gamma_v(-\infty)\not=\gamma_w(-\infty)$ and let $S$ be a connected component of $\BS$ auch that $\gamma_v(-\infty) \in \bhg H_{1,S}\mminus\bhg S$ and $\gamma_w(-\infty)\in \bhg H_{2,S}\mminus\bhg S$ (or vice versa). Again, we may assume that $\pr(v),\pr(w)\in H_{1,S}$. Then 
\[
\gamma_v( (-\infty,0] ) \subseteq H_{1,S}\quad\text{and}\quad\gamma_w(-\infty, s))\subseteq H_{2,S}
\]
for some $s<0$. Thus we find $n\in\N$ such that $\pr(\gamma'_w(s_{-n}))\in H_{2,S}$. Hence $\pr(\gamma'_v(t_{-n}))$ and $\pr(\gamma'_w(s_{-n}))$ are not on the same connected component of $\BS$. In both cases we find a contradiction. Therefore the geometric coding sequences are not equal.
\end{proof}

\begin{cor}
The map $\Seq\vert_{\wh\CS_\st}\colon \wh\CS_\st\to \Lambda_\st$ is bijective.
\end{cor}

\begin{remark}
If there is more than one shifted cell in $S\h$ or if there is a strip precell in $\h$, then the map $\Seq\colon \wh\CS\mminus\wh\CS_\st\to \Lambda\mminus\Lambda_\st$ is not injective. This is due to the decision to label each $v\in\CS'(\wt\ch)$, for each $\wt\ch\in\wt\fch_{\choices,\shmap}$, with the same label $\eps$ if $\gamma_v(\infty)\in \bhg\pr(\wt\ch)$ without distinguishing between different points in $\bhg \pr(\wt\ch)$ and without distinguishing between different shifted cells in $S\h$.
\end{remark}

Let $\Cod \sceq \left(\Seq\vert_{\wh\CS_\st}\right)^{-1}\colon \Lambda_\st\to \wh\CS_\st$.

\begin{cor}
The diagram
\[
\xymatrix{
\wh\CS_\st \ar[r]^R & \wh\CS_\st
\\
\Lambda_\st \ar[r]^\sigma \ar[u]^{\Cod} & \Lambda_\st \ar[u]_{\Cod}
}
\]
commutes and $(\Lambda_\st,\sigma)$ is a symbolic dynamics for the geodesic flow on $Y$.
\end{cor}

We end this section with the explanation of the acronyms $\NC$ and $\NIC$ (cf.\@ Remark~\ref{outlook}).

\begin{remark}\label{whynaming}
Let $\wh v$ be a unit tangent vector in $S\h$ based on $\wh\BS$ and let $\wh\gamma$ be the geodesic determined by $\wh v$. Then $\wh v$ has no geometric coding sequence if and only if $\wh v\notin \wh\CS$. By Proposition~\ref{CS1} this is the case if and only if $\wh\gamma\in\NC$. This is the reason why $\NC$ stands for ``not coded''.

Suppose now that $\wh v\in\wh\CS$. Then the geometric coding sequence is not two-sided infinite if and only if $\wh\gamma$ does not intersect $\wh\CS$ infinitely often in past and future, which by Proposition~\ref{CS3} is equivalent to $\wh\gamma\in\NIC$. This explains why $\NIC$ is for ``not infinitely often coded''.
\end{remark}

\section{Reduction and arithmetic symbolic dynamics}\label{sec_arithm}

Let $\Gamma$ be a geometrically finite subgroup of $\PSL_2(\R)$ of which $\infty$ is a cuspidal point and which satisfies \eqref{A4}. Suppose that the set of relevant isometric spheres is non-empty. Fix a basal family $\fpch$ of precells in $\h$ and let $\fch$ be the family of cells in $\h$ assigned to $\fpch$. Let $\choices$ be a set of choices associated to $\fpch$ and suppose that $\wt\fch_{\choices}$ is the family of cells in $S\h$ associated to $\fpch$ and $\choices$. Let $\shmap$ be a shift map for $\wt\fch_\choices$ and let $\wt\fch_{\choices,\shmap}$ denote the family of cells in $S\h$ associated to $\fpch$, $\choices$ and $\shmap$.

Recall the geometric symbolic dynamics for the geodesic flow on $Y$ which we constructed in Section~\ref{sec_geomsymdyn} with respect to $\fpch$, $\choices$ and $\shmap$. In particular, recall the set $\CS'(\wt\fch_{\choices,\shmap})$ of representatives for the cross section $\wh\CS=\wh\CS(\wt\fch_{\choices,\shmap})$, its subsets $\CS'(\wt\ch)$ for $\wt\ch\in\wt\fch_{\choices,\shmap}$, and the definition of the labeling of $\wh\CS$.

Let $v\in \CS'(\wt\ch)$ for some $\wt\ch\in\wt\fch_{\choices,\shmap}$ and consider the geodesic $\gamma_v$ on $\h$ determined by $v$. Suppose that $(a_n)_{n\in J}$ is the geometric coding sequence of $v$. The combination of Propositions~\ref{char_bound} and \ref{CS2} allows to determine the label $a_0$ of $v$ from the  location of $\gamma_v(\infty)$, and then to iteratively reconstruct the complete future part $(a_n)_{n\in [0,\infty)\cap J}$   of the geometric coding sequence of $v$. Hence, if the unit tangent vector $v\in \CS'(\wt\fch_{\choices,\shmap})$ is known, or more precisely, if the shifted cell $\wt\ch$ in $S\h$ with $v\in\CS'(\wt\ch)$ is known, then one can reconstruct at least the future part of the geometric coding sequence of $v$. However, if $\gamma_v$ intersects $\CS'(\wt\fch_{\choices,\shmap})$ in more than one point, then one cannot derive the shifted cell $\wt\ch$ in $S\h$ from the endpoints of $\gamma_v$. In this case, the induced discrete dynamical system on $\bhg \h$ which is conjugate 
to $(\wh\CS,R)$ or $(\wh\CS_\st,R)$ is given by local diffeomorphisms which have to keep the shifted cells in their definitions. These discrete dynamical systems are used in \cite{Pohl_mcf_Gamma0p, Pohl_mcf_general} to provide a dynamical approach to Maass cusp forms. Here, we will propose a second way to deduce discrete dynamical systems on $\bhg\h$. 

For these, we restrict, in Section~\ref{sec_redcross}, our cross section $\wh\CS$ to a subset $\wh\CS_\rd(\wt\fch_{\choices,\shmap})$ (resp.\@ to $\wh\CS_{\st,\rd}(\wt\fch_{\choices,\shmap})$ for the strong cross section $\wh\CS_\st$) such that for any $v\in \wh\CS_\rd$, the endpoints of $\gamma_v$ completely determine $v$. We will show that $\wh\CS_\rd(\wt\fch_{\choices,\shmap})$ and $\wh\CS_{\st,\rd}(\wt\fch_{\choices,\shmap})$ are cross sections for the geodesic flow on $Y$ \wrt to the same measure as $\wh \CS$ and $\wh\CS_{\st}$. More precisely, it will turn out that exactly those geodesics on $Y$ which intersect $\wh\CS$ at all also intersect $\wh\CS_\rd(\wt\fch_{\choices,\shmap})$ at all, and that exactly those which intersect $\wh\CS$ infinitely often in future and past also intersect $\wh\CS_\rd(\wt\fch_{\choices,\shmap})$ infinitely often in future and past. Moreover, $\wh\CS_{\st,\rd}(\wt\fch_{\choices,\shmap})$ is the maximal strong cross section contained in $\wh\CS_\rd(\wt\fch_{\choices,\shmap})
$. In contrast to $\wh\CS$ and $\wh\CS_\st$, the sets $\wh\CS_\rd(\wt\fch_{\choices,\shmap})$ and $\wh\CS_{\st,\rd}(\wt\fch_{\choices,\shmap})$ do depend on the choice of the family $\wt\fch_{\choices,\shmap}$. Moreover, we will deduce discrete dynamical systems $(\wt\DS, \wt F)$ and $(\wt\DS_\st, \wt F_\st)$ with $\wt\DS_\st \subseteq \wt\DS \subseteq \R\times\R$ which are conjugate to $(\wh\CS_\rd(\wt\fch_{\choices,\shmap}), R)$ resp.\@ $(\wh\CS_{\st,\rd}(\wt\fch_{\choices,\shmap}),R)$.

All construction in this process are of geometrical nature and effectively performable in a finite number of steps. Moreover, the set of labels is finite.

\subsection{Reduced cross section}\label{sec_redcross}

The set $\{ I(\wt\ch) \mid \wt\ch\in\wt\fch_{\choices,\shmap}\}$ decomposes into two sequences
\[
\mc I_1\big(\wt\fch_{\choices,\shmap}\big)  \sceq \left\{ I\big(\wt\ch_{1,1}\big) \supseteq I\big(\wt\ch_{1,2}\big) \supseteq \ldots \supseteq I\big(\wt\ch_{1,k_1}\big) \right\}
\]
where $I(\wt\ch_{1,j}) = (a_j,\infty)$ and $a_1<a_2<\ldots < a_{k_1}$, and 
\[
\mc I_2\big(\wt\fch_{\choices,\shmap}\big) \sceq \left\{ I\big(\wt\ch_{2,1}\big) \supseteq I\big(\wt\ch_{2,2}\big) \supseteq \ldots \supseteq I\big(\wt\ch_{2,k_2}\big) \right\}
\]
where $I(\wt\ch_{2,j}) = (-\infty, b_j)$ and $b_1 > b_2 > \ldots > b_{k_2}$.

Set $I(\wt\ch_{1,k_1+1})\sceq \emptyset$ and
\[
 I_\rd\big(\wt\ch_{1,j}\big) \sceq I\big(\wt\ch_{1,j}\big)\mminus I\big(\wt\ch_{1,j+1}\big) \qquad\text{for $j=1,\ldots, k_1$,}
\]
and set $I(\wt\ch_{2,k_2+1})\sceq\emptyset$ and
\[
 I_\rd\big(\wt\ch_{2,j}\big) \sceq I\big(\wt\ch_{2,j}\big)\mminus I\big(\wt\ch_{2,j+1}\big) \qquad\text{for $j=1,\ldots, k_2$.}
\]
For each $\wt\ch\in\wt\fch_{\choices,\shmap}$ set 
\[
\CS'_\rd\rueck\big(\wt\ch\big) \sceq \left\{ v\in \CS'\rueck\big(\wt\ch\big) \left\vert\  \big(\gamma_v(\infty), \gamma_v(-\infty)\big) \in I_\rd\big(\wt\ch\big)\times J\big(\wt\ch\big) \right.\right\}
\]
and
\[
 \CS'_\rd\rueck\big(\wt\fch_{\choices,\shmap}\big) \sceq \bigcup_{\wt\ch\in\wt\fch_{\choices,\shmap}} \CS'_\rd\rueck\big(\wt\ch\big).
\]
Define 
\[
\wh\CS_\rd\big(\wt\fch_{\choices,\shmap}\big) \sceq \pi\big(\CS'_\rd\rueck\big(\wt\fch_{\choices,\shmap}\big)\big) \quad\text{and}\quad \CS_\rd\big(\wt\fch_{\choices,\shmap}\big) \sceq \pi^{-1}\big(\wh\CS_\rd\big(\wt\fch_{\choices,\shmap}\big)\big).
\]
 Further set 
\begin{align*}
&& \CS'_{\st,\rd}\rueck\big(\wt\ch\big) & \sceq \CS'_\rd\rueck\big(\wt\ch\big) \cap \CS_\st \qquad\text{for each $\wt\ch\in\wt\fch_{\choices,\shmap}$},
\\
&& \CS'_{\st,\rd}\rueck\big(\wt\fch_{\choices,\shmap}\big) & \sceq \CS'_\rd\rueck\big(\wt\fch_{\choices,\shmap}\big) \cap \CS_\st = \bigcup_{\wt\ch\in\wt\fch_{\choices,\shmap}} \CS'_{\st,\rd}\rueck\big(\wt\ch\big),
\\
&& \CS_{\st,\rd}\big(\wt\fch_{\choices,\shmap}\big) & \sceq \CS_\rd\big(\wt\fch_{\choices,\shmap}\big) \cap \CS_\st,
\\
\text{and} && \wh\CS_{\st,\rd}\big(\wt\fch_{\choices,\shmap}\big) & \sceq \wh\CS_\rd\big(\wt\fch_{\choices,\shmap}\big) \cap \wh\CS_\st.
\end{align*}

\begin{remark}
Let $\wt\ch\in\wt\fch_{\choices,\shmap}$.  Note that the sets $I_\rd(\wt\ch)$ and  $\CS'_\rd(\wt\ch)$ not only depend on $\wt\ch$ but also on the choice of the complete family $\wt\fch_{\choices,\shmap}$. The set $\CS'_\rd(\wt\ch)$ is identical to 
\[
\big\{ v\in \CS'\rueck\big(\wt\ch\big) \ \big\vert\  \gamma_v\big( (0,\infty) \big)\cap \CS'\rueck\big(\wt\fch_{\choices,\shmap}\big) = \emptyset \big\}.
\]
In other words, if we say that $v\in\CS'(\wt\ch)$ has an \textit{inner intersection} %
\index{inner intersection}%
if and only if $\gamma_v( (0,\infty) )\cap \CS'(\wt\fch_{\choices,\shmap}) \not= \emptyset$, then $\CS'_\rd(\wt\ch)$ is the subset of $\CS'(\wt\ch)$ of all elements without inner intersection.
\end{remark}

\begin{remark}\label{redrepset}
By Proposition~\ref{CS=CShat}, the union 
\[
\CS'(\wt\fch_{\choices,\shmap})= \bigcup\big\{\CS'(\wt\ch)\ \big\vert\  \wt\ch\in\wt\fch_{\choices,\shmap}\big\}
\]
is disjoint and $\CS'(\wt\fch_{\choices,\shmap})$ is a set of representatives for $\wh\CS$. Since $\CS'_\rd(\wt\ch)$ is a subset of $\CS'(\wt\ch)$ for each $\wt\ch\in\wt\fch_{\choices,\shmap}$ and $\wh\CS_\rd(\wt\fch_{\choices,\shmap})=\pi(\CS'_\rd(\wt\fch_{\choices,\shmap}))$, the set $\CS'(\wt\fch_{\choices,\shmap})$ is a set of representatives for $\wh\CS_\rd(\wt\fch_{\choices,\shmap})$ and $\CS'_\rd(\wt\fch_{\choices,\shmap})$ is the disjoint union $\bigcup_{g\in\Gamma} g.\CS'_\rd(\wt\fch_{\choices,\shmap})$. Further, one easily sees that 
\[
\pi^{-1}\big( \wh\CS_{\st,\rd}\big(\wt\fch_{\choices,\shmap}\big) \big) = \CS_{\st,\rd}\big(\wt\fch_{\choices,\shmap}\big)
\]
and
\[
\pi\big( \CS'_{\st,\rd}\rueck\big(\wt\fch_{\choices,\shmap}\big)\big) = \wh\CS_{\st,\rd}\big(\wt\fch_{\choices,\shmap}\big).
\]
Moreover, $\CS'_{\st,\rd}(\wt\fch_{\choices,\shmap})$ is a set of representatives for $\wh\CS_{\st,\rd}(\wt\fch_{\choices,\shmap})$.
\end{remark}

\begin{example}\label{redsetGamma05}
Recall Example~\ref{intersectionGamma05}. Suppose first that the shift map is $\shmap_1$. Then 
\begin{align*}
I\big(\wt\ch_1\big) & = (-\infty, 1), & I\big(\wt\ch_2\big) & = (0,\infty), & I\big(\wt\ch_4\big) & = \big(\tfrac15,\infty\big), & I\big(\wt\ch_6\big) & = \big(\tfrac25,\infty\big), 
\\
I\big(\wt\ch_5\big) & = \big(\tfrac35,\infty\big), & I\big(\wt\ch_3\big) & = \big(\tfrac45,\infty\big).
\end{align*}
Therefore we have
\begin{align*}
I_\rd\big(\wt\ch_1\big) & = (-\infty, 1), & I_\rd\big(\wt\ch_2\big) & = \big(0,\tfrac15\big], & I_\rd\big(\wt\ch_4\big) & = \big(\tfrac15,\tfrac25\big], 
\\
I_\rd\big(\wt\ch_6\big) & = \big(\tfrac25,\tfrac35\big], & I_\rd\big(\wt\ch_5\big) & = \big(\tfrac35,\tfrac45\big],
& I_\rd\big(\wt\ch_3\big) & = \big(\tfrac45,\infty\big).
\end{align*}
If the shift map is $\shmap_2$, then we find
\begin{align*}
I_\rd\big(\wt\ch_{-1}\big) & = (-\infty, 0), & I_\rd\big(\wt\ch_2\big) & = \big(0,\tfrac15\big], & I_\rd\big(\wt\ch_4\big) & = \big(\tfrac15,\tfrac25\big], 
\\
I_\rd\big(\wt\ch_6\big) & = \big(\tfrac25,\tfrac35\big], & I_\rd\big(\wt\ch_5\big) & = \big(\tfrac35,\tfrac45\big],
& I_\rd\big(\wt\ch_3\big) & = \big(\tfrac45,\infty\big).
\end{align*}
Note that with $\shmap_2$, the sets $I_\rd(\cdot)$ are pairwise disjoint, whereas with $\shmap_1$ they are not.
\end{example}

\begin{lemma}\label{uniqueinter}
Let 
\[
(x,y) \in \bigcup_{\wt\ch\in\wt\fch_{\choices,\shmap}} I_\rd\big(\wt\ch\big)\times J\big(\wt\ch\big).
\]
Then there is a unique $v\in\CS'_\rd(\wt\fch_{\choices,\shmap})$ such that $(x,y)=\big(\gamma_v(\infty),\gamma_v(-\infty)\big)$. Conversely, if $v\in \CS'_\rd(\wt\fch_{\choices,\shmap})$, then 
\[
\big(\gamma_v(\infty),\gamma_v(-\infty)\big) \in \bigcup_{\wt\ch\in\wt\fch_{\choices,\shmap}} I_\rd\big(\wt\ch\big)\times J\big(\wt\ch\big).
\]
\end{lemma}

\begin{proof}
The combination of the definition of $\CS'_\rd(\wt\fch_{\choices,\shmap})$ and Lemma~\ref{char_intervals} shows that there is at least one $v\in\CS'_\rd(\wt\fch_{\choices,\shmap})$ such that $\big(\gamma_v(\infty),\gamma_v(-\infty)\big) = (x,y)$ and that for each $\wt\ch\in\wt\fch_{\choices,\shmap}$ there there is at most one such $v$.
By construction, 
\[
\left( I_\rd\big(\wt\ch_a\big)\times J\big(\wt\ch_a\big) \right) \cap \left( I_\rd\big(\wt\ch_b\big)\times J\big(\wt\ch_b\big)\right) = \emptyset
\]
for $\wt\ch_a,\wt\ch_b\in\wt\fch_{\choices,\shmap}$, $\wt\ch_a\not=\wt\ch_b$. Hence there is a unique such $v\in\CS'_\rd(\wt\fch_{\choices,\shmap})$. The remaining assertion is clear from the definition of the sets $\CS'_\rd(\wt\ch)$ for $\wt\ch\in\wt\fch_{\choices,\shmap}$.
\end{proof}

Let $l_1(\wt\fch_{\choices,\shmap})$ be the number of connected components of $\BS$ of the form $(a,\infty)$ (geodesic segment) with $a_1\leq a \leq a_{k_1}$ and let $l_2(\wt\fch_{\choices,\shmap})$ be the number of connected components of $\BS$ of the form $(a,\infty)$ with $b_{k_2} \leq a \leq b_1$. Define
\[
l\big(\wt\fch_{\choices,\shmap}\big) \sceq \max \left\{ l_1\big(\wt\fch_{\choices,\shmap}\big), l_2\big(\wt\fch_{\choices,\shmap}\big) \right\}.
\]

\begin{prop}\label{havenext}
Let $v\in \CS'(\wt\fch_{\choices,\shmap})$ and suppose that $\eta$ is the geodesic on $\h$ determined by $v$. Let $(s_j)_{j\in(\alpha,\beta)\cap \Z}$ be the geometric coding sequence of $v$. Suppose that $s_j=(\wt\ch_j,h_j)$ for $j=0,\ldots, \beta-2$. For $j=0,\ldots, \beta-2$ set 
\[
g_{-1}\sceq \id\quad\text{and}\quad g_j\sceq g_{j-1}h_j.
\]
If $\alpha=-1$, then let $\wt\ch_{-1}$ be the shifted cell in $SH$ such that $v\in\CS'(\wt\ch_{-1})$. 
Then 
\[
 s_0 \sceq \min \big\{ t\geq 0 \ \big\vert\  \eta'(t) \in \CS_\rd\big(\wt\fch_{\choices,\shmap}\big) \big\}
\]
exists and 
\[
 \eta'(s_0) \in \bigcup_{l=-1}^{\kappa} g_l.\CS'_\rd\rueck\big(\wt\ch_l\big)
\]
where $\kappa\sceq \min\{ l(\wt\fch_{\choices,\shmap}) -1 , \beta-2\}$. More precisely, $\eta'(s_0)\in g_l.\CS'_\rd(\wt\ch_l)$ for $l\in\{-1,\ldots, \kappa\}$ if and only if $\eta(\infty)\in g_l. I_\rd(\wt\ch_l)$ and $\eta(\infty)\notin g_k. I_\rd(\wt\ch_k)$ for $k=-1,\ldots, l-1$. Moreover, if $v\in\CS'_\st(\wt\fch_{\choices,\shmap})$, then $\eta'(s_0)\in \CS_{\st,\rd}(\wt\fch_{\choices,\shmap})$.
\end{prop}

\begin{proof}
Let $(t_n)_{n\in (\alpha,\beta)\cap\Z}$ be the sequence of intersection times of $v$ (with respect to $\CS$). Proposition~\ref{locint}\eqref{locintiii} resp.\@ the choice of $\wt\ch_{-1}$ shows that $\eta'(t_n)\in g_{n-1}.\CS'(\wt\ch_{n-1})$ for each $n\in [0,\beta)\cap\Z$. Since $\CS_\rd(\wt\fch_{\choices,\shmap})\subseteq \CS$, the minimum $s_0$ exists if and only if $\eta'(t_m)\in g_{m-1}.\CS'_\rd(\wt\ch_{m-1})$ for some $m\in [0,\beta)\cap\Z$. In this case, $s_0=t_n$ and $\eta'(s_0)\in g_{n-1}.\CS'_\rd(\wt\ch_{n-1})$ where
\[
 n\sceq \min \big\{ m\in [0,\beta)\cap\Z \ \big\vert\ \eta'(t_m)\in g_{m-1}.\CS'_\rd\rueck\big(\wt\ch_{m-1}\big) \big\}.
\]
Suppose that $s_0=t_n$. Note that for each $m\in[0,\beta)\cap\Z$ we have that $\eta(-\infty)$ is in $J(\wt\ch_{m-1})$. Then the definition of $\CS'_\rd(\cdot)$ shows that 
\[
 n= \min\big\{ m\in [0,\beta)\cap\Z \ \big\vert\  \eta(\infty) \in g_{m-1}. I_\rd\big(\wt\ch_{m-1}\big) \big\}.
\]
Hence it remains to show that the element $s_0$ exists and that $s_0=t_n$ for some $n\in\{0,\ldots, \kappa+1\}$.

W.l.o.g.\@ suppose that $I(\wt\ch_{-1})\in \mc I_1(\wt\fch_{\choices,\shmap})$. Then $I(\wt\ch_{-1}) = (a,\infty)$ for some $a\in\R$. Let $c_1<c_2<\ldots <c_k$ be the increasing sequence in $\R$ such that $c_1=a$ and $c_k=a_{k_1}$ and such that the set 
\[
 \{ (c_j,\infty) \mid j=1,\ldots, k\}
\]
of geodesic segments is the set of connected components of $\BS$ of the form $(c,\infty)$ with $a\leq c\leq a_{k_1}$. Then $k\leq l(\wt\fch_{\choices,\shmap})$. Let $\{(c_{j_i},\infty) \mid i=1,\ldots, m\}$ be its subfamily (indexed by $\{1,\ldots, m\}$) of geodesic segments such that for each $i\in\{1,\ldots, m\}$ we have $(c_{j_i},\infty) = b(\wt\ch'_i)$ for some $\wt\ch'_i\in\wt\fch_{\choices,\shmap}$ such that $b(\wt\ch'_i)\in \mc I_1(\wt\fch_{\choices,\shmap})$ and 
\[
c_{j_1} \leq c_{j_2} \leq \ldots \leq c_{j_m}.
\]
The definition of $I_\rd(\cdot)$ shows that $I(\wt\ch_{-1})$ is the disjoint union $\bigcup_{i=1}^m I_\rd(\wt\ch'_i)$. Moreover, $J(\wt\ch'_i) \supseteq J(\wt\ch_{-1})$ for $i=1,\ldots, m$. From Lemma~\ref{char_intervals} we know that 
\[
\big( \eta(\infty), \eta(-\infty) \big) \in I(\wt\ch_{-1})\times J(\wt\ch_{-1}).
\]
Hence there is a unique $i\in\{1,\ldots, m\}$ such that 
\[
\big( \eta(\infty), \eta(-\infty) \big) \in I_\rd(\wt\ch'_i) \times J(\wt\ch'_i).
\]
In turn, $\eta$ intersects $\CS'_\rd(\wt\ch'_i)$. Now, if $\eta$ does not intersect $\CS'_\rd(\wt\ch'_1) = \CS'_\rd(\wt\ch_{-1})$, then $\eta( [0,\infty) )$ intersects $(c_2,\infty)$ and we have $g_0.b(\wt\ch_0) = (c_2,\infty)$. If $\eta$ does not intersect $g_0.\CS'_\rd(\wt\ch_0)$, then $\eta([0,\infty))$ intersects $(c_3,\infty)$ and hence $g_1.b(\wt\ch_1) = (c_3,\infty)$ and so on. This shows that
\[
\CS'(\wt\ch'_i)=\CS'(\wt\ch_{j_i-2}) = g_{j_i-2}.\CS'(\wt\ch_{j_i-2})
\]
and hence $\eta'(t_{j_i-1})\in g_{j_i-2}.\CS'_\rd(\wt\ch_{j_i-2})$, where $j_i-1 \leq k \leq l(\wt\fch_{\choices,\shmap})$. If $\beta<\infty$, then $j_i-1\leq \beta-1$. Thus, $s_0$ exists and $\eta'(s_0)\in \bigcup_{l=-1}^{\kappa}g_l.\CS'_\rd(\wt\ch_{l})$. 
Finally, if $v\in\CS'_\st(\wt\fch_{\choices,\shmap})$, then $\eta'(s_0)\in \CS_\st$ and hence $\eta'(s_0)\in \CS_{\st,\rd}(\wt\fch_{\choices,\shmap})$.
\end{proof}

Recall the shift map $\sigma\colon \Lambda\to\Lambda$ from Section~\ref{sec_geomcodseq}. The following proposition is a consequence of Proposition~\ref{havenext}.

\begin{prop}\label{remainsstrong}
Let $v\in\CS'_{\rd}(\wt\fch_{\choices,\shmap})$ and suppose that $\eta$ is the geodesic on $\h$ determined by $v$. Let $\big( s_j \big)_{j\in (\alpha,\beta)\cap\Z}$ be the geometric coding sequence of $v$. Suppose that $s_j=(\wt\ch_j,h_j)$ for $j=0,\ldots,\beta-2$.
\begin{enumerate}[{\rm (i)}]
\item\label{rsi} Set $g_0\sceq h_0$ and for $j=0,\ldots,\beta-3$ define $g_{j+1}\sceq g_jh_{j+1}$. 
If there is a next point of intersection of $\eta$ and $\CS_\rd(\wt\fch_{\choices,\shmap})$, then this is on
\[
\bigcup_{l=0}^{\kappa} g_l.\CS'_\rd\rueck\big(\wt\ch_l\big)
\]
where $\kappa\sceq \min\{ l(\choices,\shmap), \beta-2\}$.
In this case it is on $g_l.\CS'_\rd(\wt\ch_l)$ if and only if $\eta(\infty)\in g_l. I_\rd(\wt\ch_l)$ and $\eta(\infty)\notin g_k.I_\rd(\wt\ch_k)$ for $k=0,\ldots, l-1$. If $\beta\geq 2$, then there is a next point of intersection of $\eta$ and $\CS_\rd(\wt\fch_{\choices,\shmap})$. 
\item\label{rsii} Suppose that $v\in \CS'_{\st,\rd}(\wt\fch_{\choices,\shmap})$. Let $(t_{n})_{n\in\Z}$ be the sequence of intersection times of $v$ (\wrt $CS$). Then there was a previous point of intersection of $\eta$ and $\CS_\rd(\wt\fch_{\choices,\shmap})$ and this is contained in 
\[
\big\{ \eta'(t_{-n}) \ \big\vert\  n=1,\ldots, l\big(\wt\fch_{\choices,\shmap}\big)+1 \big\}.
\]
\end{enumerate}
\end{prop}

Recall from Remark~\ref{outlook} that  $\NIC$ denotes the set of geodesics on $Y$ with at least one endpoint contained in $\pi(\bd)$.

\begin{cor}\label{redcross}
Let $\mu$ be a measure on the space of geodesics on $Y$. 
The sets $\wh\CS_\rd(\wt\fch_{\choices,\shmap})$ and  $\wh\CS_{\st,\rd}(\wt\fch_{\choices,\shmap})$ are cross sections \wrt $\mu$ if and only if $\NIC$ is a $\mu$-null set. Moreover, $\wh\CS_{\st,\rd}(\wt\fch_{\choices,\shmap})$ is the maximal strong cross section contained in $\wh\CS_\rd(\wt\fch_{\choices,\shmap})$.
\end{cor}

\subsection{Reduced coding sequences and arithmetic symbolic dynamics}\label{sec_redcodseq}

Analogous to the labeling of $\CS$ in Section~\ref{sec_geomcodseq} we define a labeling of $\CS_\rd(\wt\fch_{\choices,\shmap})$. 

Let $v\in\CS'_\rd(\wt\fch_{\choices,\shmap})$ and let $\gamma$ denote the geodesic on $\h$ determined by $v$. Suppose first that $\gamma( (0,\infty) ) \cap \CS_\rd(\wt\fch_{\choices,\shmap}) \not=\emptyset$. Proposition~\ref{remainsstrong} implies that there is a next point of intersection of $\gamma$ and $\CS_\rd(\wt\fch_{\choices,\shmap})$ and that this is on $g.\CS'_\rd(\wt\ch)$ for a (uniquely determined) pair $(\wt\ch, g)\in\wt\fch_{\choices,\shmap}\times\Gamma$. We endow $v$ with the label $(\wt\ch,g)$. 

Suppose now that $\gamma( (0,\infty) ) \cap \CS_\rd(\wt\fch_{\choices,\shmap}) = \emptyset$. Then there is no next point of intersection of $\gamma$ and $\CS_\rd(\wt\fch_{\choices,\shmap})$. We label $v$ by $\eps$.

Let $\wh v\in \wh\CS_\rd(\wt\fch_{\choices,\shmap})$ and let $v\sceq\left( \pi\vert_{\CS'_\rd(\wt\fch_{\choices,\shmap})}\right)^{-1}(\wh v)$. The label of $\wh v$ and of each element in $\pi^{-1}(\wh v)$ is defined to be the label of $v$.

Suppose that $\Sigma_\rd$ denotes the set of labels of $\wh\CS_\rd(\wt\fch_{\choices,\shmap})$. \label{def_labelsred}

\begin{remark}
Recall from Corollary~\ref{prop_labels} that $\Sigma$ is finite. Then Proposition~\ref{remainsstrong} implies that also $\Sigma_\rd$ is finite. Moreover, Remark~\ref{sides_eff} shows that the elements of $\Sigma$ can be effectively determined. From Proposition~\ref{remainsstrong} then follows that also the elements of $\Sigma_\rd$ can be effectively determined.
\end{remark}

The following definition is analogous to the corresponding definitions in Section~\ref{sec_geomcodseq}. 

\begin{defi}
Let $v\in\CS'_\rd(\wt\fch_{\choices,\shmap})$ and suppose that $\gamma$ is the geodesic on $\h$ determined by $v$. Propositions~\ref{havenext} and \ref{remainsstrong} imply that there is a unique sequence $(t_n)_{n\in J}$ in $\R$ which satisfies the following properties:
\begin{enumerate}[(i)]
\item $J=\Z\cap (a,b)$ for some interval $(a,b)$ with $a,b\in\Z\cup\{\pm\infty\}$ and $0\in (a,b)$,
\item the sequence $(t_n)_{n\in J}$ is increasing,
\item $t_0=0$,
\item for each $n\in J$ we have $\gamma'(t_n)\in \CS_\rd(\wt\fch_{\choices,\shmap})$ and 
\[
\gamma'( (t_n,t_{n+1}) )\cap \CS_\rd(\wt\fch_{\choices,\shmap}) = \emptyset\quad\text{and}\quad\gamma'( (t_{n-1},t_n) ) \cap \CS_\rd(\wt\fch_{\choices,\shmap}) = \emptyset
\]
where we set $t_b\sceq \infty$ if $b<\infty$ and $t_a\sceq -\infty$ if $a>-\infty$.
\end{enumerate}
The sequence $(t_n)_{n\in J}$ is said to be the \textit{sequence of intersection times of $v$ with respect to $\CS_\rd(\wt\fch_{\choices,\shmap})$}. 

Let $\wh v\in \wh\CS_\rd(\wt\fch_{\choices,\shmap})$ and set $v\sceq \left( \pi\vert_{\CS'_\rd(\wt\fch_{\choices,\shmap})}\right)^{-1}(\wh v)$. Then the \textit{sequence of intersection times \wrt $\CS_\rd(\wt\fch_{\choices,\shmap})$ of $\wh v$} and \textit{of} each $w\in \pi^{-1}(\wh v)$ is defined to be the sequence of intersection times of $v$ \wrt $\CS_\rd(\wt\fch_{\choices,\shmap})$.

For each $s\in \Sigma_\rd$ set
\begin{align*}
\wh\CS_{\rd,s}\big(\wt\fch_{\choices,\shmap}\big) & \sceq \big\{ \wh v\in \wh\CS_\rd\big(\wt\fch_{\choices,\shmap}\big) \ \big\vert\ \text{$\wh v$ is labeled with $s$}\big\}
\intertext{and}
\CS_{\rd,s}\big(\wt\fch_{\choices,\shmap}\big) & \sceq \pi^{-1}\big(\wh\CS_\rd\big(\wt\fch_{\choices,\shmap}\big)\big) = \big\{ v\in \CS_\rd\big(\wt\fch_{\choices,\shmap}\big) \ \big\vert\ \text{$v$ is labeled with $s$}\big\}.
\end{align*}
Let $\wh v \in \wh\CS_\rd(\wt\fch_{\choices,\shmap})$ and let $(t_n)_{n\in J}$ be the sequence of intersection times of $\wh v$ \wrt $\CS_\rd(\wt\fch_{\choices,\shmap})$. Suppose that $\wh\gamma$ is the geodesic on $Y$ determined by $\wh v$. The \textit{reduced coding sequence} %
\index{reduced coding sequence}\index{coding sequence!reduced}%
of $\wh v$ is the sequence $(a_n)_{n\in J}$ in $\Sigma_\rd$ defined by 
\[
a_n\sceq s\quad \text{if and only if} \quad \wh\gamma'(t_n)\in \wh\CS_{\rd,s}\big(\wt\fch_{\choices,\shmap}\big)
\]
for each $n\in J$. 

Let $w\in \CS_\rd(\wt\fch_{\choices,\shmap})$. The \textit{reduced coding sequence} of $w$ is defined to be the reduced coding sequence of $\pi(w)$.

Let $\Lambda_\rd$  denote the set of reduced coding sequences and let $\Lambda_{\rd,\sigma}$ be the subset of $\Lambda_\rd$ consisting of the reduced coding sequences $(a_n)_{n\in (a,b)\cap \Z}$ with $a,b\in\Z\cup\{\pm\infty\}$ for which $b\geq 2$. Further, let $\Lambda_{\st,\rd}$ denote the set of two-sided infinite reduced coding sequences.  Let $\Sigma_\rd^\all$ be the set of all finite and one- or two-sided infinite sequences in $\Sigma_\rd$. Finally, let $\Seq_\rd\colon \wh\CS_\rd(\wt\fch_{\choices,\shmap}) \to \Lambda_\rd$ be the map which assigns to $\wh v\in \wh\CS_\rd(\wt\fch_{\choices,\shmap})$ the reduced coding sequence of $\wh v$. 
\end{defi}

The proofs of Propositions~\ref{redcommons} are analogous to those of the corresponding statements in Section~\ref{sec_geomcodseq}.

\begin{prop}\label{redcommons}
\begin{enumerate}[{\rm (1)}]
\item \label{redlocint}
Let $v\in\CS'_\rd(\wt\fch_{\choices,\shmap})$. Suppose that $(t_n)_{n\in J}$ is the sequence of intersection times of $v$ and that $(a_n)_{n\in J}$ is the reduced coding sequence of $v$. Let $\gamma$ be the geodesic on $\h$ determined by $v$. Suppose that $J=\Z\cap (a,b)$ with $a,b\in\Z\cup \{\pm \infty\}$.
\begin{enumerate}[{\rm (i)}]
\item If $b=\infty$, then $a_n\in \Sigma_\rd\mminus\{\eps\}$ for each $n\in J$.
\item If $b<\infty$, then $a_n\in \Sigma_\rd\mminus\{\eps\}$ for each $n\in (a,b-2]\cap\Z$ and $a_{b-1}=\eps$.
\item Suppose that $a_n= (\wt\ch_n,h_n)$ for $n\in (a,b-1)\cap\Z$ and set 
\begin{align*}
g_0 &\sceq h_0  && \text{if $b\geq 2$,}
\\
g_{n+1} & \sceq g_nh_{n+1} && \text{for $n\in [0,b-2)\cap\Z$,}
\\
g_{-1} & \sceq \id,
\\
g_{-(n+1)} & \sceq g_{-n}h_{-n}^{-1} && \text{for $n\in [1,-(a+1))\cap \Z$.}
\end{align*}
Then $\gamma'(t_{n+1}) \in g_n.\CS'_\rd(\wt\ch_n)$ for each $n\in (a,b-1)\cap\Z$.
\end{enumerate}
\item \label{redidentical}
Let $v,w\in\CS'_{\st,\rd}(\wt\fch_{\choices,\shmap})$. If the reduced coding sequences of $v$ and $w$ are equal, then $v=w$.
\item \label{redalli}
The left shift $\sigma\colon \Sigma_\rd^\all\to \Sigma_\rd^\all$ induces a partially defined map $\sigma\colon \Lambda_\rd\to\Lambda_\rd$ resp.\@ a map $\sigma\colon \Lambda_{\rd,\sigma}\to \Lambda_\rd$. Moreover, $\Lambda_{\st,\rd}\subseteq \Lambda_{\rd,\sigma}$ and $\sigma$ restricts to a map $\Lambda_{\st,\rd}\to \Lambda_{\st,\rd}$.
\item \label{redallii} The map $\Seq_\rd\vert_{\wh\CS_{\st,\rd}(\wt\fch_{\choices,\shmap})}\colon \wh\CS_{\st,\rd}(\wt\fch_{\choices,\shmap}) \to \Lambda_{\st,\rd}$ is bijective.
\item \label{redalliii} Let $\Cod_\rd\sceq \left( \Seq_\rd\vert_{\wh\CS_{\st,\rd}(\wt\fch_{\choices,\shmap})}\right)^{-1}$. Then the diagrams
\[
\xymatrix{
\wh\CS_\rd\big(\wt\fch_{\choices,\shmap}\big) \ar[r]^R \ar[d]_{\Seq_\rd} & \wh\CS_\rd\big(\wt\fch_{\choices,\shmap}\big)  \ar[d]^{\Seq_\rd} && \wh\CS_{\st,\rd}\big(\wt\fch_{\choices,\shmap}\big) \ar[r]^R & \wh\CS_{\st,\rd}\big(\wt\fch_{\choices,\shmap}\big) 
\\
\Lambda_\rd \ar[r]^\sigma & \Lambda_\rd && \Lambda_{\st,\rd} \ar[r]^\sigma \ar[u]^{\Cod_\rd} & \Lambda_{\st,\rd} \ar[u]_{\Cod_\rd}
}
\]
commute and $(\Lambda_{\st,\rd},\sigma)$ is a symbolic dynamics for the geodesic flow on $Y$.
\end{enumerate}
\end{prop}

We will now show that the reduced coding sequence of $\wh v\in \wh\CS_\rd(\wt\fch_{\choices,\shmap})$ can be completely constructed from the knowledge of the pair $\tau(\wh v)$.

\begin{defi}\label{def_Sigmach}
Let $\wt\ch\in\wt\fch_{\choices,\shmap}$. Define
\[
\Sigma_\rd\big(\wt\ch\big) \sceq \big\{ s\in\Sigma_\rd \ \big\vert\  \exists\, v\in\CS'_\rd\rueck\big(\wt\ch\big)\colon \text{$v$ is labeled with $s$} \big\}
\]
and for $s\in\Sigma_\rd(\wt\ch)$ set\label{def_Ds}
\[
 D_s\big(\wt\ch\big) \sceq I_\rd\big(\wt\ch\big) \cap g.I_\rd\big(\wt\ch'\big) \quad\text{if $s=\big(\wt\ch',g\big)$}
\]
and
\[
 D_\eps\big(\wt\ch\big) \sceq I_\rd\big(\wt\ch\big)\mminus \bigcup \big\{ D_s\big(\wt\ch\big) \ \big\vert\  s\in\Sigma_\rd\big(\wt\ch\big)\mminus\{\eps\}\big\}.
\]
\end{defi}

\begin{example}\label{redlabelsGamma05}
Recall the Example~\ref{intersectionGamma05}. Suppose first that the shift map is $\shmap_1$. 
Then we have
\begin{align*}
\Sigma_\rd\big(\wt\ch_1\big) & = \big\{ \eps, \big(\wt\ch_1,g_5\big), \big(\wt\ch_4,g_4\big), \big(\wt\ch_6,g_4\big), \big(\wt\ch_5,g_4\big), \big(\wt\ch_3,g_4\big)\big\},
\\
\Sigma_\rd\big(\wt\ch_2\big) & = \big\{ \eps, \big(\wt\ch_2,g_1\big), \big(\wt\ch_4,g_1\big), \big(\wt\ch_6,g_1\big), \big(\wt\ch_5,g_1\big), \big(\wt\ch_3,g_1\big) \big\},
\\
\Sigma_\rd\big(\wt\ch_3\big) & = \big\{ \eps, \big(\wt\ch_1,g_5\big), \big(\wt\ch_2,g_6\big), \big(\wt\ch_4,g_6\big), \big(\wt\ch_6,g_6\big), \big(\wt\ch_5,g_6\big), \big(\wt\ch_3,g_6\big) \big\},
\\
\Sigma_\rd\big(\wt\ch_4\big) & = \big\{ \eps, \big(\wt\ch_5,g_2\big), \big(\wt\ch_3,g_2\big) \big\},
\\
\Sigma_\rd\big(\wt\ch_5\big) & = \big\{ \eps, \big(\wt\ch_6,g_4\big), \big(\wt\ch_5,g_4\big), \big(\wt\ch_3,g_4\big)\big\},
\\
\Sigma_\rd\big(\wt\ch_6\big) & = \big\{ \eps, \big(\wt\ch_3,g_3\big) \big\}.
\end{align*}
Hence
\begin{align*}
D_{\big(\wt\ch_1,g_5\big)}\big(\wt\ch_1\big) & = \big(\tfrac45,1\big),
&
D_{\big(\wt\ch_4,g_4\big)}\big(\wt\ch_1\big) & = \big(-\infty, \tfrac35\big],
&
D_{\big(\wt\ch_6,g_4\big)}\big(\wt\ch_1\big) & =  \big(\tfrac35,\tfrac7{10}\big],
\\
D_{\big(\wt\ch_5,g_4\big)}\big(\wt\ch_1\big) & = \big(\tfrac7{10},\tfrac{11}{15}\big],
&
D_{\big(\wt\ch_3,g_4\big)}\big(\wt\ch_1\big) & = \big(\tfrac{11}{15},\tfrac45\big),
& 
D_{\eps}\big(\wt\ch_1\big) & =  \big\{\tfrac45\big\},
\intertext{and}
D_{\big(\wt\ch_2,g_1\big)}\big(\wt\ch_2\big) & = \big(0,\tfrac{1}{10}\big],
&
D_{\big(\wt\ch_4,g_1\big)}\big(\wt\ch_2\big) & = \big(\tfrac1{10},\tfrac{2}{15}\big],
&
D_{\big(\wt\ch_6,g_1\big)}\big(\wt\ch_2\big) & = \big(\tfrac2{15},\tfrac3{20}\big],
\\
D_{\big(\wt\ch_5,g_1\big)}\big(\wt\ch_2\big) & = \big(\tfrac3{20},\tfrac{4}{25}\big],
&
D_{\big(\wt\ch_3,g_1\big)}\big(\wt\ch_2\big) & = \big(\tfrac{4}{25},\tfrac15\big),
&
D_{\eps}\big(\wt\ch_2\big) & = \big\{\tfrac15\big\},
\intertext{and}
D_{\big(\wt\ch_1,g_5\big)}\big(\wt\ch_3\big) & = \big(\tfrac45,1\big),
&
D_{\big(\wt\ch_2,g_6\big)}\big(\wt\ch_3\big) & = \big(1,\tfrac65\big],
&
D_{\big(\wt\ch_4,g_6\big)}\big(\wt\ch_3\big) & = \big(\tfrac65,\tfrac75\big],
\\
D_{\big(\wt\ch_6,g_6\big)}\big(\wt\ch_3\big) & = \big(\tfrac75,\tfrac85\big],
&
D_{\big(\wt\ch_5,g_6\big)}\big(\wt\ch_3\big) & = \big(\tfrac85,\tfrac95\big],
&
D_{\big(\wt\ch_3,g_6\big)}\big(\wt\ch_3\big) & = \big(\tfrac95,\infty\big),
\\
D_{\eps}\big(\wt\ch_3\big) & = \{1\},
\intertext{and}
D_{\big(\wt\ch_5,g_2\big)}\big(\wt\ch_4\big) & = \big(\tfrac15,\tfrac3{10}\big],
&
D_{\big(\wt\ch_3,g_2\big)}\big(\wt\ch_4\big) & = \big(\tfrac{3}{10},\tfrac25\big),
&
D_{\eps}\big(\wt\ch_4\big) & = \big\{\tfrac25\big\},
\intertext{and}
D_{\big(\wt\ch_6,g_4\big)}\big(\wt\ch_5\big) & = \big(\tfrac35,\tfrac7{10}\big],
&
D_{\big(\wt\ch_5,g_4\big)}\big(\wt\ch_5\big) & = \big(\tfrac7{10},\tfrac{11}{15}\big],
&
D_{\big(\wt\ch_3,g_4\big)}\big(\wt\ch_5\big) & = \big(\tfrac{11}{15},\tfrac45\big),
\\
D_{\eps}\big(\wt\ch_5\big) & = \big\{\tfrac45\big\},
\intertext{and}
D_{\big(\wt\ch_3,g_3\big)}\big(\wt\ch_6\big) & = \big(\tfrac25,\tfrac35\big),
&
D_{\eps}\big(\wt\ch_6\big) & = \big\{\tfrac35\big\}.
\end{align*}
Suppose now that the shift map is $\shmap_2$. Then $\Sigma_\rd\big(\wt\ch_2\big), \Sigma_\rd\big(\wt\ch_4\big), \Sigma_\rd\big(\wt\ch_5\big)$ and $\Sigma_\rd\big(\wt\ch_6\big)$ are as for $\shmap_1$. The sets $D_{\ast}\big(\wt\ch_2\big), D_\ast\big(\wt\ch_4\big), D_\ast\big(\wt\ch_5\big)$ and $D_\ast\big(\wt\ch_6\big)$ remain unchanged as well. We have
\begin{align*}
\Sigma_\rd\big(\wt\ch_{-1}\big) & = \big\{\eps, \big(\wt\ch_{-1},g_7\big), \big(\wt\ch_4,g_7\big), \big(\wt\ch_6,g_7\big), \big(\wt\ch_5,g_7\big), \big(\wt\ch_3,g_7\big) \big\}
\intertext{and}
\Sigma_\rd\big(\wt\ch_3\big) & = \big\{ \eps, \big(\wt\ch_{-1},g_4\big), \big(\wt\ch_2,g_2\big), \big(\wt\ch_4,g_6\big), \big(\wt\ch_6,g_6\big), \big(\wt\ch_5,g_6\big), \big(\wt\ch_3,g_6\big)\big\}.
\end{align*}
Therefore
\begin{align*}
D_{\big(\wt\ch_{-1},g_7\big)}\big(\wt\ch_{-1}\big) & =\big(-\tfrac15,0\big),
&
D_{\big(\wt\ch_4,g_7\big)}\big(\wt\ch_{-1}\big) & = \big(-\infty,-\tfrac25\big),
\\
D_{\big(\wt\ch_6,g_7\big)}\big(\wt\ch_{-1}\big) & = \big(-\tfrac25,-\tfrac{3}{10}\big),
&
D_{\big(\wt\ch_5,g_7\big)}\big(\wt\ch_{-1}\big) & = \big(-\tfrac{3}{10},-\tfrac{4}{15}\big),
\\
D_{\big(\wt\ch_3,g_7\big)}\big(\wt\ch_{-1}\big) & = \big(-\tfrac{4}{15},-\tfrac15\big),
&
D_{\eps}\big(\wt\ch_{-1}\big) & = \big\{-\tfrac15\big\},
\intertext{and}
D_{\big(\wt\ch_{-1},g_4\big)}\big(\wt\ch_3\big) & = \big(\tfrac45,1\big),
&
D_{\big(\wt\ch_2,g_6\big)}\big(\wt\ch_3\big) & = \big(1,\tfrac65\big],
\\
D_{\big(\wt\ch_4,g_6\big)}\big(\wt\ch_3\big) & = \big(\tfrac65,\tfrac75\big],
&
D_{\big(\wt\ch_6,g_6\big)}\big(\wt\ch_3\big) & = \big(\tfrac75,\tfrac85\big],
\\
D_{\big(\wt\ch_5,g_6\big)}\big(\wt\ch_3\big) & = \big(\tfrac85,\tfrac95\big],
&
D_{\big(\wt\ch_3,g_6\big)}\big(\wt\ch_3\big) & = \big(\tfrac95,\infty\big),
\\
D_{\eps}\big(\wt\ch_3\big) & = \{1\}.
\end{align*}
\end{example}

The next corollary follows immediately from Proposition~\ref{remainsstrong}.

\begin{cor}\label{reddecomp}
Let $\wt\ch\in\wt\fch_{\choices,\shmap}$. Then $I_\rd(\wt\ch)$ decomposes into the disjoint union $\bigcup\{ D_s(\wt\ch) \mid s\in\Sigma_\rd(\wt\ch)\}$. Let $v\in \CS'_\rd(\wt\ch)$ and suppose that $\gamma$ is the geodesic on $\h$ determined by $v$. Then $v$ is labeled with $s$ if and only if $\gamma(\infty)$ belongs to $D_s(\wt\ch)$.
\end{cor}

Our next goal is to find a discrete dynamical system on the geodesic boundary of $\h$ which is conjugate to $(\wh \CS_\rd(\wt\fch_{\choices,\shmap}), R)$. To that end we set 
\[
 \wt \DS\sceq \bigcup_{\wt\ch\in\wt\fch_{\choices,\shmap}} I_\rd\big(\wt\ch\big)\times J\big(\wt\ch\big).
\]
For $\wt\ch\in\wt\fch_{\choices,\shmap}$ and $s\in\Sigma_\rd(\wt\ch)$ we set 
\[
 \wt D_s\big(\wt\ch\big) \sceq D_s\big(\wt\ch\big) \times J\big(\wt\ch\big).
\]
We define the partial map $\wt F\colon \wt \DS\to \wt \DS$ by 
\[
\wt  F\vert_{\wt D_s(\wt\ch)} (x,y) \sceq (g^{-1}.x,g^{-1}.y)
\]
if $s=(\wt\ch',g)\in \Sigma_\rd(\wt\ch)$ and $\wt\ch\in\wt\fch_{\choices,\shmap}$.

Recall the map 
\[
\tau\colon\wh\CS_\rd\big(\wt\fch_{\choices,\shmap}\big)  \to  \bhg \h\times \bhg \h,\quad                  
\wh v  \mapsto  \big(\gamma_v(\infty), \gamma_v(-\infty)\big)
\]
where $v\sceq \big(\pi\vert_{\CS'_\rd(\wt\fch_{\choices,\shmap})}\big)^{-1}(\wh v)$ and $\gamma_v$ is the geodesic on $\h$ determined by $v$.

\begin{prop}\label{Ftilde}
The set $\wt\DS$ is the disjoint union 
\[
 \wt\DS = \bigcup\left\{ \wt D_s\big(\wt\ch\big) \left\vert\  \wt\ch\in\wt\fch_{\choices,\shmap},\ s\in\Sigma_\rd\big(\wt\ch\big)\right.\right\}.
\]
If $(x,y)\in \wt\DS$, then there is a unique element $v\in\CS'_\rd(\wt\fch_{\choices,\shmap})$ such that 
\[
\big(\gamma_v(\infty), \gamma_v(-\infty)\big)=(x,y).
\]
If and only if $(x,y)\in\wt D_s(\wt\ch)$, the element $v$ is labeled with $s$.
Moreover, the partial map $\wt F$ is well-defined and the discrete dynamical system $(\wt \DS, \wt F)$ is conjugate to $(\wh \CS_\rd(\wt\fch_{\choices,\shmap}), R)$ via $\tau$.
\end{prop}

\begin{proof}
The sets $I_\rd(\wt\ch)\times J(\wt\ch)$, $\wt\ch\in\wt\fch_{\choices,\shmap}$, are pairwise disjoint by construction. Corollary~\ref{reddecomp} states that for each $\wt\ch\in\wt\fch_{\choices,\shmap}$ the set $I_\rd(\wt\ch)$ is the disjoint union $\bigcup\{ D_s(\wt\ch)\mid s\in\Sigma_\rd(\wt\ch)\}$. Therefore, the sets $\wt D_s(\wt\ch)$, $\wt\ch\in\wt\fch_{\choices,\shmap}$, $s\in\Sigma_\rd(\wt\ch)$, are pairwise disjoint and
\[
 \wt\DS = \bigcup \left\{ \wt D_s\big(\wt\ch\big) \left\vert\  \wt\ch\in\wt\fch_{\choices,\shmap},\ s\in \Sigma_\rd\big(\wt\ch\big) \right.\right\}.
\]
This implies that $\wt F$ is well-defined. Let $(x,y)\in \wt\DS$. By Lemma~\ref{uniqueinter} there is a unique $\wt\ch\in\wt\fch_{\choices,\shmap}$ and a unique $v\in\CS'_\rd(\wt\ch)$ such that $\big( \gamma_v(\infty), \gamma_v(-\infty) \big) = (x,y)$. Corollary~\ref{reddecomp} shows that $v$ is labeled with $s\in\Sigma_\rd$ if and only if $\gamma_v(\infty) \in D_s(\wt\ch)$, hence if $(x,y)\in \wt D_s(\wt\ch)$. It remains to show that $(\wt\DS,\wt F)$ is conjugate to $(\wh\CS_\rd(\wt\fch_{\choices,\shmap}), R)$ by $\tau$. Lemma~\ref{uniqueinter} shows that $\tau$ is a bijection between $\wh\CS_\rd(\wt\fch_{\choices,\shmap})$ and $\wt\DS$. Let $\wh v\in \wh\CS_\rd$ and $v\sceq \big( \pi\vert_{\CS'_\rd(\wt\fch_{\choices,\shmap})} \big)^{-1}(\wh v)$. Suppose that $\wt\ch\in\wt\fch_{\choices,\shmap}$ is the (unique) shifted cell in $S\h$ such that $v\in \CS'_\rd(\wt\ch)$, and let $(s_j)_{j\in(\alpha,\beta)\cap\Z}$ be the reduced coding sequence of $v$. Recall that $s_0$ is the label of $v$ and $\wh v$. 
Corollary~\ref{reddecomp} shows that $\gamma_v(\infty) \in D_{s_0}(\wt\ch)$. The map $R$ is defined for $\wh v$ if and only if $s_0\not=\eps$. In precisely this case, $\wt F$ is defined for $\tau(\wh v)$. 

Suppose that $s_0\not=\eps$, say $s_0=(\wt\ch',g)$. Then the next intersection of $\gamma_v$ and $\CS_\rd(\wt\fch_{\choices,\shmap})$ is on $g.\CS'_\rd(\wt\ch')$, say it is $w$. Then $R(\wh v) = \pi(w) \seqc \wh w$ and $\big(\pi\vert_{\CS'_\rd(\wt\fch_{\choices,\shmap})} \big)^{-1}(\wh w) = g^{-1}.w$. Let $\eta$ be the geodesic on $\h$ determined by $g^{-1}.w$.  We have to show that $\wt F(\tau(\wh v)) = \big(\eta(\infty),\eta(-\infty)\big)$. To that end note that  $g.\eta(\R) = \gamma_v(\R)$ and  hence $\big( \eta(\infty), \eta(-\infty) \big) = \big( g^{-1}.\gamma_v(\infty), g^{-1}.\gamma_v(-\infty) \big)$. Since $\tau(\wh v)\in \wt D_{s_0}(\wt\ch)$, the definition of $\wt F$ shows that 
\[
\wt F(\tau(\wh v)) = \wt F\big( (\gamma_v(\infty),\gamma_v(-\infty)) \big) = \big( g^{-1}.\gamma_v(\infty), g^{-1}.\gamma_v(-\infty)\big) = \big(\eta(\infty),\eta(-\infty)\big).
\]
Thus, $(\wt\DS, \wt F)$ is conjugate to $(\wh\CS_\rd(\wt\fch_{\choices,\shmap}), R)$ by $\tau$.
\end{proof}

The following corollary proves that we can reconstruct the future part of the reduced coding sequence of $\wh v\in 
\wh \CS_\rd(\wt\fch_{\choices,\shmap})$ from $\tau(\wh v)$.

\begin{cor}
Let $\wh v\in  \wh \CS_\rd(\wt\fch_{\choices,\shmap})$ and suppose that $(s_j)_{j\in J}$ is the reduced coding sequence of $v$. Then 
\[
s_j = s \qquad\text{if and only if}\qquad \wt F^j\big(\tau(\wh v)\big)\in \wt D_s\big(\wt\ch\big)\ \text{for some $\wt\ch\in\wt\fch_{\choices,\shmap}$}
\]
for each $j\in J\cap \N_0$. For $j\in\N_0\mminus J$, the map $\wt F^j$ is not defined for $\tau(\wh v)$.
\end{cor}

The next proposition shows that we can also reconstruct the past part of the reduced coding sequence of $\wh v\in \wh \CS_\rd(\wt\fch_{\choices,\shmap})$ from $\tau(\wh v)$. Its proof is constructive.

\begin{prop}\label{back}
\begin{enumerate}[{\rm (i)}]
\item \label{backi} The elements of 
\[
\left\{ g^{-1}.\wt D_{(\wt\ch,g)}\big(\wt\ch'\big) \left\vert\ \wt\ch'\in\wt\fch_{\choices,\shmap},\ \big(\wt\ch,g\big)\in\Sigma_\rd\big(\wt\ch'\big)\right.\right\}
\]
are pairwise disjoint.
\item \label{backii} Let $\wh v\in \wh\CS_\rd(\wt\fch_{\choices,\shmap})$ and suppose that $(a_j)_{j\in J}$ is the reduced coding sequence of $\wh v$. Then $-1\in J$ if and only if 
\[
\tau(\wh v) \in \bigcup \left\{ g^{-1}.\wt D_{(\wt\ch,g)}\big(\wt\ch'\big) \left\vert\ \wt\ch'\in\wt\fch_{\choices,\shmap},\ \big(\wt\ch,g\big)\in\Sigma_\rd\big(\wt\ch'\big)\right.\right\}.
\]
In this case, 
\begin{align*}
a_{-1} = \big(\wt\ch,g\big)\qquad\text{if and only if}\qquad \tau(\wh v) \in g^{-1}.\wt D_{(\wt\ch,g)}\big(\wt\ch'\big)
\end{align*}
for some $\wt\ch'\in\wt\fch_{\choices,\shmap}$ and $(\wt\ch,g)\in\Sigma_\rd(\wt\ch')$.
\end{enumerate}
\end{prop}

\begin{proof}
We will prove \eqref{backii}, which directly implies \eqref{backi}.  To that end set 
\[
v\sceq \left(\pi\vert_{\CS'_\rd(\wt\fch_{\choices,\shmap})}\right)^{-1}(\wh v)
\]
and suppose that $(t_j)_{j\in J}$ is the sequence of intersection times of $v$ with respect to $\CS_\rd(\wt\fch_{\choices,\shmap})$. 

Suppose first that $v\in \CS'_\rd(\wt\ch)$ and that $-1\in J$. Then there exists a (unique) pair $(\wt\ch',g)\in\wt\fch_{\choices,\shmap}\times\Gamma$ such that $\gamma'_v(t_{-1})\in g^{-1}.\CS'_\rd(\wt\ch')$. Since the unit tangent vector $\gamma'_v(t_0)=v$ is contained in $\CS'_\rd(\wt\ch)$, the element $\gamma'_v(t_{-1})$ is labeled with $(\wt\ch,g)$. Hence $(\wt\ch,g)\in\Sigma_\rd(\wt\ch')$. Then 
\begin{align*}
\tau(\wh v) = \big(\gamma_v(\infty),\gamma_v(-\infty)\big) & \in \left(g^{-1}.I_\rd\big(\wt\ch'\big) \times g^{-1}.J\big(\wt\ch'\big) \right) \cap \left( I_\rd\big(\wt\ch\big)\times J\big(\wt\ch\big) \right)
\\
& = \left(  g^{-1}. I_\rd\big(\wt\ch'\big) \cap I_\rd\big(\wt\ch\big) \right) \times \left( g^{-1}.J\big(\wt\ch'\big) \cap J\big(\wt\ch\big) \right)
\\
& = g^{-1}.\left( \left(I_\rd\big(\wt\ch'\big)\cap g I_\rd\big(\wt\ch\big) \right) \times\left( J\big(\wt\ch\big)\cap g.J\big(\wt\ch\big) \right)
 \right)
\\
& \subseteq g^{-1}.\wt D_{(\wt\ch,g)}\big(\wt\ch'\big).
\end{align*}
Conversely suppose that $\tau(\wh v)\in g^{-1}.\wt D_{(\wt\ch,g)}(\wt\ch')$ for some $\wt\ch'\in\wt\fch_{\choices,\shmap}$ and some element $(\wt\ch,g)\in \Sigma_\rd(\wt\ch')$. Consider the geodesic $\eta\sceq g.\gamma_v$. Then 
\[
\big( \eta(\infty),\eta(-\infty) \big) = g.\tau(\wh v)\in \wt D_{(\wt\ch,g)}\big(\wt\ch'\big).
\]
By Proposition~\ref{Ftilde}, there is a unique $u\in\CS'_\rd(\wt\fch_{\choices,\shmap})$ such that 
\[
 \big( \gamma_u(\infty), \gamma_u(-\infty)\big) = \big(\eta(\infty),\eta(-\infty)\big).
\]
Moreover, $u$ is labeled with $(\wt\ch,g)$. Let $(s_k)_{k\in K}$ be the sequence of intersection times of $u$ \wrt $\CS_\rd(\wt\fch_{\choices,\shmap})$. Then $1\in K$ and, by Proposition~\ref{Ftilde},
\begin{align*}
\tau\big(\pi(\gamma'_u(s_1))\big) & = \tau\big(R(\pi(u))\big) = \wt F\big(\tau(\pi(u))\big) = \wt F\big(\gamma_u(\infty), \gamma_u(-\infty)\big)
\\
& = \big(g^{-1}.\gamma_u(\infty), g^{-1}.\gamma_u(-\infty) \big) = \big( \gamma_v(\infty),\gamma_v(-\infty) \big) = \tau(\wh v).
\end{align*}
This shows that $\gamma'_u(s_1) = g.v = g.\gamma'_v(t_0)$. Then 
\[
g^{-1}.\gamma'_u(s_0) \in \gamma'_v\big((-\infty,0)\big)\cap \CS_\rd\big(\wt\fch_{\choices,\shmap}\big).
\]
Hence, there was a previous point of intersection of $\gamma_v$ and $\CS_\rd(\wt\fch_{\choices,\shmap})$ and this is $g^{-1}.\gamma'_u(s_0)$. Recall that $g^{-1}.\gamma'_u(s_0)$ is labeled with $(\wt\ch,g)$. This completes the proof.
\end{proof}

Let $\wt F_\bk\colon \wt\DS\to\wt\DS$ be the partial map defined by 
\[
\wt F_\bk\vert_{g^{-1}\wt D_{(\wt\ch,g)}(\wt\ch')}(x,y) \sceq (g.x,g.y)
\]
for $\wt\ch'\in\wt\fch_{\choices,\shmap}$ and $(\wt\ch,g)\in\Sigma_\rd(\wt\ch')$.

\begin{cor}
\begin{enumerate}[{\rm (i)}]
\item The partial map $\wt F_\bk$ is well-defined.
\item Let $\wh v\in\wh\CS_\rd(\wt\fch_{\choices,\shmap})$ and suppose that $(s_j)_{j\in J}$ is the reduced coding sequence of $\wh v$. For each $j\in J\cap (-\infty, -1]$ and each $(\wt\ch,g)\in\Sigma_\rd$ we have
\[
s_j  = (\wt\ch,g)\qquad \text{if and only if}\qquad \wt F_\bk^j\big(\tau(\wh v)\big) \in g^{-1}.\wt D_{ (\wt\ch,g) }\big(\wt\ch'\big)
\]
for some $\wt\ch'\in\wt\fch_{\choices,\shmap}$.
For $j\in \Z_{<0}\mminus J$, the map $\wt F_\bk^j$ is not defined for $\tau(\wh v)$.
\end{enumerate}
\end{cor}

We end this section with the statement of the discrete dynamical system which is conjugate to the strong reduced cross section $\wh\CS_{\st,\rd}(\wt\fch_{\choices,\shmap})$.

The set of labels of $\wh\CS_{\st,\rd}(\wt\fch_{\choices,\shmap})$ is given by 
\[
 \Sigma_{\st,\rd} \sceq \Sigma_\rd\mminus\{\eps\}.
\]
For each $\wt\ch\in\wt\fch_{\choices,\shmap}$ set
\[
 \Sigma_{\st,\rd}\big(\wt\ch\big) \sceq \Sigma_\rd\big(\wt\ch\big)\mminus\{\eps\}.
\]
Recall the set $\bd$ from Section~\ref{sec_base}. For $s\in \Sigma_{\st,\rd}(\wt\ch)$ set 
\begin{align*}
D_{\st, s}\big(\wt\ch\big) &\sceq D_s\big(\wt\ch\big)\mminus\bd
\intertext{and}
\wt D_{\st,s}\big(\wt\ch\big) & \sceq D_{\st,s}\big(\wt\ch\big) \times \big(J\big(\wt\ch\big)\mminus\bd\big).
\end{align*}
Further let
\[
\wt\DS_\st\sceq \bigcup_{\wt\ch\in\wt\fch_{\choices,\shmap}}\big( I_\rd\big(\wt\ch\big)\mminus\bd\big) \times \big(J\big(\wt\ch\big)\mminus\bd\big)
\]
and define the map $\wt F_\st\colon \wt\DS_\st\to \wt\DS_\st$ by
\[
\wt F_\st\vert_{\wt D_{\st,s}(\wt\ch)}(x,y) \sceq (g^{-1}.x,g^{-1}.y)
\]
if $s=(\wt\ch',g)\in \Sigma_{\st,\rd}(\wt\ch)$ and $\wt\ch\in\wt\fch_{\choices,\shmap}$. The map $\wt F_\st$ is the ``restriction'' of $\wt F$ to the strong reduced cross section $\wh\CS_{\st,\rd}(\wt\fch_{\choices,\shmap})$. In particular, the following proposition is the ``reduced'' analogon of Proposition~\ref{Ftilde}.

\begin{prop}
\begin{enumerate}[{\rm (i)}]
\item The set $\wt\DS_\st$ is the disjoint union
\[
\wt\DS_\st = \bigcup\left\{ \wt\D_{\st,s}\big(\wt\ch\big)   \left\vert\   \wt\ch\in\wt\fch_{\choices,\shmap},\ s\in \Sigma_{\st,\rd}\big(\wt\ch\big)    \vphantom{\wt\D_{\st,s}\big(\wt\ch\big) }       \right.\right\}.
\]
If $(x,y)\in\wt\DS_\st$, then there is a unique element $v\in \CS'_{\st,\rd}(\wt\fch_{\choices,\shmap})$ such that $\big(\gamma_v(\infty),\gamma_v(-\infty)\big) = (x,y)$. If and only if $(x,y)\in\wt D_{\st,s}(\wt\ch)$, the element $v$ is labeled with $s$.
\item The map $\wt F_\st$ is well-defined and the discrete dynamical system $\big(\wt\DS_\st, \wt F_\st\big)$ is conjugate to $\big(\wh\CS_{\st,\rd}(\wt\fch_{\choices,\shmap}),R\big)$ via $\tau$.
\end{enumerate}
\end{prop}

\subsection{Generating function for the future part}\label{sec_gen}

Suppose that the sets $I_\rd(\wt\ch)$, $\wt\ch\in\wt\fch_{\choices,\shmap}$, are pairwise disjoint. Set 
\[
 \DS\sceq \bigcup_{\wt\ch\in\wt\fch_{\choices,\shmap}} I_\rd\big(\wt\ch\big)
\]
and consider the partial map $F\colon \DS\to\DS$ given by
\[
 F\vert_{D_s(\wt\ch)} x \sceq g^{-1}.x
\]
if $s=(\wt\ch',g)\in \Sigma_\rd(\wt\ch)$ and $\wt\ch\in\wt\fch_{\choices,\shmap}$.

\begin{prop}\label{gen}
\begin{enumerate}[{\rm (i)}]
\item\label{geni} The set $\DS$ is the disjoint union 
\[
\DS = \bigcup\left\{ D_s\big(\wt\ch\big) \left\vert\ \wt\ch\in\wt\fch_{\choices,\shmap},\ s\in\Sigma_\rd\big(\wt\ch\big) \vphantom{ D_s\big(\wt\ch\big)}   \right.\right\}.
\]
If $x\in\DS$, then there is (a non-unique) $v\in\CS'_\rd(\wt\fch_{\choices,\shmap})$ such that $\gamma_v(\infty)=x$. Suppose that $v\in \CS'_\rd(\wt\fch_{\choices,\shmap})$ with $\gamma_v(\infty)=x$ and let $(a_n)_{n\in J}$ be the reduced coding sequence of $v$. Then $a_0=s$ if and only if $x\in D_s(\wt\ch)$ for some $\wt\ch\in\wt\fch_{\choices,\shmap}$.
\item The partial map $F$ is well-defined.
\end{enumerate}
\end{prop}

\begin{proof}
Suppose that $\wt\ch_1,\wt\ch_2\in\wt\fch_{\choices,\shmap}$ and $s_1\in\Sigma_\rd(\wt\ch_1)$, $s_2\in\Sigma_\rd(\wt\ch_2)$ such that $D_{s_1}(\wt\ch_1)\cap D_{s_2}(\wt\ch_2)\not=\emptyset$. Pick $x\in D_{s_1}(\wt\ch_1)\cap D_{s_2}(\wt\ch_2)$. Then $x\in I_\rd(\wt\ch_1) \cap I_\rd(\wt\ch_2)$, which implies that $\wt\ch_1=\wt\ch_2$. Now Corollary~\ref{reddecomp} yields that $s_1=s_2$. Therefore the union in \eqref{geni} is disjoint and hence $F$ is well-defined. Corollary~\ref{reddecomp} shows that the union equals $\DS$. 

Let $(x,y)\in\wt\DS$. Then $(x,y)\in\wt D_s(\wt\ch)$ if and only if $x\in D_s(\wt\ch)$. Proposition~\ref{Ftilde} implies the remaining statements of \eqref{geni}.
\end{proof}

Proposition~\ref{gen} shows that 
\[
\Big(F, (\DS_s(\wt\ch))_{\wt\ch\in\wt\fch_{\choices,\shmap}, s\in \Sigma_\rd(\wt\ch)}\Big)
\]
is like a generating function for the future part of the symbolic dynamics $(\Lambda_\rd, \sigma)$. In comparison with a real generating function, the map $i\colon \Lambda_\rd\to \DS$ is missing. Indeed, if there are strip precells in $\h$, then there is no unique choice for the map $i$. To overcome this problem, we restrict ourselves to the strong reduced cross section $\wh\CS_{\st,\rd}(\wt\fch_{\choices,\shmap})$. 

\begin{prop}\label{onlyforward}
Let $v,w\in \CS'_{\st,\rd}(\wt\fch_{\choices,\shmap})$. Suppose that $(a_n)_{n\in\Z}$ is the reduced coding sequence of $v$  and $(b_n)_{n\in\Z}$ that of $w$. If $(a_n)_{n\in \N_0} = (b_n)_{n\in\N_0}$, then $\gamma_v(\infty) = \gamma_w(\infty)$.
\end{prop}

\begin{proof}
The proof of Proposition~\ref{geomsequnique} shows the corresponding statement for geometric coding sequences. The proof of the present statement is analogous.
\end{proof}

We set 
\[
 \DS_\st \sceq \bigcup_{\wt\ch\in\wt\fch_{\choices,\shmap}} I_\rd\big(\wt\ch\big)\mminus\bd
\]
and define the map $F_\st\colon \DS_\st\to \DS_\st$ by 
\[
 F_\st\vert_{D_{\st,s}(\wt\ch)} x \sceq g^{-1}.x
\]
if $s=(\wt\ch',g)\in \Sigma_{\st,\rd}(\wt\ch)$ and  $\wt\ch\in\wt\fch_{\choices,\shmap}$. Further define $i\colon \Lambda_{\st,\rd} \to \DS_\st$ by 
\[
 i\big( (a_n)_{n\in\Z} \big) \sceq \gamma_v(\infty),
\]
where $v\in \CS'_{\st,\rd}(\wt\fch_{\choices,\shmap})$ is the unit tangent vector with reduced coding sequence $(a_n)_{n\in\N}$. Proposition~\ref{redidentical} shows that $v$ is unique, and Proposition~\ref{onlyforward} shows that $i$ only depends on $(a_n)_{n\in\N_0}$. Therefore 
\[
\left( F_\st, i, (D_s(\wt\ch))_{\wt\ch\in\wt\fch_{\choices,\shmap},s\in \Sigma_{\st,\rd}(\wt\ch)}\right)
\]
is a generating function for the future part of the symbolic dynamics $(\Lambda_{\st,\rd},\sigma)$.

\begin{example}\label{genfutHecke}
For the Hecke triangle group $G_n$ and its family of shifted cells $\wt\fch_{\choices,\shmap} = \{ \wt\ch\}$ from Example~\ref{labelsHecke} we have $I(\wt\ch) = I_\rd(\wt\ch)$ and $\Sigma_\rd = \Sigma$. Obviously, the associated symbolic dynamics $(\Lambda,\sigma)$ has a generating function for the future part. Recall the set $\bd$ from Section~\ref{sec_base}. Here we have $\bd = G_n.\infty$. Then
\[
 \DS = \R^+ \quad\text{and}\quad \DS_\st = \R^+\mminus G_n.\infty.
\]
Since there is only one (shifted) cell in $S\h$, we omit $\wt\ch$ from the notation in the following. We have
$D_g = (g.0,g.\infty)$ and
\[
D_{\st,g} = (g.0,g.\infty)\mminus G_n.\infty 
\quad\text{for $g\in \big\{ U_n^kS \ \big\vert\  k=1,\ldots, n-1 \big\}$.} 
\]
The generating function for the future part of $(\Lambda,\sigma)$ is $F\colon \DS \to \DS$, 
\[
F\vert_{D_g} x \sceq g^{-1}. x \quad\text{for $g\in \big\{ U_n^kS \ \big\vert\  k=1,\ldots, n-1\big\}$.}. 
\]
For the symbolic dynamics $(\Lambda_\st, \sigma)$ arising from the strong cross section $\wh\CS_\st$ the generating function for the future part is $F_\st\colon \DS_\st\to\DS_\st$, 
\[
F_\st\vert_{D_{\st,g}} x \sceq g^{-1}.x \quad\text{for $g\in \big\{ U_n^kS \ \big\vert\  k=1,\ldots, n-1 \big\}$.}
\]
\end{example}

\begin{example}
Recall Example~\ref{redsetGamma05}. If the shift map is $\shmap_2$, then the sets $I_\rd(\cdot)$ are pairwise disjoint and hence there is a generating function for the future part of the symbolic dynamics. In contrast, if the shift map is $\shmap_1$, the sets $I_\rd(\cdot)$ are not disjoint. Suppose that $\gamma$ is a geodesic on $H$ such that $\gamma(\infty) = \tfrac12$. Then $\gamma$ intersects $\CS'_\rd\big(\wt\fch_{\choices,\shmap_1}\big)$ in, say, $v$. Example~\ref{redlabelsGamma05} shows that one cannot decide whether $v\in\CS'\big(\wt\ch_6\big)$ and hence is labeled with $\big(\wt\ch_3,g_3\big)$, or whether $v\in\CS'\big(\wt\ch_1\big)$ and thus is labeled with $\big(\wt\ch_4,g_4\big)$. This shows that the symbolic dynamics arising from the shift map $\shmap_1$ does not have a generating function for the future part.
\end{example}

\section{Transfer Operators}\label{sec_transop}

Suppose that $(X,f)$ is a discrete dynamical system, where $X$ is a set and $f$ is a self-map of $X$. Further let $\psi\colon X\to \C$ be a function. The transfer operator $\mc L$ of $(X,f)$ with potential $\psi$ is defined by
\[
\mc L \varphi(x) = \sum_{y\in f^{-1}(x)} e^{\psi(y)} \varphi(y)
\]
with some appropriate space of complex-valued functions on $X$ as domain of definition. They originate from thermodynamic formalism. For more background information, we refer to \eg \cite[Section~14]{Chaosbook}, \cite{Ruelle}, \cite{Mayer_survey}. 

Here, we only consider potentials of the form
\[
 \psi(y) = -\beta \log |f'(y)|,
\]
where $\beta\in \C$.

Let $\Gamma$ be a geometrically finite subgroup of $\PSL(2,\R)$ of which $\infty$ is a cuspidal point and which satisfies \eqref{A4}. Suppose that the set of relevant isometric spheres is non-empty. Fix a basal family $\fpch$ of precells in $\h$ and let $\fch$ be the family of cells in $\h$ assigned to $\fpch$. Let $\choices$ be a set of choices associated to $\fpch$ and suppose that $\wt\fch_\choices$ is the family of cells in $S\h$ associated to $\fpch$ and $\choices$. Let $\shmap$ be a shift map for $\wt\fch_\choices$ and let $\wt\fch_{\choices,\shmap}$ denote the family of cells in $S\h$ associated to $\fpch$, $\choices$ and $\shmap$.

We restrict ourselves to the strong reduced cross section $\wh\CS_{\st,\rd}(\wt\fch_{\choices,\shmap})$. Recall the discrete dynamical system $(\DS_\st, \wt F_\st)$ from Section~\ref{sec_redcodseq} as well as the set $\Sigma_{\st,\rd}$, its subsets $\Sigma_{\st,\rd}(\wt\ch)$ and the sets $\wt D_{\st,s}(\wt\ch)$. The local inverses of $\wt F_\st$ are
\[
\wt F_{\wt\ch,s} \sceq \left( \wt F_\st\vert_{\wt D_{\st,s}\big(\wt\ch\big)}\right)^{-1} \colon\left\{
\begin{array}{ccl}
\wt F_{\st}\big( \wt D_{\st,s}\big(\wt\ch\big)\big) & \to & \wt D_{\st,s}\big(\wt\ch\big)
\\
(x,y) & \mapsto & (g.x,g.y)
\end{array}
\right.
\]
for $s\in \Sigma_{\st,\rd}(\wt\ch)$, $\wt\ch\in\wt\fch_{\choices,\shmap}$. To abbreviate, we set
\[
 \wt E_{\wt\ch,s}  \sceq \wt F_{\st}\big(\wt D_{\st,s}\big(\wt\ch\big)\big)
\]
for $s\in \Sigma_{\st,\rd}(\wt\ch)$, $\wt\ch\in \wt\fch_{\choices,\shmap}$.

For two sets $M, N$ let $\Fct(M,N)$ denote the set of functions from $M$ to $N$. 
Then the transfer operator with parameter $\beta\in \C$ 
\[
\mc L_\beta \colon \Fct(\wt D_{\st},\C) \to \Fct(\wt D_{\st},\C)
\]
associated to $\wt F_\st$ is given by 
\[
\left( \mc L_\beta \varphi \right) (x,y) \sceq \sum_{\wt\ch\in \wt\fch_{\choices,\shmap}} \sum_{s\in \Sigma_{\st,\rd}(\wt\ch)} \big| \det\big( \wt F_{\wt\ch,s}'(x,y)\big)\big|^\beta \varphi\big( \wt F_{\wt\ch,s}(x,y) \big) 1_{\wt E_{\wt\ch,s}}(x,y),
\]
where $1_{\wt E_{\wt\ch,s}}$ is the characteristic function of $\wt E_{\wt\ch,s}$. Note that the set $\wt E_{\wt\ch,s}$, the domain of definition of $\wt F_{\wt\ch,s}$, in general is not open. A priori, it is not even clear whether $E_{\wt\ch,s}$ is dense in itself. To avoid any problems with well-definedness, the derivative of $\wt F_{\wt\ch,s}$ shall be defined as the restriction to $\wt E_{\wt\ch,s}$ of the derivative of $F_{\wt\ch}\colon \wt F(\wt D_s) \to \wt D_s$, $(x,y)\mapsto (g.x,g.y)$. Moreover, the maps $\wt F_{\wt\ch,s}$ and $\wt F'_{\wt\ch,s}$ are extended arbitrarily on $\wt D_{\st}\mminus \wt E_{\wt\ch,s}$.

Let $\beta\in\C$ and consider the map
\[
\tau_{2,\beta}\colon\Gamma \times \Fct(\wt D_\st,\C)  \to  \Fct(\wt D_\st, \C),\quad (g,\varphi)  \mapsto  \tau_{2,\beta}(g)\varphi
\]
where
\[
 \tau_{2,\beta}(g^{-1})\varphi(x,y) \sceq |g'(x)|^\beta |g'(y)|^\beta \varphi(g.x,g.y).
\]
The map $\tau_{2,\beta}$ is an action of $\Gamma$ on $\Fct(\wt D_\st,\C)$, and reminiscent of principal series representations. The transfer operator becomes
\[
\mc L_\beta\varphi = \sum_{\wt\ch\in\wt\fch_{\choices,\shmap}} \sum_{(\wt\ch',g)\in \Sigma_{\st,\rd}(\wt\ch)} 1_{\wt E_{\wt\ch,(\wt\ch',g)}}\cdot \tau_{2,\beta}(g^{-1})\varphi.
\]
Suppose now that the sets $I_\rd(\wt\ch)$, $\wt\ch\in \wt\fch_{\choices,\shmap}$, are pairwise disjoint so that the map $F_\st$ from Section~\ref{sec_gen} is a generating function for the future part of $(\Lambda_{\st,\rd},\sigma)$. Its local inverses are
\[
 F_{\wt\ch,s}\sceq \left(F_\st\vert_{D_{\st,s}\big(\wt\ch\big)} \right)^{-1}\colon\left\{
\begin{array}{ccl}
F_\st\big(D_{\st,s}\big(\wt\ch\big)\big) & \to & D_{\st,s}\big(\wt\ch\big)
\\
x & \mapsto & g.x
\end{array}
\right.
\]
for $s=(\wt\ch',g)\in \Sigma_{\st,\rd}(\wt\ch)$, $\wt\ch\in\wt\fch_{\choices,\shmap}$. If we set
\[
 E_{\wt\ch,s} \sceq F_\st\big(D_{\st,s}\big(\wt\ch\big)\big),
\]
then the transfer operator with parameter $\beta$ associated to $F_\st$ is the map 
\[
 \mc L_\beta\colon \Fct(\DS_\st,\C) \to \Fct(\DS_\st,\C)
\]
given by
\[
 \mc L_\beta\varphi = \sum_{\wt\ch\in\wt\fch_{\choices,\shmap}}\sum_{(\wt\ch',g)\in\Sigma_{\st,\rd}(\wt\ch)} 1_{E_{\wt\ch,(\wt\ch',g)}} \tau_\beta(g^{-1})\varphi,
\]
where
\[
 \tau_\beta\colon\Gamma\times\Fct(\DS_\st,\C)  \to  \Fct(\DS_\st,\C),\quad
(g,\varphi)  \mapsto  \tau_\beta(g)\varphi
\]
with
\[
 \tau_\beta(g^{-1})\varphi(x) = |g'(x)|^\beta \varphi(g.x).
\]
For extensive examples of arising transfer operators we refer to \cite{Hilgert_Pohl, Pohl_mcf_Gamma0p, Pohl_mcf_general}.

\bibliography{ap_bib}
\bibliographystyle{amsalpha}
\end{document}